\documentclass[10pt]{article}

\usepackage[utf8]{inputenc}
\usepackage{amsmath,amsfonts,amssymb}
\usepackage{amsthm}
\usepackage[a4paper]{geometry}
\usepackage{color}
\usepackage{tikz}
\usetikzlibrary{patterns}
\usepackage{bm}
\usepackage{transparent}
\usepackage[dvipsnames]{xcolor}
\usepackage{graphicx}
\usepackage{subcaption}
\usepackage{wrapfig}
\usepackage{tabularx,float}
\usepackage[final, linkcolor=blue]{hyperref}
\hypersetup{
	colorlinks   = true, 
	urlcolor     = blue, 
	linkcolor    = blue, 
	citecolor   = red 
}

\usepackage{comment} 
\usepackage{tabularx} 
\usepackage{enumerate,mathrsfs}
\usepackage{mathtools}
\usepackage{enumitem}
\usepackage{orcidlink}
\usepackage{titling}
\usepackage{cleveref}
\usepackage[misc]{ifsym}
\usepackage{xargs}    
\usepackage[colorinlistoftodos,prependcaption,textsize=small]{todonotes}
\usepackage[titletoc]{appendix}
\usepackage[intoc]{nomencl}
\usepackage{soul}
\usepackage[font = small]{caption}
\usepackage{bbm}
\usepackage{overpic}
\usepackage[subrefformat=parens,labelformat=parens]{subcaption}
\usepackage{calc} 

\usepackage{mathrsfs} 
\usepackage{cite} 

\newcommand\numberthis{\addtocounter{equation}{1}\tag{\theequation}}

\reversemarginpar

\newtheorem{theorem}{Theorem}[section]
\newtheorem{proposition}[theorem]{Proposition}
\newtheorem{corollary}[theorem]{Corollary}
\newtheorem{lemma}[theorem]{Lemma}

\newtheorem{assumpA}{Assumption}
\newtheorem{assumpB}{Assumption}
\newtheorem{assumpC}{Assumption}

\definecolor{tumblue}{RGB}{0,101,189}
\definecolor{tumblue_dark}{RGB}{0,82,147}
\definecolor{tumblue_darker}{RGB}{7, 33, 64}
\definecolor{tumblue_light}{RGB}{100,160,200}
\definecolor{tumblue_lighter}{RGB}{152,198,234}
\definecolor{tumgray}{RGB}{153,153,153}
\definecolor{tumgrayd}{RGB}{90,100,110}
\definecolor{tumorange}{RGB}{227,114,34}

\theoremstyle{definition}
\newtheorem{definition}[theorem]{Definition}
\newtheorem{remark}[theorem]{Remark}

\numberwithin{equation}{section}

\newcommand{\eps}{\varepsilon}

\newcommand\norm[1]{\left\lVert#1\right\rVert}
\newcommand\normm[1]{\big\lVert#1\big\lVert}
\newcommand\normk[1]{\lVert#1\lVert}
\newcommand\abs[1]{\left\lvert#1\right\rvert}
\newcommand{\R}{\mathbb{R}}

\newcommand*{\dd}{\mathrm{d}}

\newcommand*{\nout}{n_\textup{out}}

\newcommand*{\nhid}{n_\textup{hid}}
\newcommand*{\nin}{n_\textup{in}}
\newcommand*{\RN}{\textup{RN}_{\eps,\delta}^k(\mathcal{X},\R^{\nout})}
\newcommand*{\RNs}{\textup{RN}_{\eps,\delta,\sigma}^k(\mathcal{X},\R^{\nout})}
\newcommand*{\MLP}{\textup{MLP}_{\delta}^k(\mathcal{X},\R^{\nout})}
\newcommand*{\FNN}{\textup{FNN}_{\delta}^k(\mathcal{X},\R^{\nout})}
\newcommand*{\NODE}{\textup{NODE}^k(\mathcal{X},\R^{\nout})}

\newcommand{\Wt}{\widetilde{W}}
\newcommand{\bt}{\tilde{b}}

\newcommand\blfootnote[1]{%
	\begingroup
	\renewcommand\thefootnote{}\footnote{#1}%
	\addtocounter{footnote}{-1}%
	\endgroup
}

\usepackage{fancyhdr}
\title{\textbf{Universal Approximation Constraints of \\ Narrow ResNets: The Tunnel Effect}}

\author{Christian Kuehn \orcidlink{0000-0002-7063-6173}$^{1,2,3}$,
 Sara-Viola Kuntz \orcidlink{0009-0000-4611-9742}$^{1,2,3}$
\& Tobias Wöhrer \orcidlink{0000-0001-6993-7385}$^4$
}

\date{
	\small{$^1$\textit{Technical University of Munich, School of Computation, Information and Technology, \\ Department of Mathematics, Boltzmannstraße 3, 85748 Garching, Germany} \\
	$^2$\textit{Munich Data Science Institute (MDSI), Garching, Germany } \\
	$^3$\textit{Munich Center for Machine Learning (MCML), München, Germany }\\
    $^4$\textit{TU Wien, Department of Mathematics, Institute of Analysis and Scientific Computing, Vienna, Austria}}\\[7mm]
	\large{March 30, 2026}
}

\begin{document}

	\maketitle
	
\begin{abstract}
We analyze the universal approximation constraints of narrow Residual Neural Networks (ResNets) both theoretically and numerically. For deep neural networks without input space augmentation, a central constraint is the inability to represent critical points of the input-output map. We prove that this has global consequences for target function approximations and show that the manifestation of this defect is typically a shift of the critical point to infinity, which we call the  ``tunnel effect'' in the context of classification tasks. While ResNets offer greater expressivity than standard multilayer perceptrons (MLPs), their capability strongly depends on the signal ratio between the skip and residual channels. We establish quantitative approximation bounds for both the residual-dominant (close to MLP) and skip-dominant (close to neural ODE) regimes. These estimates depend explicitly on the channel ratio and uniform network weight bounds. Low-dimensional examples further provide a detailed analysis of the different ResNet regimes and how architecture-target incompatibility influences the approximation error.
	\end{abstract}
	
	\vspace{2mm}
	
	
	\vspace{1mm}
	
	\vspace{3mm}
    
	\blfootnote{\textcolor{white}{.}\\[-2.5mm]
	\hspace{-5.4mm}\Letter \; ckuehn@ma.tum.de (Christian Kuehn)  \\[0.8mm]
	\Letter\; saraviola.kuntz@ma.tum.de (Sara-Viola Kuntz) \\[0.8mm]
	\Letter\; tobias.woehrer@tuwien.ac.at (Tobias Wöhrer) 
	}

\newpage
	{\hypersetup{linkcolor=black}
    \tableofcontents}
	
\newpage 
    \section{Introduction}\label{sec:introduction}

	Universal approximation theory  of neural networks has initially focused on shallow networks of arbitrary width \cite{Cybenko1989, Hornik1989}. However, the success of deep neural networks has shifted interest toward architectures characterized by bounded width and significant depth \cite{LuPuWang2017expressive}. As a result, research in recent years focused on determining the expressive power of these practical implementations where layer capacity is limited.
    
	 \begin{wrapfigure}{r}{0.3\textwidth}
    		\includegraphics[width = 0.3\textwidth, trim={1em 2em 8em 2em},clip]{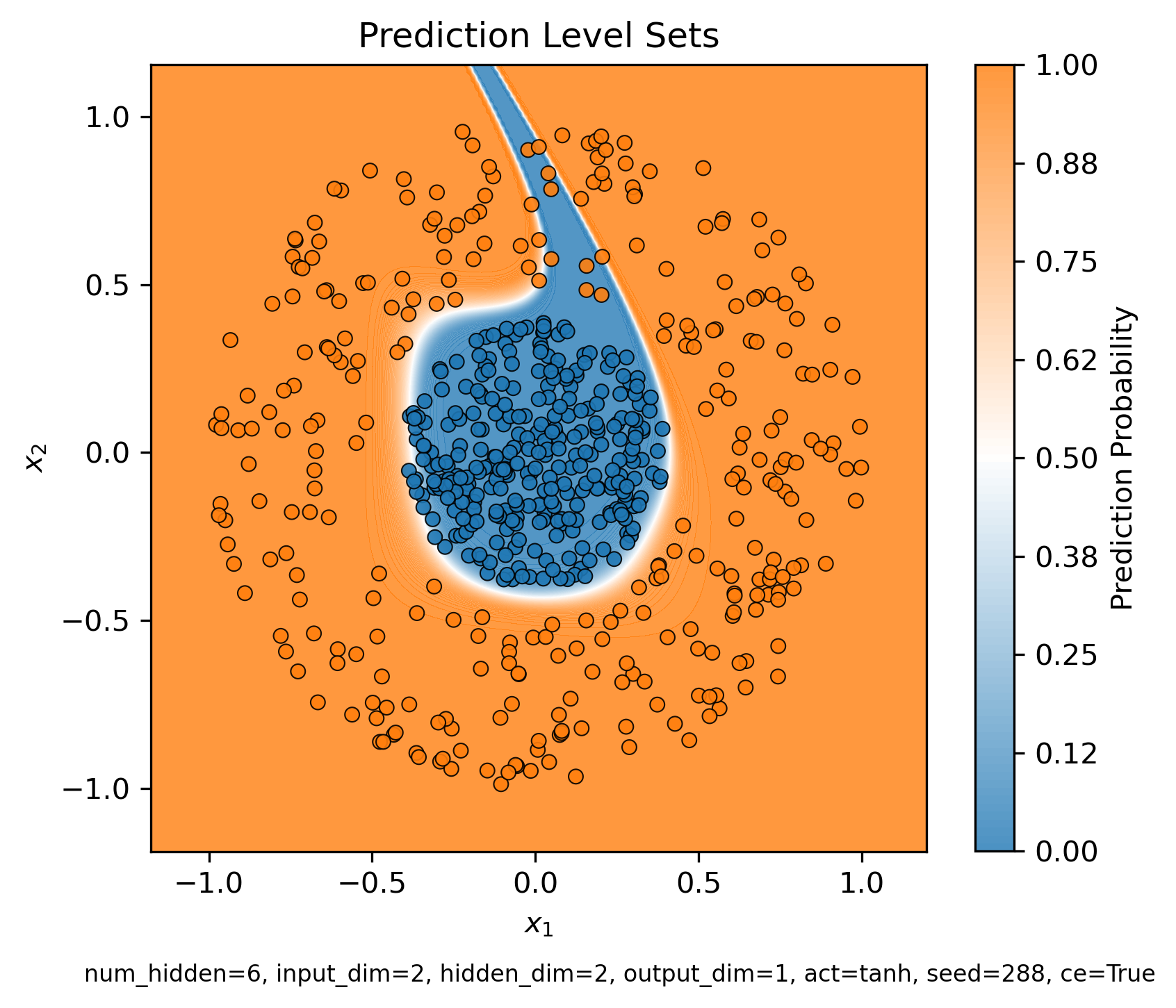}
    		\caption{\small The ``tunnel effect'' in non-augmented MLPs. \label{fig:intro}}	
\end{wrapfigure}

	It is established that both feedforward neural networks and neural ODEs are unable to approximate certain functions well for narrow architectures where the natural input space is not augmented \cite{Kuehn2023, Dupont2019}. Mathematically, a central issue is the inability of these models' input-output maps to express critical (or stationary) points. This topological constraint does not just have local consequences but limits global approximation and generalization capabilities. The classification of the two-dimensional circle toy dataset illustrates this limitation: a trained MLP with two neurons per layer and an arbitrary number of layers generates a ``tunnel'' in the prediction level sets which leads to unavoidable misclassifications, see Figure~\ref{fig:intro}. Since the model structure cannot express a critical point, it must shift the critical point to infinity which creates the barrier. Remarkably, the same observation holds for the structurally very different neural ODEs, where the input-output map is realized through continuous trajectories \cite{Dupont2019, Chen2018}.

	The structure of Residual Neural Networks (ResNets) \cite{he16} can be understood as ``in-between'' classical feedforward neural networks (such as MLPs) and neural ODEs, and is able to resolve these topological constraints \cite{lj18resnet}. The idea of ResNets is to split each layer into an identity channel that simply advances the layer input unchanged and a residual channel that is a nonlinear transformation containing the trainable parameters. The layer update is formulated as $h_{l} = h_{l-1} + f(h_{l-1}, \theta_l)$, such that the network learns the difference between the layer input and the layer output rather than the full transformation. This approach enables the training of very deep architectures as it provides stable gradient computations \cite{haber2017, Weinan2017}. Beyond these remarkable optimization benefits, it is striking that the architecture also improves the expressivity of the input-output map and is able to embed critical points without additional input augmentation. However, we observe that this ability depends rather strongly on the choice of parameter initialization. Let us again consider the circle toy dataset. Training a ResNet of width two, moderate depth, and standard weight initialization for sigmoidal activations, the ``tunnel effect'' of Figure~\ref{fig:intro} typically does not appear. By contrast, when the parameters are rescaled closer to $0$ at initialization, the ``tunnel'' appears after training.

    This work proves that the ``tunnel effect'' appears whenever ResNets are unable to embed critical points but try to approximate target functions that contain such points. It further provides explicit estimates on the embedding restrictions of ResNets in the model parameter regimes close to MLPs and neural ODEs. This implies that ResNets that roughly maintain a balance between the identity channel and the transformation channel have advantageous expressivity in practical settings. We formulate these results by parameterizing the ratio between the identity and residual channel of each ResNet layer. 

    \subsubsection*{ResNet Embedding Capabilities Depending on Channel Ratio}
    Let us outline the embedding results in more detail by assuming ResNets of constant input width and network parameters $\theta$ (weights and biases) that are restricted to a fixed parameter regime,  $\theta\in \Theta$, where $\Theta$ fulfills explicit (and reasonable in implementations) upper and lower bounds. In this context, we reformulate the ResNet layer iteration with a \emph{skip parameter}~$\eps> 0$ and a \emph{residual parameter}~$\delta > 0$ as
    \begin{equation}\label{eq:iterationintro}
        h_{l} = \eps h_{l-1} + \delta f_l(h_{l-1}, \theta_l),\quad l\in\{1, \ldots, L\},
    \end{equation}
    where the canonical choice of the residual function is given by
    \begin{equation}\label{eq:introcanonical}
        f_l(\theta_l, h)=\Wt_l \sigma (W_l h + b_l) + \bt_l,
    \end{equation}
    with weight matrices $W_l, \Wt_l$, bias vectors $\bt_l, b_l$ and a sigmoidal activation function $\sigma$. 
    Our goal is to show how the ResNets' ability to express critical points depends on the ratio $\alpha := \frac{\delta}{\eps}>0$. 

Let us start with pointing out that we can always normalize \eqref{eq:iterationintro} to $\alpha = 1$ by choosing $\Wt_l = \hat{W}_l \eps / \delta$. But then, we naturally also have to rescale the considered model parameter regime $\Theta$ when comparing embedding capabilities. Here we take the perspective of fixing $\Theta$ and analyzing the expressivity for different regimes of $\alpha> 0$.  
\begin{itemize}
    \item \emph{$0< \alpha \ll 1$}: In the case that $\alpha$ is sufficiently small, in relation to the Lipschitz constant $K_f$ of the residual transformation $f_l$, we show that ResNets are unable to express critical points. This is true independently of the number of ResNet layers $L$. We also show that this specifically holds for ResNets that are Euler discretizations of continuous neural ODEs (where $\eps = 1$), as long as the step size $\alpha = 1/L$ is sufficiently small, which implies that the number of layers $L$ is sufficiently large. This links the discrete ResNet expressivity to that of continuous neural ODEs and we provide explicit estimates depending on $\alpha$ and the $\Theta$-bounds, see Figure~\ref{fig:alpha1}.
\item \emph{$\alpha \gg 1$}:
In the case that the ratio $\alpha$ is sufficiently large in relation to the lower Lipschitz constant $k_f$, we show that the ability of ResNets to express critical points also breaks down. When setting $\eps = 0$ in \eqref{eq:iterationintro}, the skip connection vanishes and we recover a standard MLP. For such MLPs without input space augmentation, it is known that critical points cannot be embedded, see Figure~\ref{fig:alpha1}. We extend these results to ResNets with $\alpha$ sufficiently large and give explicit bounds dependent on the upper and lower bounds of the parameter space $\Theta$.
\end{itemize}
For canonical residual functions \eqref{eq:introcanonical} with sigmoidal activation functions the constant $K_f$ is determined by a uniform upper bound on the weights and biases, while $k_f$ additionally relies on a uniform lower bound on the singular values of the matrix product $\Wt_l W_l$.
We reiterate that the expressivity results remain relevant for ``standard'' ResNets, where $\eps = \delta = 1$. In this case, the restrictions are expressed through the bounds of the parameter regime $\Theta$, necessary to obtain suitable Lipschitz constants $K_f$ and $k_f$.
    
Our theoretical results do not focus on the trainability of ResNet through gradient-based algorithms, but rather on the fundamental limitations on expressivity depending on the parameter regimes. Techniques such as batch normalization are training-oriented modifications that employ related ideas. Their objective, however, is to ensure the reachability of optimal parameters assuming a well-posed setting. ResNets with rescaled channels were originally also investigated from such a trainability perspective \cite{he16}. To link the implications of our theoretical results to the parameter training of ResNet models, we provide numerical examples for the different parameter regimes.

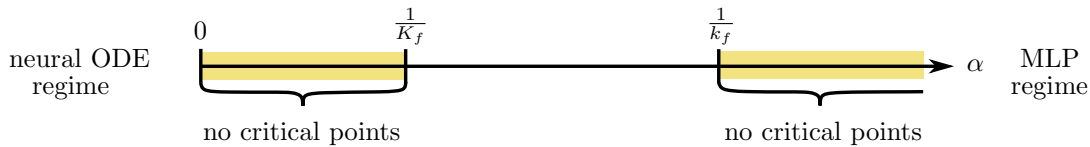
\begin{figure}
	\centering
	\begin{overpic}[scale = 0.45,tics=10]{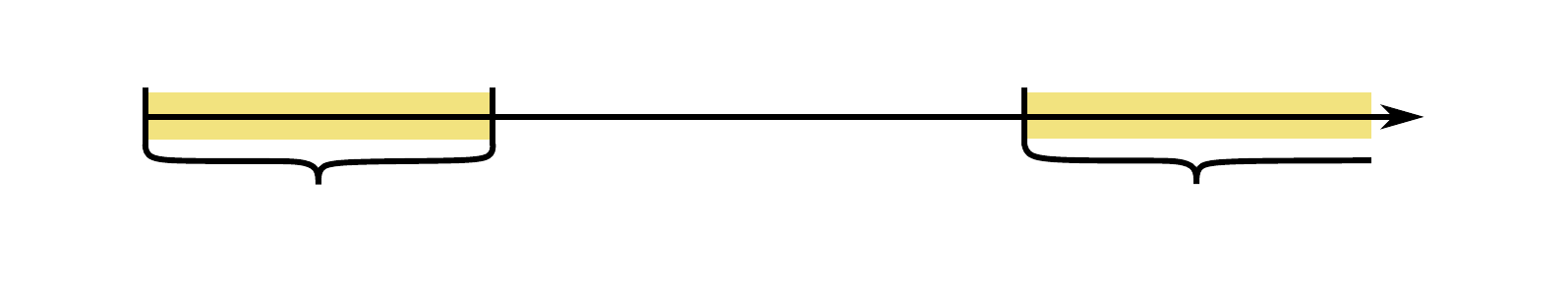}
        \put(8.5,14){$0$}
        \put(30,15){$\frac{1}{K_f}$}
        \put(64.,15){$\frac{1}{k_f}$}
		\put(92.2,10.5){$\alpha$}
        \put(98,11){MLP}
        \put(97,8){regime}
        \put(-11.5,11){neural ODE}
        \put(-9,8){regime}
		\put(9.5,3){no critical points}
        \put(66,3){no critical points}
	\end{overpic}
	\caption{Expressivity of the non-augmented ResNets depending on the ratio $\alpha \coloneqq \frac{\delta}{\eps}$ and the upper and lower Lipschitz constants of the residual function.}
	\label{fig:alpha1}
\end{figure}

    \subsubsection*{Related Work}

    The work \cite{lj18resnet} proves universal approximation for ResNets with one-neuron hidden layer and ReLU activation. According to our definitions below, this corresponds to a non-augmented ResNet architecture of input width and hence is in agreement with our results. While the approximation scheme is constructive, the model depth and model parameters do not represent typically trained ResNets. In comparison, we assume bounded and smooth activation functions, and prove weight dependent estimates.

    \cite{MWSB24} proves that if the weights of a ResNet are initialized from a (Lipschitz) continuous training parameter function, then the parameters maintain this property throughout a gradient descent optimization, provided the weights maintain a uniform upper bound through rescaling (weight clipping). Our results characterize the expressivity of different parameter regimes of ResNets which improves interpretability. While we observe the parameter regimes remain throughout training, the corresponding implicit regularization proofs are an open problem.

   In \cite{HardtMa17identity} the expressivity of ReLU ResNets is investigated. The paper considers ResNets simplified to a linear structure and shows that if the spectral norm of each $A_l$ is small, $x\mapsto Ax$ cannot express non-degenerate critical points. We prove results of similar kind for general nonlinear ResNets.
   
   The works \cite{zhang20approx, LuPuWang2017expressive} investigate the expressivity of ResNets but always assume augmented input spaces.

\subsubsection*{Structure of the Paper}

In Section~\ref{sec:resnets}, we introduce the considered ResNet architectures, derive their input gradients, and examine the global topological restrictions caused by the absence of critical points. In Section~\ref{sec:node}, we link ResNets with small channel ratios to continuous neural ODEs, establishing explicit approximation errors and proving they share the same inability to embed critical points. In Section~\ref{sec:mlp}, we analyze large channel ratios, showing that ResNets behave like perturbed feed-forward neural networks and similarly fail to express critical points. Finally, in Section~\ref{sec:numerics}, we present numerical experiments on low-dimensional datasets to illustrate the theoretical constraints and visualize the resulting ``tunnel effect''.

\section{Critical Points in Residual Neural Networks}
\label{sec:resnets}

We begin in Section~\ref{sec:resnet_universalapprox} to recall the universal approximation and universal embedding properties of continuous input–output mappings relevant to our setting. We show that, for such mappings, the inability to represent a critical point in a single output coordinate implies that the universal approximation property cannot hold.

Section~\ref{sec:resnet_model} introduces the Residual Neural Network (ResNet) architectures considered in this work, whose layers depend on a skip parameter~$\eps$ and a residual parameter~$\delta$. Depending on the input and hidden dimensions, we distinguish between non-augmented and augmented architectures.

Section~\ref{sec:assumptions} states the specific assumptions on ResNets that are necessary for our analysis in the remaining part of this work.

In Section~\ref{sec:resnet_gradient}, we derive the ResNet input gradient. In the non-augmented case, we identify conditions under which the gradient never vanishes. In particular, we prove that if the ratio $\alpha \coloneqq \frac{\delta}{\eps}$ is sufficiently small or sufficiently large, the network admits no critical points.

Finally, Section~\ref{sec:resnet_tunnel} examines the topological consequences of the absence of critical points. For small~$\alpha$, the networks can be interpreted as discretizations of neural ODEs, whereas for large~$\alpha$, they act as perturbations of classical feed-forward neural networks (FNNs). In both regimes, we derive bounds on the distance to the respective limiting models and discuss the resulting topological restrictions.

\subsection{Universal Approximation and Critical Points}
\label{sec:resnet_universalapprox}

Our goal in this work is to study the expressivity of ResNets depending on the specific architectural choices. For that purpose, we specialize the general definitions of universal approximation and universal embedding from~\cite{kk2025} to the Euclidean setting, focusing on the space $C^k(\mathcal{X},\R^{\nout})$, $\mathcal{X}\subset \R^{\nin}$ open, of $k$-times continuously differentiable functions $f: \mathcal{X} \rightarrow \R^{\nout}$. We use the max-norm $\norm{y}_\infty \coloneqq \max_{i \in \{1,\ldots,n\}} \abs{y_i}$ for vectors $y\in \R^n$ and denote the induced sup-norm of $f \in C^k(\mathcal{X},\R^{\nout})$ on a set $\mathcal{D}\subset \mathcal{X}$ by
$$\norm{f}_{\infty,\mathcal{D}} \coloneqq \sup_{x \in \mathcal{D}} \norm{f(x)}_\infty = \sup_{x \in \mathcal{D}} \max_{i \in \{1,\ldots,\nout\}} \abs{f_i(x)}.$$

\begin{definition}[Universal Approximation \cite{kk2025}]
	A neural network family $\mathcal{N} = \{\Phi_\theta: \mathcal{X} \rightarrow \R^{\nout}\}_{\theta \in \Theta}$ with $\mathcal{X}\subset \R^{\nin}$ open and parameters $\theta\in \Theta \subset \R^p$ has the universal approximation property with respect to the space $\left( C^k(\mathcal{X},\R^{\nout}), \norm{\cdot}_{\infty,\mathcal{X}}\right)$, $k \geq 0$, if for every $\varepsilon >0$, every compact subset $\mathcal{K}\subset \mathcal{X}$ and every function $\Psi \in C^k(\mathcal{X},\R^{\nout})$, there exists~$\theta\in \Theta$, such that $\norm{\Phi_\theta-\Psi}_{\infty,\mathcal{K}} < \varepsilon$.
\end{definition}

As we study in the upcoming sections the input-output map  of different neural network architectures directly, it is useful to additionally introduce the stronger concept of universal embedding, which requires an exact representation of the target function.

\begin{definition}[Universal Embedding \cite{kk2025}]
    A neural network family $\mathcal{N} = \{\Phi_\theta: \mathcal{X} \rightarrow \R^{\nout}\}_{\theta \in \Theta}$ with $\mathcal{X}\subset \R^{\nin}$ open and parameters $\theta\in \Theta \subset \R^p$ has the universal embedding property with respect to the space $C^k(\mathcal{X},\R^{\nout})$, $k \geq 0$, if for every function $\Psi \in C^k(\mathcal{X},\R^{\nout})$, there exists~$\theta\in \Theta$, such that $\Phi_\theta(x) = \Psi(x)$ for all $x \in \mathcal{X}$.
\end{definition}

An important property characterizing the dynamics of the input-output map of neural networks $\Phi \in C^1(\mathcal{X},\R^{\nout})$ with input space $\mathcal{X}\subset \R^{\nin}$ open and output space $\R^{\nout}$, is the existence of critical points, i.e., zeros of the network input gradient $\nabla_x \Phi$. It is crucial not to confuse the input gradient $\nabla_x \Phi$ with the parameter gradient $\nabla_\theta \Phi$, where $\theta\in\Theta$ denotes all parameters of the considered neural network. The parameter gradient $\nabla_\theta \Phi$ is needed for backpropagation algorithms in the training process of neural networks.

It is instructive to study the existence of critical points in the input-output map of different neural network architectures, as this has direct implications for their embedding and approximation capabilities \cite{kk2025}. The following theorem implies that neural network architectures in which at least one component map has no critical point cannot have the universal approximation property.

\begin{theorem}[Maps without Critical Points are not Universal Approximators]
	\label{th:uniapprox}
	Consider a set of functions $\mathcal{S}\subset C^1(\mathcal{X},\R^{\nout})$, $\mathcal{X}\subset \R^{\nin}$ open, wherein for every map $\Phi \in \mathcal{S}$, there exists a component $i\in \{1,\ldots,\nout\}$, such that $\nabla_x\Phi_i(x) \neq 0$ for all $x\in \mathcal{X}$. 
	Then the set $\mathcal{S}$ cannot have the universal approximation property with respect to the space $\left( C^k(\mathcal{X},\R^{\nout}), \norm{\cdot}_{\infty,\mathcal{X}}\right)$ for every $k \geq 0$.
\end{theorem}

\begin{proof}
	The statement follows directly by generalizing \cite[Theorem 2.6]{kk2025} from scalar maps to multiple output components: by the respective theorem, there exists a compact set $\mathcal{K}\subset \mathcal{X}$ with non-empty interior $\text{int}(\mathcal{K})$, such that the sup-norm between the quadratic function 
	\begin{equation} \label{eq:Psi_y}
		\Psi_z: \mathcal{K}\rightarrow \R, \quad \Psi_z(x) = \sum_{j = 1}^{\nin}(x_j-z_j)^2 \quad \text{with} \quad z\in \text{int}(\mathcal{K})
	\end{equation}
	and any scalar map without any critical point cannot be made arbitrarily small. Consequently, the quadratic function $\Psi_z$ also cannot be approximated with arbitrary precision by any of the considered component maps $\Phi_i \in C^1(\mathcal{X},\R)$ satisfying $\nabla_x\Phi_i(x) \neq 0$ for all $x\in \mathcal{X}$, i.e., $\norm{\Phi_i-\Psi_z}_{\infty,\mathcal{K}} \geq \mu $ for some $\mu>0$. Hence, every map $\Psi \in C^k(\mathcal{X},\R^{\nout})$ with $i$-th component map $\Psi_i \vert_{\mathcal{K}} = \Psi_z$ as defined in~\eqref{eq:Psi_y} cannot be approximated with arbitrary precision by any map $\Phi \in \mathcal{S}$ on $\mathcal{K}$, where the $i$-th component fulfills the given assumptions.
\end{proof}

As Theorem~\ref{th:uniapprox} considers maps in which single output components cannot have any critical points, we can, without loss of generality, restrict our upcoming analysis to scalar neural network architectures with $\nout = 1$.
In the recent work~\cite{kk2025}, the expressivity of multilayer perceptrons (MLPs) and neural ODEs is studied by characterizing the existence and regularity of critical points. One of the main results shows that non-augmented MLPs and non-augmented neural ODEs cannot have critical points and hence lack the universal approximation property (see also Section~\ref{sec:node_criticalpoints} and Section~\ref{sec:mlp_criticalpoints}). In the augmented case, the universal approximation property of MLPs and neural ODEs is well established in the literature \cite{Dupont2019,Hornik1989,Hornik1991,Kidger2022,Pinkus1999,Zhang2020a}.
In the following sections, we first introduce ResNet architectures and then study their expressivity using Theorem~\ref{th:uniapprox}.

\subsection{Residual Neural Networks}
\label{sec:resnet_model}

Residual neural networks (ResNets) are feedforward neural networks structured in layers $h_l \in \R^{\nhid}$, $l \in \{1,\ldots,L\}$, consisting of $\nhid$ nodes each. Given an initial layer $h_0 \in \R^{\nhid}$, the hidden layers are iteratively updated by
\begin{equation} \label{eq:resnet_updaterule} 
	h_{l} = \eps h_{l-1} + \delta f_l(h_{l-1},\theta_l), \qquad  l \in \{1, \ldots, L\}, 
\end{equation}
with a (typically nonlinear) \emph{residual function} $f_l:\R^{\nhid} \times \R^{p_{l}} \rightarrow \R^{\nhid}$ and \emph{hidden parameters} $\theta_{l} \in \Theta_l \subset \R^{p_{l}}$, where $ \Theta_l$ denotes the set of parameters of layer $l$ and $\Theta = \Theta_1 \times \ldots \times \Theta_L \subset \R^p$ the total parameter space of the ResNet. The layer update rule~\eqref{eq:resnet_updaterule} includes two terms: a skip connection scaled by the \emph{skip parameter} $\eps> 0$, and a residual term weighted by the \emph{residual parameter} $\delta> 0$.

In contrast to classical feed-forward neural networks, corresponding to the case $\eps = 0$ and $\delta>0$ in~\eqref{eq:resnet_updaterule} (cf.~Section~\ref{sec:mlp}), the ResNet update rule contains the linear term $\eps h_{l-1}$. It is called a skip- or shortcut connection (cf.~\cite{he16}), as it allows the layer input to bypass the transformation of the residual function, see Figure~\ref{fig:skipconnection}\subref{fig:skipconnection_a} for a visualization. The residual function $f_l$ can be an arbitrary map, but typical choices include
\begin{equation} \label{eq:resnet_typical_f} 
	f_l(h_{l-1},\theta_{l}) \coloneqq \widetilde{W}_l\sigma_l(W_lh_{l-1}+b_l)+\tilde{b}_l = \widetilde{W}_l\sigma_l(a_l)+\tilde{b}_l
\end{equation}
for $l \in \{1,\ldots,L\}$ with parameters $\theta_{l} \coloneqq (W_l, \widetilde{W}_l,b_l,\tilde{b}_l)$ consisting of weight matrices $W_l \in \R^{m_l \times \nhid}$, $\widetilde{W}_l \in \R^{\nhid \times m_l}$, and biases $b_l \in \R^{m_l}$, $\tilde b_l \in \R^{\nhid}$. The map $\sigma_l: \R \rightarrow \R$ is called an \emph{activation function}, applied component-wise to the \emph{pre-activated states} $a_l \coloneqq W_l h_{l-1} + b_l \in \R^{m_l}$, as visualized in Figure~\ref{fig:skipconnection}\subref{fig:skipconnection_b}. By a slight abuse of notation, we also write $\sigma_l$ for the component-wise extension $\sigma_l:\R^{m_l} \rightarrow \R^{m_l}$ as in \eqref{eq:resnet_typical_f}. The intended meaning will be clear from the argument. Typical choices for the activation function $\sigma_l$ include $\tanh$, $\operatorname{sigmoid}(y) = (1+e^{-y})^{-1}$ or $\operatorname{ReLU}(y) = \max\{0,y\}$. The hidden dimension $\nhid$ is constant across the layers of the ResNet, whereas the dimension $m_l$ of the pre-activated state $a_l$, which also controls the number of parameters used in layer~$h_l$, can vary.

To be flexible in the input and output dimensions of ResNets, two additional transformations are applied before the layer $h_0$ and after the layer $h_L$, resulting in the input-output map
\begin{equation}\label{eq:resnet}
	\Phi: \mathcal{X} \rightarrow \R^{\nout}, \;\mathcal{X}\subset \R^{\nin}, \qquad \Phi(x) = \tilde{\lambda}(h_L(\lambda(x))),
\end{equation}
with \emph{input transformation} $\lambda: \R^{\nin} \rightarrow \R^{\nhid}$, \emph{output transformation} $\tilde{\lambda}: \R^{\nhid}\rightarrow \R^{\nout}$ and composite map $h_L:\R^{\nhid} \rightarrow \R^{\nhid}$, which maps the \emph{transformed input} $h_0 = \lambda(x)$ to the last hidden layer $h_L$ via the iterative update rule \eqref{eq:resnet_updaterule}. For the ResNet architecture \eqref{eq:resnet}, we call $x \in \mathcal{X}$ the \emph{input}, $h_0 = \lambda(x)$ the \emph{transformed input}, $h_1, \ldots, h_L$ the \emph{hidden layers}, and $\Phi(x)$ the \emph{output} of the neural network. Often, the transformations $\lambda$ and $\tilde{\lambda}$ are chosen to be affine linear, but nonlinear functions are also possible. For example, for classification tasks with $\nout = 1$, the output is often normalized to a probability, i.e., $\Phi(x)\in[0,1]$.  

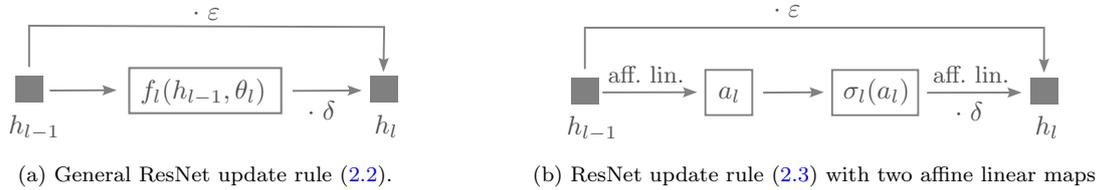
\begin{figure}[h]
	\centering
	\begin{subfigure}{0.38\textwidth}
		\centering
		\begin{overpic}[scale = 0.5,tics=10]{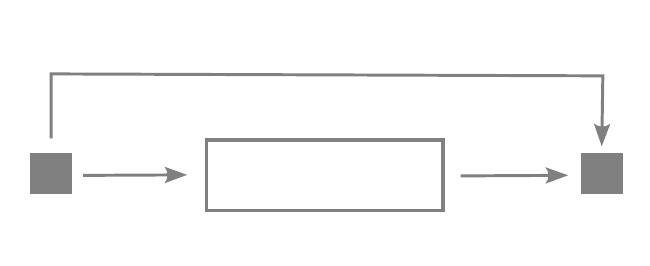}
			\put(3,3.5){\textcolor{black!70}{$h_{l-1}$}}
			\put(34.5,12){\textcolor{black!70}{$f_l(h_{l-1},\theta_{l})$}}
			\put(90.5,3.5){\textcolor{black!70}{$h_l$}}
			\put(74.5,7){\textcolor{black!70}{$\cdot \; \delta$}}
			\put(47,31){\textcolor{black!70}{$\cdot \; \eps$}}
		\end{overpic}
		\caption{General ResNet update rule~\eqref{eq:resnet_updaterule}.}
		\label{fig:skipconnection_a}
	\end{subfigure}
	\hspace{5mm}
	\begin{subfigure}{0.57\textwidth}
		\centering
		\begin{overpic}[scale = 0.5,tics=10]{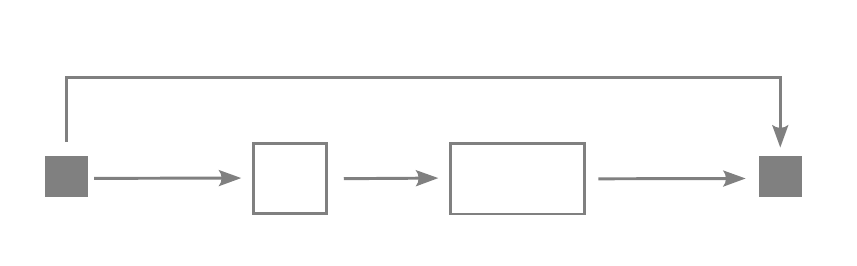}
			\put(4.5,2.5){\textcolor{black!70}{$h_{l-1}$}}
			\put(32,9){\textcolor{black!70}{$a_l$}}
			\put(54.7,9){\textcolor{black!70}{$\sigma_l(a_l)$}}
			\put(89.5,2.5){\textcolor{black!70}{$h_l$}}
			\put(75,5){\textcolor{black!70}{$\cdot \; \delta$}}
			\put(42,24){\textcolor{black!70}{$\cdot \; \eps$}}
			\put(71,12){\small\textcolor{black!70}{aff.\ lin.}}
			\put(12,12){\small\textcolor{black!70}{aff.\ lin.}}
		\end{overpic}
		\caption{ResNet update rule~\eqref{eq:resnet_typical_f} with two affine linear maps.}
		\label{fig:skipconnection_b}
	\end{subfigure}
	\caption{Structure of the update rule of a residual neural network with skip parameter $\eps$ and residual parameter $\delta$. Every layer $h_l \in \R^{\nhid}$ is represented by a square.}
	\label{fig:skipconnection}
\end{figure}

The case $\eps = \delta = 1$ in the update rule~\eqref{eq:resnet_updaterule} corresponds to classical ResNet architectures as introduced in~\cite{he16}. The general case with $\eps, \delta > 0$ can also be seen as a special case of highway networks with carry gate~$\eps$ and transform gate $\delta$~\cite{Srivastava2015}. In the case $\eps = 1$ and $\delta \rightarrow 0$, ResNets are connected to neural ODEs, which we introduce in Section~\ref{sec:node}. The case $\delta>0$ and $\eps  = 0$ with residual function $f_l$ as defined in~\eqref{eq:resnet_typical_f} leads to classical feed-forward neural networks, such as multilayer perceptrons, introduced in Section~\ref{sec:mlp}. In our upcoming analysis, we focus on ResNets with $\eps,\delta>0$, which include both a linear skip connection and a non-linear residual term. In the following, we define the class of considered ResNet architectures.

\begin{definition}[Residual Neural Network] \label{def:resnet}
	For $k \geq 0$, the set $\RN\subset C^k(\mathcal{X},\R^{\nout})$ with $\mathcal{X}\subset \R^{\nin}$ open, $\eps,\delta>0$, denotes all ResNet architectures $\Phi: \mathcal{X} \rightarrow \R^{\nout}$ as defined in~\eqref{eq:resnet} with 
	\begin{itemize}
		\item input transformation $\lambda \in C^k(\R^{\nin},\R^{\nhid})$,
		\item output transformation $\tilde\lambda \in C^k(\R^{\nhid},\R^{\nout})$,
		\item for $l \in \{1,\ldots,L\}$: residual functions $f_l(\cdot, \theta_l)\in C^{k}(\R^{\nhid}, \R^{\nhid})$ for each fixed $\theta_l \in \Theta_l \subset \R^{p_l}$.
	\end{itemize}
	In the case of $f_l$ of the form~\eqref{eq:resnet_typical_f}, the condition $f_l(\cdot, \theta_l)\in C^{k}(\R^{\nhid}, \R^{\nhid})$ is equivalent to $\sigma_l\in C^{k}(\R,\R)$ for $l \in \{1,\ldots,L\}$. We call the corresponding ResNets \emph{canonical} and denote the subset of \emph{canonical ResNet architectures} by $\RNs \subset \RN \subset C^k(\mathcal{X},\R^{\nout})$.
\end{definition}

The regularity of ResNets $\Phi\in \RN$, $k\geq 0$, $\mathcal{X}\subset \R^{\nin}$ open, follows directly from the regularity of the residual functions $f_l$ or the activation functions~$\sigma_l$, respectively.

\begin{remark}
    Throughout this work, we assume strictly positive parameters $\eps$ and $\delta$ to simplify the notation. This choice is based on typical ResNets (where $\eps = \delta = 1$) and their interpretation as neural ODE discretizations (where $\eps = 1$, $\delta>0$). As most of the upcoming results only depend on the magnitude of the ratio $\alpha \coloneqq \frac{\delta}{\eps}$, they also extend to the case $\alpha <0$, by considering the absolute value~$\abs{\alpha}$ in the respective bounds.
\end{remark}

Depending on the input dimension $\nin$ and the hidden dimension $\nhid$, we distinguish between non-augmented and augmented ResNet architectures.

\begin{definition}[ResNet Classification]\label{def:resnet_classes}
	The class of ResNets $\RN$, $k \geq 0$, $\mathcal{X}\subset \R^{\nin}$ open, is subdivided as follows:
	\begin{itemize}
		\item \emph{Non-augmented ResNet} $\Phi \in \textup{RN}_{\eps,\delta,\textup{N}}^k(\mathcal{X},\R^{\nout})$: it holds $\nin \geq \nhid$. 
		\item \emph{Augmented ResNet} $\Phi \in \textup{RN}_{\eps,\delta,\textup{A}}^k(\mathcal{X},\R^{\nout})$: it holds $\nin < \nhid$. 
	\end{itemize}
	For canonical ResNets $\Phi \in \textup{RN}_{\eps,\delta,\sigma}^k(\mathcal{X},\R^{\nout})$, we analogously denote non-augmented architectures by $\textup{RN}_{\eps,\delta,\sigma,\textup{N}}^k(\mathcal{X},\R^{\nout})$ and augmented architectures by $\textup{RN}_{\eps,\delta,\sigma,\textup{A}}^k(\mathcal{X},\R^{\nout})$.
\end{definition}

The concept of non-augmented and augmented ResNets is visualized in Figure~\ref{fig:resnet_architectures}. The classification of ResNets is independent of the choice of the residual functions $f_l$, and hence also independent of the intermediate dimensions $m_l$ in the case of canonical ResNets. The distinction of architectures becomes relevant for the analysis of the expressivity of the ResNet input-output map in the upcoming Section~\ref{sec:resnet_gradient} and Section~\ref{sec:resnet_tunnel}. Our work mainly focuses on the restrictions of non-augmented ResNet architectures induced by Theorem~\ref{th:uniapprox}. Before we calculate the ResNet input gradient in Section~\ref{sec:resnet_gradient} to characterize the existence of critical points, we discuss in the following  Section~\ref{sec:assumptions} the assumptions on non-augmented architectures relevant for our analysis.

\begin{figure}[h]
	\centering
	\begin{subfigure}{0.41\textwidth}
		\centering
		\begin{overpic}[scale = 0.25,tics=10]{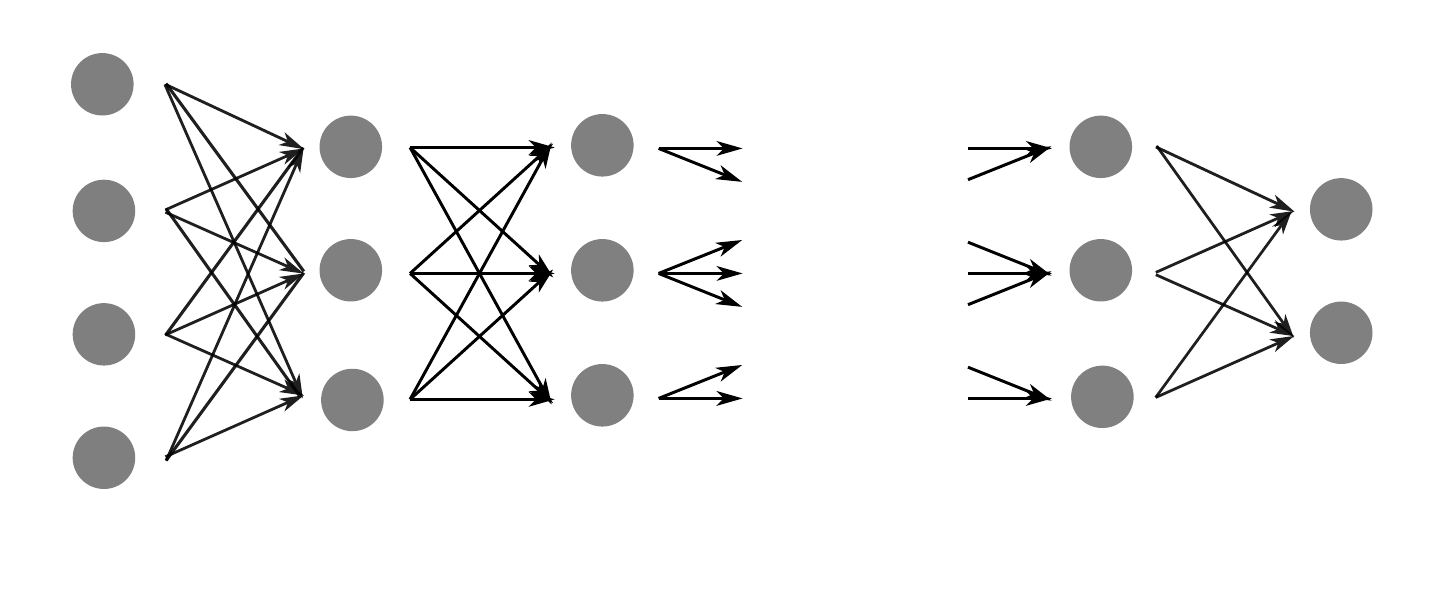}
			\put(5,1){$\textcolor{black!70}{x}$}
			\put(22.5,1){$\textcolor{black!70}{h_0}$}
			\put(40,1){$\textcolor{black!70}{h_1}$}
			\put(75,1){$\textcolor{black!70}{h_L}$}
			\put(89,1){$\textcolor{black!70}{\Phi(x)}$}
		\end{overpic}
		\caption{Example of a non-augmented ResNet $\Phi \in \textup{RN}_{\eps,\delta,\textup{N}}^k(\R^4,\R^{2})$ with $\nhid = 3$.}
		\label{fig:resnet_arch_nonaugmented}
	\end{subfigure}
    \hspace{12mm}
	\begin{subfigure}{0.41\textwidth}
		\centering
		\begin{overpic}[scale = 0.25,tics=10]{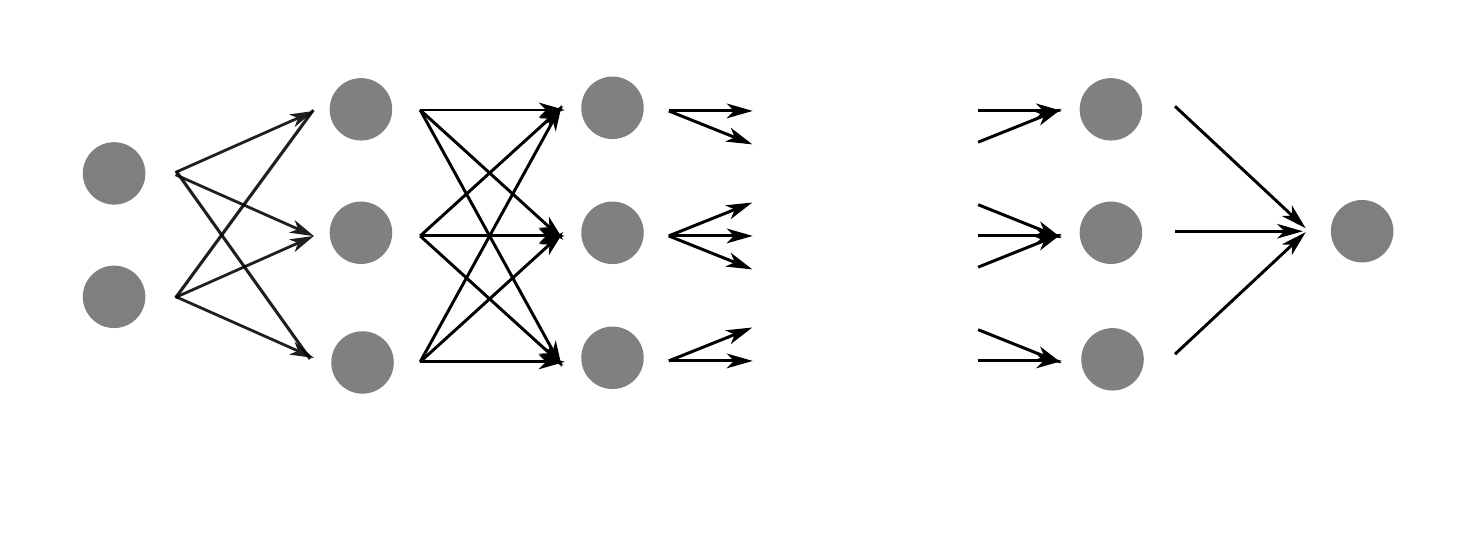}
			\put(6,1){$\textcolor{black!70}{x}$}
			\put(22,1){$\textcolor{black!70}{h_0}$}
			\put(40,1){$\textcolor{black!70}{h_1}$}
			\put(73,1){$\textcolor{black!70}{h_L}$}
			\put(88,1){$\textcolor{black!70}{\Phi(x)}$}
		\end{overpic}
		\caption{Example of an augmented ResNet $\Phi \in \textup{RN}_{\eps,\delta,\textup{A}}^k(\R^2,\R)$ with $\nhid = 3$.}
		\label{fig:resnet_arch_augmented}
	\end{subfigure}
	\caption{Classification of ResNet architectures depending on the input and the hidden dimension. Every node of the neural network is represented as a circle.}
	\label{fig:resnet_architectures}
\end{figure}
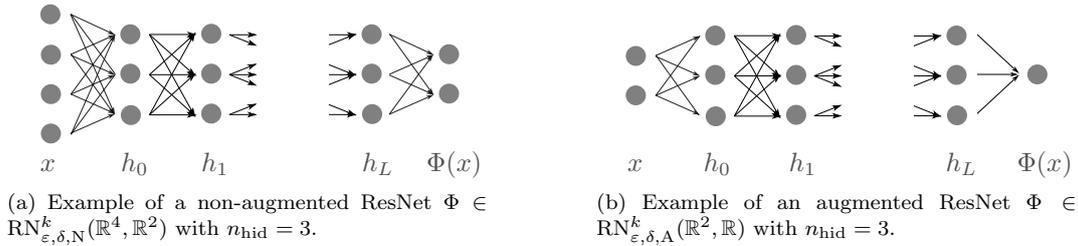

\subsection{Assumptions on Non-Augmented Architectures}
\label{sec:assumptions}

In this section, we state the assumptions used for the analysis of non-augmented ResNet architectures in Section~\ref{sec:resnet_gradient} and Section~\ref{sec:resnet_tunnel}. These conditions also appear in the study of ResNets related to neural ODEs (Section~\ref{sec:node}) and to feed-forward neural networks (Section~\ref{sec:mlp}). The assumptions are introduced below for reference; each of the upcoming theorems explicitly states which of the assumptions are required.

\subsubsection*{Assumptions on the Activation Functions}

Canonical ResNets $\Phi \in \textup{RN}^k_{\eps,\delta,\sigma}(\mathcal{X},\R^{\nout})$, $k \geq 0$, $\mathcal{X}\subset \R^{\nin}$ open, with residual functions as defined in~\eqref{eq:resnet_typical_f}, depend on the component-wise applied activation functions $\sigma_l \in C^k(\R,\R)$, $l \in \{1,\ldots,L\}$. The same dependence applies to multilayer perceptrons, which are introduced and analyzed in Section~\ref{sec:mlp}. In the following, we state three assumptions, all of which are satisfied by standard sigmoidal nonlinearities such as $\tanh$ and $\operatorname{sigmoid}(y) = (1+e^{-y})^{-1}$, but exclude activation functions of ReLU-type. In the literature, ReLU activation functions are often excluded if smoothness properties such as the continuity of the derivative is needed (cf.~\cite{MWSB24}). 

\begin{assumpA}[Lipschitz Continuous Activation]
	\label{ass:A_activation_lipschitz}
	The activation functions $\sigma_l \in C^0(\R,\R)$ are uniformly globally Lipschitz continuous with Lipschitz constant $K_\sigma>0$. In particular, if $\sigma_l \in C^1(\R,\R)$, Lipschitz continuity is equivalent to $\norm{\sigma_l'}_{\infty,\R} \leq K_\sigma$ for all $l \in \{1,\ldots,L\}$.
\end{assumpA}

\begin{assumpA}[Bounded Activation]
	\label{ass:A_activation_bounded}
	The activation functions $\sigma_l \in C^0(\R,\R)$ are uniformly bounded, i.e., there exists a constant $S > 0$ such that $\norm{\sigma_l}_{\infty,\R} \leq S$ for all $l \in \{1,\ldots,L\}$.
\end{assumpA}

\begin{assumpA}[Strictly Monotone Activation]
	\label{ass:A_activation_monotone}
	The activation functions $\sigma_l \in C^0(\R,\R)$ are strictly monotone for all $l \in \{1,\ldots,L\}$. In particular, if $\sigma_l \in C^1(\R,\R)$, strict monotonicity is equivalent to $ \abs{\sigma_l'(y)} > 0$ for every $y \in \R$ and all $l \in \{1,\ldots,L\}$. 
\end{assumpA}

\begin{remark}
	The three assumptions on the activation functions are standard in the analytical study of deep neural networks. The global Lipschitz continuity of Assumption~\ref{ass:A_activation_lipschitz} ensures stability of the input-output mapping. By the mean value theorem, the upper bound of $\norm{\sigma_l'}_{\infty,\R}$ defines a Lipschitz constant of the activation function. The boundedness of the activation function in Assumption~\ref{ass:A_activation_bounded} is relevant when estimating the distance between ResNets and neural ODEs or ResNets and MLPs. The monotonicity Assumption~\ref{ass:A_activation_monotone} guarantees non-degeneracy of the layer-wise Jacobians and is essential in our analysis of critical points.
\end{remark}

\subsubsection*{Assumptions on the Parameters}

In the following, we formulate two assumptions on the parameters appearing in the considered neural networks. For a parameter tuple $\theta = (A_1, \dots, A_k)$ consisting of matrices or vectors $A_j \in \R^{a_j \times b_j}$, $j\in \{1,\ldots,k\}$ we consider the standard Euclidean norm $\norm{\cdot}_2$, the max-norm $\norm{\cdot}_\infty$ and their induced matrix norms defined by
\begin{equation*}
    \norm{A_j}_q \coloneqq \sup_{x \in \R^{b_j} \setminus \{0\}} \frac{\norm{A_jx}_q}{\norm{x}_q}, \qquad q \in \{2,\infty\}.
\end{equation*}
To ensure that the upcoming uniform parameter bound applies to each component of the tuple $\theta$, we define the norm of $\theta$ as the maximum of the individual norms, i.e., $\norm{\theta}_q \coloneqq \max_{j \in \{1,\dots,k\}} \norm{A_j}_q$. Since different architectures involve different parameters, their notation and dimensions are specified in the respective theorems in which they are used. 

\begin{assumpB}[Bounded Parameters]
	\label{ass:B_weights_bounded}
	For a set of parameters $\Theta \subset \R^p$, all parameters are uniformly bounded in the max-norm or the Euclidean norm, i.e., there exist constants $\omega_\infty\geq 0$ or $\omega_2\geq 0$, such that it holds $\norm{\theta}_\infty \leq \omega_\infty$ or $\norm{\theta}_2 \leq \omega_2$ respectively for all $\theta \in \Theta$.
\end{assumpB}

\begin{remark}~
\begin{itemize}
\item	Since all norms on finite-dimensional spaces are equivalent, uniform boundedness in any norm is sufficient for Assumption~\ref{ass:B_weights_bounded}.
\item  For a fixed neural network with a finite number of parameters $\theta$, the uniform upper bound is trivially satisfied.
	
\item The assumption on bounded parameters, especially weight matrices, is standard in the analysis and implementation of deep neural networks. It helps prevent exploding gradients during training and ensures reasonable global Lipschitz constants of the considered networks~\cite{miyato_spectral_2018,salimans_weight_2016,schoenlieb2025}. In practice, many learning algorithms enforce these bounds implicitly or explicitly through regularization techniques, such as weight decay \cite{Goodfellowbook2016} or spectral regularization \cite{yoshida2017spectral}. 
\end{itemize}
\end{remark}

\begin{assumpB}[Full Rank Weight Matrices]
	\label{ass:B_weights_fullrank}
	For a set of parameters $\Theta \subset \R^p$, all weight matrices $W \in \R^{a \times b}$ contained in $\Theta$ have full rank, i.e., $\textup{rank}(W) = \min\{a,b\}$.
\end{assumpB}

\begin{remark}
	The full-rank assumption is generic, as by~\cite[Lemma~3.6]{kk2025}, the subset of singular weight matrices has Lebesgue measure zero. This assumption is especially relevant to prove the non-existence of critical points in certain non-augmented neural network architectures.
\end{remark}

\subsubsection*{Assumptions on the Input and Output Transformations}

We make the following assumption on the input and output transformations $\lambda$ and $\tilde{\lambda}$ of the considered non-augmented network architectures with $\nin \geq \nhid$. 

\begin{assumpC}[Non-Singular Input and Output Transformation] 
	\label{ass:C_input_output}
	The input and output transformations $\lambda \in C^1(\R^{\nin},\R^{\nhid})$ and $\tilde\lambda \in C^1(\R^{\nhid},\R^{\nout})$ with $\nin\geq\nhid$ fulfill the following:
	\begin{itemize}
		\item The Jacobian matrix $\partial_x \lambda(x)\in \R^{\nhid \times \nin}$ has full rank $\nhid$ for every $x \in \R^{\nin}$. 
		\item The Jacobian matrix $\partial_{y}{\tilde{\lambda}}(y)\in \R^{\nout \times \nhid}$ has full rank $\min\{\nhid,\nout\}$ for every $y \in \R^{\nhid}$.
	\end{itemize}
\end{assumpC}

\begin{remark} \label{rem:lambda_typical}
	Assumption~\ref{ass:C_input_output} is satisfied if the input and output transformations $\lambda$ and $\tilde{\lambda}$ have the typical form of the residual function as in~\eqref{eq:resnet_typical_f}, such that
	\begin{equation} \label{eq:lambda_typical}
		\lambda(x) = \widetilde{W}_0 \sigma_0(W_0 x + b_0) + \tilde{b}_0, \qquad  \tilde{\lambda}(y) = \widetilde{W}_{L+1} \sigma_{L+1}(W_{L+1} y + b_{L+1}) + \tilde{b}_{L+1},
	\end{equation}
	with a component-wise applied activation function $\sigma$ that satisfies Assumption~\ref{ass:A_activation_monotone} and weight matrices that fulfill Assumption~\ref{ass:B_weights_fullrank}, cf.~\cite[Lemma C.1]{kk2025}. 
\end{remark}

\subsection{Existence of Critical Points}
\label{sec:resnet_gradient}

To understand the existence of critical points in ResNet architectures, we derive in the following the input gradient of ResNets $\Phi \in \textup{RN}_{\eps,\delta}^1(\mathcal{X},\R)$, $\mathcal{X}\subset \R^{\nin}$ open. Due to Theorem~\ref{th:uniapprox}, we can without loss of generality restrict our analysis to single output components, or equivalently, scalar neural networks. The upcoming proposition applies to both non-augmented and augmented ResNets.  

\begin{proposition}[ResNet Input Gradient] \label{prop:resnet_gradient}
	The input gradient of a scalar ResNet $\Phi \in \textup{RN}^1_{\eps,\delta}(\mathcal{X},\R)$ at $x \in \mathcal{X}$, $\mathcal{X}\subset \R^{\nin}$ open, with iterative update rule~\eqref{eq:resnet_updaterule} is given by
	\begin{equation}
    \begin{aligned}
		\nabla_x \Phi(x) & = \Bigl[ \partial_{h_L}\tilde{\lambda}(h_L) \cdot (\eps \cdot \textup{Id}_{\nhid} + \delta \cdot \partial_{h_{L-1}} f_L(h_{L-1},\theta_L)) \cdots \Bigr. \\
		&\hspace{17.5mm} \cdots \Bigl. (\eps \cdot \textup{Id}_{\nhid} + \delta \cdot \partial_{h_{0}} f_1(h_{0},\theta_1))\cdot \partial_x\lambda(x)\Bigr]^\top \quad \in \R^{\nin} \label{eq:resnet_gradient}
	\end{aligned}
    \end{equation}
	with Jacobian matrix $\partial_{h_{l-1}}f_l(h_{l-1},\theta_l)$ of the residual function $f_l$ with respect to the layer $h_{l-1}$. 
	
	For canonical ResNets $\Phi \in \textup{RN}^1_{\eps,\delta,\sigma}(\mathcal{X},\R)$ with the explicit residual function in \eqref{eq:resnet_typical_f}, the Jacobian matrix $\partial_{h_{l-1}}f_l(h_{l-1},\theta_l)$ of the ResNet input gradient is given by
	\begin{equation*}
		\partial_{h_{l-1}}f_l(h_{l-1},\theta_l) =  \widetilde{W}_{l}\sigma'_{l}(a_{l})W_{l} \in \R^{\nhid \times \nhid}, \qquad l\in \{1,\ldots,L\},
	\end{equation*}
	with diagonal matrix $\sigma'_l(a_{l}) \coloneqq \textup{diag}(\sigma_l'([a_l]_1),\ldots, \sigma_l'([a_l]_{m_l}))$, where $a_l \coloneqq W_l h_{l-1} + b_l \in \R^{m_l}$.
\end{proposition}

\begin{proof}
	By the multi-dimensional chain rule applied to ResNets as defined in~\eqref{eq:resnet}, it holds that
	\begin{align*}
		\partial_{x} \Phi(x) & =\frac{\partial \Phi}{\partial h_L}\cdot\frac{\partial h_L}{\partial h_{L-1}} \cdots \frac{\partial h_1}{\partial h_0} \cdot \frac{\partial h_0}{\partial x}  \\
		& = \partial_{h_L}\tilde{\lambda}(h_L) \cdot (\eps \cdot \textup{Id}_{\nhid} + \delta \cdot \partial_{h_{L-1}} f_L(h_{L-1},\theta_L)) \cdots \\
		& \hspace{16.5mm} \cdots (\eps \cdot \textup{Id}_{\nhid} + \delta \cdot \partial_{h_{0}} f_1(h_{0},\theta_1))\cdot \partial_x\lambda(x) \quad \in \R^{1 \times \nin}.
	\end{align*}
	The first result follows by taking the transpose, as $\nabla_x \Phi(x) = [\partial_{x} \Phi(x)]^\top$. 
	
	For canonical ResNets $\Phi \in \textup{RN}^k_{\eps,\delta,\sigma}(\mathcal{X},\R)$ with the explicit residual function in~\eqref{eq:resnet_typical_f}, the Jacobian matrix $\partial_{h_{l-1}}f_l(h_{l-1},\theta_l)$ with respect to the layer $h_{l-1}$ is given by
	\begin{equation*}
		\partial_{h_{l-1}}f_l(h_{l-1},\theta_l) =   \partial_{h_{l-1}} \left( \widetilde{W}_l \sigma_l(W_l h_{l-1} + b_l) + \tilde{b}_l\right) =   \widetilde{W}_{l}\sigma'_{l}(a_{l})W_{l} \quad \in \R^{\nhid \times \nhid},
	\end{equation*}
	for $l\in \{1,\ldots,L\}$, with diagonal matrix $\sigma'_l(a_{l}) \coloneqq \textup{diag}(\sigma_l'([a_l]_1),\ldots, \sigma_l'([a_l]_{m_l}))$, as the activation function is applied component-wise to the pre-activated state $a_l \coloneqq W_l h_{l-1} + b_l \in \R^{m_l}$.
\end{proof}

To apply Theorem~\ref{th:uniapprox} to residual architectures, we formulate the criterion for non-vanishing gradients via the rank of the layer-wise Jacobians. The following lemma shows that the rank \emph{solely depends on the ratio} $\alpha \coloneqq \frac{\delta}{\eps}$ of the residual parameter $\delta$ and the skip parameter $\eps$ and not on their individual size.

\begin{lemma}[Layer-Wise Jacobians] \label{lem:resnet_Dl_fullrank}
	Given $l \in \{1,\ldots,L\}$, parameters $\eps, \delta >0$, and a residual function $f_l(\cdot, \theta_l)\in C^{1}(\R^{\nhid},\R^{\nhid})$ with $\theta_l \in \Theta_l \subset \R^{p_l}$, the layer-wise Jacobian at $h_{l-1} \in \R^{\nhid}$
	\begin{equation*}
		D_l \coloneqq \eps \cdot \textup{Id}_{\nhid} + \delta \cdot \partial_{h_{l-1}}f_l(h_{l-1},\theta_l) \in \R^{\nhid \times \nhid}
	\end{equation*}
	has full rank if and only if $ -\frac{1}{\alpha} = -\frac{\eps}{\delta}$ is not an eigenvalue of the Jacobian matrix $\partial_{h_{l-1}}f_l(h_{l-1},\theta_l)$.
\end{lemma}

\begin{proof}
    The statement follows directly from the definition of the eigenvalues of  $\partial_{h_{l-1}}f_l(h_{l-1},\theta_l)$ after rescaling the matrix $D_l$ by $\frac{1}{\delta}$.
\end{proof}

In the case of non-augmented ResNets $\Phi \in \textup{RN}^1_{\eps,\delta,\textup{N}}(\mathcal{X},\R)$, $\mathcal{X}\subset \R^{\nin}$ open, Lemma~\ref{lem:resnet_Dl_fullrank} allows us to formulate  conditions on the transformations $\lambda$, $\tilde{\lambda}$ and the parameters $\eps$ and $\delta$, under which the gradient $\nabla_x\Phi(x)$ never vanishes for any $x \in \mathcal{X}$. 

\begin{proposition}[Non-Augmented ResNets without Critical Points]\label{prop:resnet_nonaugmented}
	Let $\Phi \in \textup{RN}^1_{\eps,\delta,\textup{N}}(\mathcal{X},\R)$, $\mathcal{X}\subset \R^{\nin}$ open, $\eps,\delta>0$, be a scalar non-augmented ResNet, which fulfills:
	\begin{itemize}
		\item $ -\frac{1}{\alpha} = -\frac{\eps}{\delta}$ is not an eigenvalue of the Jacobian matrix $\partial_{h_{l-1}}f_l(h_{l-1},\theta_l)$ for all $l \in \{1,\ldots,L\}$, $h_{l-1} \in \R^{\nhid}$ and parameters $\theta_l \in \Theta_l \subset \R^{p_l}$.
		\item The input and output transformations $\lambda$ and $\tilde{\lambda}$ fulfill Assumption~\ref{ass:C_input_output}.
	\end{itemize}
	Then $\Phi$ cannot have any critical points, i.e., $\nabla_x \Phi(x) \neq 0$ for all $x \in \mathcal{X}$.
\end{proposition}

\begin{proof}
	By Lemma~\ref{lem:resnet_Dl_fullrank}, all layer-wise Jacobians $D_l \coloneqq \eps \cdot \textup{Id}_{\nhid} + \delta \cdot \partial_{h_{l-1}}f_l(h_{l-1},\theta_l)$ have full rank for all $h_{l-1} \in \R^{\nhid}$ and $\theta_l \in \Theta_l$.	
	As the ResNet $\Phi$ is scalar and non-augmented, it holds $\nin \geq \nhid \geq \nout = 1$, such that the dimensions in the matrix product~\eqref{eq:resnet_gradient} are monotonically decreasing. Together with Assumption~\ref{ass:C_input_output}, it follows that the gradient $\nabla_x \Phi(x)$ always has full rank~$1$, uniformly in $x \in \mathcal{X}$. 
\end{proof}

Depending on the ratio $\alpha \coloneqq \frac{\delta}{\eps}$ between the skip parameter $\eps$ and the residual parameter $\delta$, we analyze in Section~\ref{sec:node_criticalpoints} and Section~\ref{sec:mlp_criticalpoints} when the first assumption of Proposition~\ref{prop:resnet_nonaugmented} is fulfilled. We identify two parameter regimes, $0<\alpha \ll 1$ and $\alpha \gg 1$, for which non-augmented ResNets cannot have any critical points, as visualized in Figure~\ref{fig:alpha1}. In these parameter regimes, ResNets lack the universal approximation property as implied by Theorem~\ref{th:uniapprox}. Consequently, we extend the results of~\cite{kk2025}, regarding the non-existence of critical points in non-augmented neural ODEs and non-augmented MLPs, to ResNets. 

In intermediate parameter regimes of~$\alpha$, non-augmented ResNets can have the universal approximation property, as demonstrated in~\cite{Lin2018} for the standard case $\eps = \delta = 1$. In contrast, for augmented ResNet architectures, no equivalent statement to Proposition~\ref{prop:resnet_nonaugmented} about the non-existence of critical points exists. Instead, it follows analogously to~\cite[Theorem~3.18]{kk2025} that due to dimension augmentation, critical points can exist in augmented architectures even when all other assumptions of Proposition~\ref{prop:resnet_nonaugmented} are fulfilled. 

The following table provides an overview of the main results of this work for non-augmented ResNets in the parameter regimes $0<\alpha \ll 1$ and $\alpha \gg 1$. 

\begin{figure}[h]
    \vspace{4mm}
	\begin{tabular*}{\textwidth}{p{2cm}|p{6.1cm}|p{6.1cm}}
		& \hspace{11mm}\textbf{ResNets with $0<\alpha \ll 1$} & \hspace{12mm} \textbf{ResNets with $\alpha \gg 1$} \\[5pt]
		\hline && \\[-5pt]
		\textbf{Regime}
		&
		Close to neural ODEs, as introduced in Section~\ref{sec:node_model}
        \vspace{5pt} 
		&
		Close to FNNs, such as MLPs, introduced in Section~\ref{sec:mlp_model}
		\vspace{5pt} \\
		\hline &&  \\[-5pt]
		\textbf{Distance}
		&
		  Theorem~\ref{th:eulerdiscretization} and Corollary~\ref{cor:node_resnet_error}: small approximation error between ResNets and Neural ODEs
		  \vspace{5pt}
		& 
		Theorem~\ref{th:mlp_resnet_error} and Corollary~\ref{cor:mlp_resnet_error}: small approximation error between ResNets and FNNs / MLPs
		\vspace{5pt} \\
        \hline &&  \\[-5pt]
        \textbf{Critical Points}
		& 
		Theorem~\ref{th:node_criticalpoints}: no critical points for non-augmented neural ODEs 
		&
		Theorem~\ref{th:mlp_criticalpoints}: no critical points for non-augmented FNNs / MLPs \\[5mm]
        & Theorem~\ref{th:resnet_criticalpoints_alphasmall}:  no critical points for non-augmented ResNets if  $0 < \alpha < \frac{1}{K_f}$ with Lipschitz constant $K_f$ of the residual function  & 
        Theorem~\ref{th:resnet_criticalpoints_alphalarge}:  no critical points for non-augmented ResNets if  $\alpha > \frac{1}{k_f}$ with lower Lipschitz constant $k_f$ of the residual function\\
	\end{tabular*}
	\captionof{table}{Summary of the main results regarding the relationship between ResNets, neural ODEs, and FNNs, and the existence of critical points in the parameter regimes $0<\alpha \ll 1$ and $\alpha \gg 1$.} \label{tab:results}
\end{figure}

While the previous results establish the absence of critical points in specific regimes, we now turn to a global topological perspective. In the following, we quantify the expressivity of non-augmented ResNets by measuring their distance to function classes that are already known to lack universal approximation.

\subsection{Global Topological Restrictions}
\label{sec:resnet_tunnel}

In Section~\ref{sec:resnet_gradient}, we discussed that non-augmented ResNets cannot have critical points in the parameter regimes $0<\alpha \ll 1$ and $\alpha \gg 1$. Complementing that direct analysis, we characterize the expressivity of ResNets by comparing them to reference networks $\overline{\Phi}$ without critical points, such as non-augmented neural ODEs (cf.~Theorem~\ref{th:node_criticalpoints}) and non-augmented FNNs (cf.~Theorem~\ref{th:mlp_criticalpoints}).
In Section~\ref{sec:node_relationship}, we show that the distance between ResNets with $\eps = 1$ and neural ODEs scales linearly in the residual parameter~$\delta$. Similarly, in Section~\ref{sec:mlp_relationship}, we establish that for fixed $\delta>0$ and small $\eps>0$, the distance between ResNets and FNNs scales asymptotically linearly in the skip parameter~$\eps$.

We divide the upcoming analysis into two steps. First, we establish the topological restrictions of a continuously differentiable map $\overline{\Phi}$ lacking critical points. Second, we introduce a uniform distance bound to show that these topological restrictions extend to any continuous neural network $\Phi$ that closely approximates $\overline{\Phi}$.

\subsubsection{Topological Restrictions of Maps Without Critical Points}
\label{sec:topological_restrictions_phibar}

We first discuss the topological implications for a reference network $\overline{\Phi} \in C^1(\mathcal{X},\R)$, $\mathcal{X}\subset \R^{\nin}$ open, under the assumption that $\nabla_x\overline{\Phi}(x) \neq 0$ for all $x\in \mathcal{X}$. To study how this property restricts the expressivity of the network, we introduce level sets, sub-level sets, and super-level sets.

\begin{definition}[Level Sets]
    Given $f \in C^0(\mathcal{X},\R)$, $\mathcal{X}\subset \R^{\nin}$, we define the following sets for $c \in \R$:
    \begin{itemize}
        \item Level set $S_c(f) \coloneqq \{ x \in \mathcal{X} \mid f(x) = c\}$,
        \item Sub-level set $S_c^{\leq}(f) \coloneqq \{ x \in \mathcal{X} \mid f(x) \leq c\}$, strict sub-level set $S_c^{<}(f) \coloneqq \{ x \in \mathcal{X} \mid f(x) < c\}$,
        \item Super-level set $S_c^{\geq}(f) \coloneqq \{ x \in \mathcal{X} \mid f(x) \geq c\}$, strict super-level set $S_c^{>}(f) \coloneqq \{ x \in \mathcal{X} \mid f(x) > c\}$.
    \end{itemize}
\end{definition}

The absence of critical points has direct implications on the compactness of the closed sub- and super-level sets. Although the following lemma relates to classical Morse theory, we prove it directly using elementary calculus.

\begin{lemma}[Non-Compactness of Sub- and Super-Level Sets] \label{lem:levelsets_noncompact}
    Let $\overline{\Phi} \in C^1(\mathcal{X},\R)$, $\varnothing \neq \mathcal{X}\subset \R^{\nin}$ open, with $\nabla_x\overline{\Phi}(x) \neq 0$ for all $x\in \mathcal{X}$. Then for every $c \in  (\inf_{x \in \mathcal{X}} \overline{\Phi}(x),\sup_{x \in \mathcal{X}} \overline{\Phi}(x))$, the sub- and super-level sets $S_c^{\leq}(\overline{\Phi})$ and $S_c^{\geq}(\overline{\Phi})$ are non-compact in $\R^{\nin}$.
\end{lemma}

\begin{proof}
    As $\overline{\Phi} \in C^1(\mathcal{X},\R)$ has no critical points, it is non-constant around each $x\in \mathcal{X}$, such that the interval $ (\inf_{x \in \mathcal{X}} \overline{\Phi}(x),\sup_{x \in \mathcal{X}} \overline{\Phi}(x))$ is non-empty. Hence, for each $c \in  (\inf_{x \in \mathcal{X}} \overline{\Phi}(x),\sup_{x \in \mathcal{X}} \overline{\Phi}(x))$, the strict sub- and super-level sets $S_c^{<}(\overline{\Phi})$ and $S_c^{>}(\overline{\Phi})$ are non-empty. 
    Assume by contradiction that the sub-level set $S_c^{\leq}(\overline{\Phi})$ is compact in $\R^{\nin}$. By the extreme value theorem (cf.~\cite{Rudin1976}), the continuous function $\overline{\Phi}$ attains its minimum on $S_c^{\leq}(\overline{\Phi})$ at some point $x_{\min}$. Since the strict sub-level set $S_c^{<}(\overline{\Phi})$ is non-empty, there exists some $y$ with $\overline{\Phi}(y) < c$, which implies the minimum must satisfy $\overline{\Phi}(x_{\min}) \leq \overline{\Phi}(y) < c$.
    As $\overline{\Phi}$ is continuous and $\mathcal{X}$ is open, the strict sub-level set $S_c^{<}(\overline{\Phi})$ is an open set in $\R^{\nin}$ and is contained in the interior $\text{int}(S_c^{\leq}(\overline{\Phi}))$. Hence, the minimum is attained in the interior $x_{\min} \in \text{int}(S_c^{\leq}(\overline{\Phi}))$. A necessary condition for the existence of a local minimum $x_\text{min}$ in the interior of a domain for a continuously differentiable function is that $\nabla_x \overline{\Phi}(x_\text{min}) = 0$ \cite{Forster2017}. This contradicts the assumption $\nabla_x\overline{\Phi}(x) \neq 0$ for all $x \in \mathcal{X}$, thus, $S_c^{\leq}(\overline{\Phi})$ cannot be compact in $\R^{\nin}$.
    
    The statement for the super-level set $S_c^{\geq}(\overline{\Phi})$ follows analogously by replacing the considered minimum with the maximum $x_{\max}$ attained in the interior $\text{int}(S_c^{\geq}(\overline{\Phi}))$. 
\end{proof}

Lemma~\ref{lem:levelsets_noncompact} has direct implications on classification tasks, where decision boundaries are defined as level sets. In the following, we state the classical binary classification problem in an abstract form (cf.~\cite{boyd2004convex}). Throughout this section, we consider a compact domain $\varnothing \neq \mathcal{K}\subset \R^{\nin}$ as our reference frame, since practical classification datasets are bounded.

\begin{definition}[Binary Classification Problem]\label{def:binary_classification}
    For a compact set $\varnothing \neq \mathcal{K} \subset \R^{\nin}$ and $c_0,c_1\in\R$ with $c_0<c_1$ consider a given dataset
    \begin{equation*}
        A_{c_0,c_1} \coloneqq \left\{ (x_i,y_i)_{i = 1}^{N_\textup{data}} \mid x_i \in \mathcal{K}, y_i \in \{c_0,c_1\}\right\}.
    \end{equation*}
    A function $\Phi\in C^0(\mathcal{K},\R)$ successfully classifies $A_{c_0,c_1}$ if there exists $c^\ast\in (c_0,c_1)$, such that
    \begin{align*}
         &A_{c_0} \coloneqq \left\{ (x_i)_{i = 1}^{N_\textup{data}} \mid (x_i,c_0) \in A_{c_0,c_1}\right\}\subset S_{c^\ast}^{<}(\Phi), \\
         &A_{c_1} \coloneqq \left\{ (x_i)_{i = 1}^{N_\textup{data}} \mid (x_i,c_1) \in A_{c_0,c_1}\right\} \subset S_{c^\ast}^{>}(\Phi).
    \end{align*}
   The level set $S_{c^\ast}(\Phi)$ is called the decision boundary of $\Phi$. 
\end{definition}

For a given dataset, the goal of a classification problem is to classify correctly as many data points as possible. As we study in this section the topological restrictions induced by neural networks without critical points independently of a specific dataset, we focus on the topology of level sets, especially the decision boundary. For a reference map $\overline{\Phi}$ without critical points, the topology of its decision boundary is strictly constrained. We separate the one-dimensional case from the higher-dimensional case $\nin \geq 2$, as the boundary of a compact set $\mathcal{K}$ is disconnected in one dimension but can be connected in higher dimensions, leading to different topological restrictions on the decision boundary. 

\paragraph{One-Dimensional Case}

In the one-dimensional case $\nin = 1$, Lemma~\ref{lem:levelsets_noncompact} is equivalent to the fact that all one-dimensional maps without critical points are strictly monotone. Naturally, it follows that strictly monotone maps cannot satisfactorily approximate non-monotone target functions. This fact can also be formulated from the perspective of binary classifications. For a strictly increasing map $\overline{\Phi} \in C^1(\mathcal{X},\R)$ and any $c^\ast \in \overline{\Phi}(\mathcal{X})$, there exists a unique $x_{c^\ast} \in \mathcal{X}$ with $\overline{\Phi}(x_{c^\ast}) = c^\ast$, and the strict level sets take the form
\begin{equation}\label{eq:levelsets}
    S_{c^\ast}^{<}(\overline{\Phi}) = (-\infty, x_{c^\ast}) \cap \mathcal{X}, \qquad S_{c^\ast}(\overline{\Phi}) = \{x_{c^\ast}\}, \qquad  S_{c^\ast}^{>}(\overline{\Phi}) = (x_{c^\ast}, \infty) \cap \mathcal{X}.
\end{equation}
This implies that for such functions $\overline{\Phi}$, any dataset whose classes $A_{c_0}$ and $A_{c_1}$ cannot be separated by two disjoint intervals, cannot be successfully classified. Therefore, $\overline{\Phi}$ can only separate data that is split into two disjoint intervals. 

\paragraph{Higher-Dimensional Case}

If the boundary $\partial \mathcal{K}$ is connected, Lemma~\ref{lem:levelsets_noncompact} has direct implications for nested datasets, such as the two-dimensional circle dataset, where one class is entirely surrounded by another. To perfectly classify such a dataset, the decision boundary $S_{c^\ast}(\overline{\Phi})$ would need to form a closed curve (or a hypersphere for $\nin > 2$), contained entirely within the interior of the compact domain $\mathcal{K}$. However, this would create a strictly bounded, and therefore compact, sub- or super-level set $S_{c^\ast}^{\leq}(\overline{\Phi})$ or $S_{c^\ast}^{\geq}(\overline{\Phi})$, which directly contradicts Lemma~\ref{lem:levelsets_noncompact}. 

As we show in the upcoming Theorem~\ref{thm:constraints_nd}, this implies that any decision boundary attempting to separate a nested dataset is mathematically forced to intersect the domain boundary $\partial \mathcal{K}$. When minimizing the empirical classification error, the natural geometric compromise is to form a ``tunnel'', a continuous, narrow region of the inner class extending to the boundary $\partial \mathcal{K}$. As illustrated in the introduction in Figure~\ref{fig:intro}, this topology effectively shifts the required critical point outside of the observation domain $\mathcal{K}$. We will visualize the higher-dimensional ``tunnel effect'' in the following section, where we demonstrate that these topological limitations transfer to any network $\Phi$ in close proximity to $\overline{\Phi}$.

\subsubsection{Extension to Networks in Close Proximity}

In the upcoming analysis, we study the topological restrictions implied by a small distance in the sup-norm between a function $\Phi \in C^0(\mathcal{X},\R)$, $\mathcal{X}\subset \R^{\nin}$ open, and a second reference map $\overline{\Phi} \in C^1(\mathcal{X},\R)$ without critical points, i.e., $\nabla_x\overline{\Phi}(x) \neq 0$ for all $x\in \mathcal{X}$. We assume that there exists $\mu>0$, such that on a compact set $\mathcal{K}\subset \mathcal{X}$ it holds that
\begin{equation}\label{eq:distance_global}
	\norm{\Phi-\overline{\Phi}}_{\infty,\mathcal{K}} \leq \mu.
\end{equation}
The reference map $\overline{\Phi}$ may be a neural ODE, an MLP, or a ResNet in the restricted parameter regimes $0<\alpha \ll 1$ or $\alpha \gg 1$ identified in Section~\ref{sec:resnet_gradient}.
The estimate~\eqref{eq:distance_global} implies that the map $\Phi$ is trapped in a $\mu$-tube around $\overline{\Phi}$. As we will show, this forces the topological limitations of $\overline{\Phi}$ discussed in Section~\ref{sec:topological_restrictions_phibar} to carry over to $\Phi$. Again, we treat the one-dimensional and higher-dimensional cases separately, as the boundary $\partial \mathcal{K}$ of a compact set is disconnected in one dimension, but can be connected in higher dimensions.

\paragraph{One-Dimensional Case}

While every one-dimensional map $\overline{\Phi}$ without critical points is strictly monotone, the $\mu$-close map $\Phi$ is allowed to oscillate and possess critical points. However, as the distance between $\Phi$ and $\overline{\Phi}$ is bounded, the level sets of $\Phi$ have topological restrictions, too.

 \begin{theorem}[Topological Restrictions for $\nin = 1$] \label{thm:constraints_1d}
    Consider $\Phi \in C^0(\mathcal{X},\R)$ with $\mathcal{X}\subset \R$ open, and $\overline{\Phi} \in C^1(\mathcal{X},\R)$ with $\nabla_x\overline{\Phi}(x) \neq 0$ for all $x\in \mathcal{X}$. Let $\varnothing \neq \mathcal{K} = [k_0,k_1] \subset \mathcal{X}$ be a compact interval and $a = \overline{\Phi}(k_0)$, $b = \overline{\Phi}(k_1)$. If
    \begin{equation*}
        \norm{\Phi - \overline{\Phi}}_{\infty,\mathcal{K}} \leq \mu < \frac{\abs{b-a}}{2},
    \end{equation*}
    then for any $c \in (\min(a,b) + \mu, \max(a,b) - \mu)$, none of the sub- and super-level sets $S_c^{<}(\Phi)$ and $S_c^{>}(\Phi)$ can be entirely contained within the interior $\textup{int}(\mathcal{K}) = (k_0, k_1)$.
\end{theorem}

\begin{proof}
    Without loss of generality, we assume that $\overline{\Phi}$ is strictly monotonically increasing, such that $a < b$. The bound for the distance between $\Phi$ and $\overline{\Phi}$ implies that
    \begin{equation*}
        \Phi(k_0) \leq \overline{\Phi}(k_0) + \mu = a + \mu, \qquad \text{and} \qquad 
        \Phi(k_1) \geq \overline{\Phi}(k_1) - \mu = b - \mu.
    \end{equation*}
    Since $\mu < \frac{\abs{b-a}}{2}$, the interval $(a+\mu, b-\mu)$ is non-empty, such that for any $c \in (a+\mu, b-\mu)$ it follows $\Phi(k_0) < c < \Phi(k_1)$. Consequently, any sub-level set $S_c^{<}(\Phi)$ must contain $k_0$ and any super-level set $S_c^{>}(\Phi)$ must contain $k_1$, such that the statement follows. By symmetry, the same argumentation holds for monotonically decreasing maps.
\end{proof}

Theorem~\ref{thm:constraints_1d} implies that $\Phi$ cannot correctly classify any dataset $A_{c_0,c_1}$ where the boundary points $k_0, k_1$ belong to the same class but surround a different inner class. We visualize these constraints in Figure~\ref{fig:approximation_1D} using a centered classification task. The data points are subdivided in two classes via the quadratic function
\begin{equation*}
    \Psi_z: \mathcal{K} = [k_0,k_1]\rightarrow \R, \quad \Psi_z(x) = (x-z)^2 \quad \text{with} \quad z\in (k_0,k_1)
\end{equation*}
in the following way: all data points in the region $S_{c^\ast}^{<}(\Psi_z)$ are assigned the label $c_0$, whereas all data points in the region $S_{c^\ast}^{>}(\Psi_z)$ are assigned the label $c_1$. The function $\Psi_z$ was also used in Theorem~\ref{th:uniapprox} to show that functions without critical points cannot have the universal approximation property. 

On the one hand, the map $\Phi$ can form local minima to better approximate the given data than the strictly monotone map $\overline{\Phi}$, as shown in Figure~\ref{fig:approximation_1D}. On the other hand, $\Phi$ is trapped in a $\mu$-tube around~$\overline{\Phi}$, such that its level sets $S_{c^\ast}^{<}(\Phi)$ and $S_{c^\ast}^{>}(\Phi)$ can be disconnected, but cannot be entirely contained in the interval $(k_0,k_1)$. Figure~\ref{fig:approximation_1D} shows that the classification of $\Phi$ improves over the classification of $\overline{\Phi}$, but Theorem~\ref{thm:constraints_1d} guarantees failure of a perfect classification. 

\begin{figure}[h]
	\centering
		\begin{overpic}[scale = 0.7,tics=10]
			{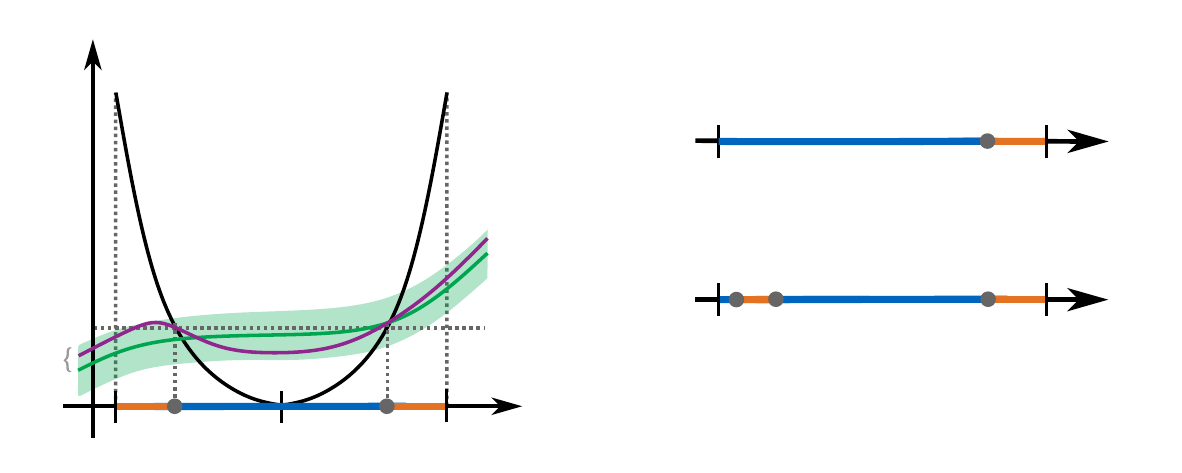}
				\put(46,4.8){$\R$}
				\put(7,38){$\R$}
                \put(95,14){$\R$}
				\put(95,27){$\R$}
				\put(23.1,2){$z$}
                \put(9,2){$k_0$}
                \put(37,2){$k_1$}
                \put(60,11){$k_0$}
                \put(88,11){$k_1$}
                \put(60,24){$k_0$}
                \put(88,24){$k_1$}
                \put(35.5,34){$\Psi_z(x)$}
				\put(3.1,9){\textcolor{gray}{$\mu$}}
                \put(4.8,12){\textcolor{gray}{$c^\ast$}}
				\put(42,17){\textcolor{Green}{$\overline{\Phi}(x)$}}
				\put(42,20){\textcolor{Plum}{$\Phi(x)$}}
                \put(67,32){\textcolor{Green}{Level Sets of $\overline{\Phi}$}}
                \put(67,18){\textcolor{Plum}{Level Sets of $\Phi$}}
                \put(56,4.8){\textcolor{tumblue}{Blue: $S_{c^\ast}^{<}$ \quad } \textcolor{gray}{Gray: $S_{c^\ast}$ \quad } \textcolor{tumorange}{Orange: $S_{c^\ast}^{>}$}}
			\end{overpic}
	\caption{Topological restrictions induced by the absence of critical points in the one-dimensional case: the neural network $\Phi$ has limited accuracy when classifying the data defined by the level sets of the quadratic function $\Psi_z$, as it is trapped in a $\mu$-tube around a strictly monotonically increasing reference function $\overline{\Phi}$.}
    \label{fig:approximation_1D}
\end{figure}

\paragraph{Higher-Dimensional Case}

In the higher-dimensional case $\nin \geq 2$, the boundary $\partial \mathcal{K}$ of a compact set $\mathcal{K}$ can be connected, which fundamentally changes the topological restrictions compared to the one-dimensional case. The following theorem proves that the decision boundary $S_{c^\ast}(\Phi)$ of a map $\Phi$ that is $\mu$-close to a reference map $\overline{\Phi}$ without critical points cannot be entirely contained within the interior of $\mathcal{K}$, but instead must intersect the boundary $\partial \mathcal{K}$.

\begin{theorem}[Topological Restrictions for $n_{\textup{in}} \geq 2$] \label{thm:constraints_nd}
    Consider $\Phi \in C^0(\mathcal{X},\R)$ with $\mathcal{X}\subset \R^{n_{\textup{in}}}$ open, $n_{\textup{in}} \geq 2$, and $\overline{\Phi} \in C^1(\mathcal{X},\R)$ with $\nabla_x\overline{\Phi}(x) \neq 0$ for all $x\in \mathcal{X}$. Let $\varnothing \neq \mathcal{K} \subset \mathcal{X}$ be a compact set with a connected boundary $\partial \mathcal{K}$, and let $[a,b] \subset \overline{\Phi}(\mathcal{K})$ be a non-empty interval. If 
    \begin{equation*}
        \norm{\Phi - \overline{\Phi}}_{\infty,\mathcal{K}}  \leq \mu < \frac{b-a}{2},
    \end{equation*}
    then for all $c \in (a+\mu, b-\mu)$, the level set $S_c(\Phi)$ intersects the boundary $\partial \mathcal{K}$, i.e., $S_c(\Phi) \cap \partial \mathcal{K} \neq \varnothing$.
\end{theorem}

\begin{proof}
    Let $0 < \nu < \frac{b-a}{2}-\mu$ be arbitrary. 
    Since $[a,b] \subset \overline{\Phi}(\mathcal{K})$, the sub-level set $\mathcal{S} \coloneqq S_{a+\nu}^{\leq}(\overline{\Phi})$ has a non-empty intersection with the domain $\mathcal{K}$. With the assumptions on $\overline{\Phi}$ and $\Phi$ and $a+\nu\in (\inf_{x \in \mathcal{X}} \overline{\Phi}(x),\sup_{x \in \mathcal{X}} \overline{\Phi}(x))$, it follows from Lemma~\ref{lem:levelsets_noncompact} that $\mathcal{S}$ is non-compact in $\R^{n_{\textup{in}}}$. Furthermore, by the continuity of $\overline{\Phi}$, the set $\mathcal{S} = \overline{\Phi}^{-1}((-\infty,a+\nu])$ is closed in $\mathcal{X}$.
    
    We first show that $\mathcal{S}$ must intersect the boundary of $\mathcal{K}$, as visualized in Figure~\ref{fig:levelsets}. Assume by contradiction that $\mathcal{S} \cap \partial \mathcal{K} = \varnothing$. Since there exists $x^\ast \in \mathcal{K}$ with $\overline{\Phi}(x^\ast) = a$, the intersection $\mathcal{S} \cap \mathcal{K}$ is non-empty. If this intersection does not intersect the boundary $\partial \mathcal{K}$, then $\mathcal{S} \cap \mathcal{K}$ is entirely contained in the interior $\textup{int}(\mathcal{K})$. Because $\mathcal{S}$ is closed and $\mathcal{K}$ is compact, their intersection $\mathcal{S} \cap \mathcal{K}$ is compact, too. The continuous function $\overline{\Phi}$ must then attain its global minimum on the compact set $\mathcal{S} \cap \mathcal{K}$ at some point $x_{\min}$. As $\mathcal{S} \cap \mathcal{K} \subset \textup{int}(\mathcal{K})$,  $x_{\min}$ would be an interior point of $\mathcal{K}$.
    A necessary condition for the existence of the minimum $x_{\min}$ of the continuously differentiable function $\overline{\Phi}: \textup{int}(\mathcal{K}) \rightarrow \R$ is that $\nabla_x\overline{\Phi}(x_{\min}) = 0$, as $\textup{int}(\mathcal{K}) \subset \R^{n_{\textup{in}}}$ is an open set \cite{Forster2017}. This contradicts the assumption that $\overline{\Phi}$ has no critical points, thus it follows that $\mathcal{S} \cap \partial \mathcal{K} \neq \varnothing$.
    By an analogous argument applied to the super-level set $S_{b-\nu}^{\geq}(\overline{\Phi})$, we conclude that:
    \begin{equation} \label{eq:proof_levelsets_1}
        S_{a+\nu}^{\leq}(\overline{\Phi}) \cap \partial \mathcal{K} \neq \varnothing \qquad \text{and} \qquad S_{b-\nu}^{\geq}(\overline{\Phi}) \cap \partial \mathcal{K} \neq \varnothing.
    \end{equation}
    The uniform estimate $\norm{\Phi - \overline{\Phi}}_{\infty,\mathcal{K}} \leq \mu$ implies the following set inclusions:
    \begin{equation}\label{eq:proof_levelsets_2}
        S_{a+\nu}^{\leq}(\overline{\Phi}) \subset S_{a+\nu+\mu}^{\leq}(\Phi) \qquad \text{and} \qquad S_{b-\nu}^{\geq}(\overline{\Phi}) \subset S_{b-\nu-\mu}^{\geq}(\Phi).
    \end{equation}
    Combining the properties \eqref{eq:proof_levelsets_1} and \eqref{eq:proof_levelsets_2} yields:
    \begin{equation}\label{eq:proof_levelsets_3}
        S_{a+\nu+\mu}^{\leq}(\Phi) \cap \partial \mathcal{K} \neq \varnothing \qquad \text{and} \qquad S_{b-\nu-\mu}^{\geq}(\Phi) \cap \partial \mathcal{K} \neq \varnothing.
    \end{equation}
    The setting is visualized in Figure~\ref{fig:levelsets}. The intersections~\eqref{eq:proof_levelsets_3} imply the existence of a point $x_a \in \partial \mathcal{K}$ with $\Phi(x_a) \leq a+\nu+\mu$ and a point $x_b \in \partial \mathcal{K}$ with $\Phi(x_b) \geq b-\nu-\mu$. Since $\Phi$ is continuous and the boundary $\partial \mathcal{K}$ is a connected subset of $\R^{n_{\textup{in}}}$, the image $\Phi(\partial \mathcal{K})$ is an interval in $\R$ by the intermediate value property~\cite[Corollary 22.3]{Ross2013}. 
    Given the assumption on $\nu$, we have $a+\nu+\mu < b-\nu-\mu$, which implies:
    \begin{equation*}
        [a+\nu+\mu, b-\nu-\mu] \subset [\Phi(x_a), \Phi(x_b)] \subset \Phi(\partial \mathcal{K}).
    \end{equation*}
    Since $\nu > 0$ was arbitrary, it follows that for every $c \in (a+\mu, b-\mu)$, there exists some $x_c \in \partial \mathcal{K}$ such that $\Phi(x_c) = c$. Therefore, the level set $S_c(\Phi)$ necessarily intersects the boundary $\partial \mathcal{K}$.
\end{proof}

\begin{remark}
    The distinction between the one-dimensional and higher-dimensional case lies in the topology of the boundary $\partial \mathcal{K}$ of the compact domain $\mathcal{K}$. In one dimension, the boundary of a compact interval consists of two disjoint points and is therefore never connected. In higher dimensions ($n_{\textup{in}} \geq 2$), the boundary of a compact set is for example connected for all convex sets $\mathcal{K}$.
\end{remark}

\begin{figure}
\centering
\hspace{2mm}
\begin{minipage}[t]{0.45\textwidth}

\scalebox{0.78}{	\begin{tikzpicture}[
		lbl/.style={font=\normalsize},
		ptr/.style={->, thin, black, >=stealth}
		]
		
			\definecolor{cDarkOrange}{RGB}{245,195,145}
		\definecolor{cLightOrange}{RGB}{252,228,205}
		\definecolor{cGray}{RGB}{0,0,0}
		\definecolor{cLightBlue}{RGB}{182,211,237}
		\definecolor{cDarkBlue}{RGB}{153,196,237}
		\definecolor{cLineOrange}{RGB}{195,110,40}
		\definecolor{cLineBlue}{RGB}{70,120,190}
		\definecolor{cLabelOrange}{RGB}{130,77,27}
		\definecolor{cLabelBlue}{RGB}{80,130,200}
		
		\fill[cDarkOrange] (-3.9,-3.9) rectangle (3.9,3.9);
		
		\fill[cLightOrange]
		(-0.90, 3.9)
		.. controls (-0.90, 2.0) and (-3.50, 1.5) .. (-3.50, 0)
		arc (180:360:3.50)
		.. controls (3.50, 1.5) and (0.90, 2.0) .. (0.90, 3.9)
		-- cycle;
		\fill[pattern=north west lines, pattern color=cLineOrange!70]
		(-0.90, 3.9)
		.. controls (-0.90, 2.0) and (-3.50, 1.5) .. (-3.50, 0)
		arc (180:360:3.50)
		.. controls (3.50, 1.5) and (0.90, 2.0) .. (0.90, 3.9)
		-- cycle;
		
		\fill[white]
		(-0.86, 3.9)
		.. controls (-0.86, 2.0) and (-3.20, 1.5) .. (-3.20, 0)
		arc (180:360:3.20)
		.. controls (3.20, 1.5) and (0.86, 2.0) .. (0.86, 3.9)
		-- cycle;
		
		\draw[cLineOrange!80!white, thick]
		(-0.86, 3.9)
		.. controls (-0.86, 2.0) and (-3.20, 1.5) .. (-3.20, 0);
		\draw[cLineOrange!80!white, thick]
		(-3.20, 0) arc (180:360:3.20);
		\draw[cLineOrange!80!white, thick]
		(3.20, 0)
		.. controls (3.20, 1.5) and (0.86, 2.0) .. (0.86, 3.9);
		
		\fill[white]
		(-0.39, 3.9)
		.. controls (-0.39, 2.0) and (-2.70, 1.5) .. (-2.70, 0)
		arc (180:360:2.70)
		.. controls (2.70, 1.5) and (0.39, 2.0) .. (0.39, 3.9)
		-- cycle;
		
			\draw[black, semithick]
		(-0.6, 3.9) 
		.. controls (-0.6, 2.0) and (-3., 1.5) .. (-3., 0); 
		
		\draw[black, semithick]
		(-3., 0)
		.. controls (-3., -1.0) and (-1.85, -1.40) .. (-1.70, -2.00)
		.. controls (-1.60, -2.40) and (-1.30, -2.55) .. (-1.00, -2.67)
		arc (245:360:2.8);
		
		\draw[black, semithick]
		(2.98, -.15) 
		.. controls (3., 1.5) and (0.5, 2.0) .. (0.5, 3.9);
		
		
		\fill[cLightBlue]
		(-0.34, 3.9)
		.. controls (-0.34, 2.0) and (-2.40, 1.5) .. (-2.40, 0)
		arc (180:360:2.40)
		.. controls (2.40, 1.5) and (0.34, 2.0) .. (0.34, 3.9)
		-- cycle;
		\fill[pattern=north west lines, pattern color=cDarkBlue!20]
		(-0.34, 3.9)
		.. controls (-0.34, 2.0) and (-2.40, 1.5) .. (-2.40, 0)
		arc (180:360:2.40)
		.. controls (2.40, 1.5) and (0.34, 2.0) .. (0.34, 3.9)
		-- cycle;
		
		\draw[cLineBlue!80!white, thick]
		(-0.34, 3.9)
		.. controls (-0.34, 2.0) and (-2.40, 1.5) .. (-2.40, 0);
		\draw[cLineBlue!80!white, thick]
		(-2.40, 0) arc (180:360:2.40);
		\draw[cLineBlue!80!white, thick]
		(2.40, 0)
		.. controls (2.40, 1.5) and (0.34, 2.0) .. (0.34, 3.9);
		
		\fill[cDarkBlue]
		(-0.30, 3.9)
		.. controls (-0.30, 2.0) and (-2.10, 1.5) .. (-2.10, 0)
		arc (180:360:2.10)
		.. controls (2.10, 1.5) and (0.30, 2.0) .. (0.30, 3.9)
		-- cycle;
		
		\draw[black, very thick] (-3.9,-3.9) rectangle (3.9,3.9);
		
		\draw[black, semithick, rotate around={72:(-2.75, 0.50)}] (-2.75, 0.30) ellipse (0.32 and 0.07);
		\draw[black, semithick, rotate around={95:(-2.68, -0.25)}] (-2.68, -0.25) ellipse (0.18 and 0.06);
		\draw[black, semithick, rotate around={108:(-2.55, -0.75)}] (-2.55, -0.75) ellipse (0.10 and 0.05);

		
		\node[anchor=north west, font=\normalsize, black]
		at (-3.4, 3.4) {$S_{b-\nu}^{\geq}(\overline{\Phi})$};
		
		\node[anchor=north west, font=\normalsize, black] (loLabel)
		at (-3.9, 2.4) {$S_{b-\nu-\mu}^{\geq}(\Phi)$};
		\draw[ptr] (loLabel.south) -- (-2.8, 1.1);
		
		\node[anchor=west, font=\normalsize, black]
		at (-0.6, -2.7) {$S_c(\Phi)$};
		
		\node[anchor=center, font=\normalsize, black] (sbLabel)
		at (0, -1.) {$S_{a+\nu+\mu}^{\leq}(\Phi)$};
		\draw[ptr] (sbLabel.south) -- (1.5, -1.6);
		
		\node[anchor=center, font=\normalsize, black]
		at (0, -.) {$\mathcal{S} := S_{a+\nu}^{\leq}(\overline{\Phi})$};
		
		\node[anchor=west, font=\normalsize, black]
		at (3.95, -2.7) {$\partial\mathcal{K}$};
		
	\end{tikzpicture}	}
    \captionof{figure}{Geometry of the level sets in the proof of Theorem~\ref{thm:constraints_nd}. As the reference map $\overline{\Phi}$ has no critical points, the level sets $S_{a+\nu}^{\leq}(\overline{\Phi})$ and $S_{b-\nu}^{\geq}(\overline{\Phi})$ must intersect the connected boundary $\partial \mathcal{K}$. As the distance between $\Phi$ and $\overline{\Phi}$ is bounded, also the level sets $S_{a+\nu+\mu}^{\leq}(\Phi)$ and $S_{b-\nu-\mu}^{\geq}(\Phi)$ have to intersect $\partial \mathcal{K}$. Any decision boundary $S_c(\Phi)$ forced in between $S_{a+\nu+\mu}^{\leq}(\Phi)$ and $S_{b-\nu-\mu}^{\geq}(\Phi)$ has to intersect $\partial \mathcal{K}$, too.}
\label{fig:levelsets}
\end{minipage}
\hspace{4mm}
\begin{minipage}[t]{0.45\textwidth}
    \centering
\scalebox{0.78}{	\begin{tikzpicture}[
		lbl/.style={font=\normalsize},
		ptr/.style={->, thin, black, >=stealth}
		]
		
		\draw[black,dashed, thick]
		(-0.425, 3.9)
		.. controls (-0.425, 2.0) and (-2.95, 1.5) .. (-2.95, 0);
		\draw[black,dashed, thick]
		(-2.95, 0) arc (180:360:2.95);
		\draw[black,dashed, thick]
		(2.95, 0)
		.. controls (2.95, 1.5) and (0.425, 2.0) .. (0.425, 3.9);
		
			\draw[black, semithick]
	(-0.6, 3.9) 
	.. controls (-0.6, 2.0) and (-3., 1.5) .. (-3., 0); 
	
	\draw[black, semithick]
	(-3., 0)
	.. controls (-3., -1.0) and (-1.85, -1.40) .. (-1.70, -2.00)
	.. controls (-1.60, -2.40) and (-1.30, -2.55) .. (-1.00, -2.67)
	arc (245:360:2.8);
	
	\draw[black, semithick]
	(2.98, -.15) 
	.. controls (3., 1.5) and (0.5, 2.0) .. (0.5, 3.9);
		
		\draw[black, semithick, rotate around={72:(-2.75, 0.50)}] (-2.75, 0.30) ellipse (0.32 and 0.07);
		\draw[black, semithick, rotate around={95:(-2.68, -0.25)}] (-2.68, -0.25) ellipse (0.18 and 0.06);
		\draw[black, semithick, rotate around={108:(-2.55, -0.75)}] (-2.55, -0.75) ellipse (0.10 and 0.05);
		
		\draw[black, very thick] (-3.9,-3.9) rectangle (3.9,3.9);
		
		\draw[black, thick] (-2.2, -3.4) -- (-1.2, -3.4);
		\node[right] at (-1.1, -3.4) {$S_c(\Phi)$};
		
		\draw[black, dashed, thick] (0.6, -3.4) -- (1.6, -3.4);
		\node[right] at (1.7, -3.4) {$S_c(\overline{\Phi})$};
		
		\node[anchor=west, font=\normalsize, black]
		at (3.95, -2.7) {$\partial\mathcal{K}$};
		
	\end{tikzpicture}}
    \captionof{figure}{Topological failure on the circle dataset. The decision boundary $S_c(\overline{\Phi})$ has to intersect the boundary $\partial \mathcal{K}$, as the reference map $\overline{\Phi}$ has no critical points. Even though $\Phi$ can have critical points in $\mathcal{K}$, the fact that $\Phi$ is $\mu$-close to $\overline{\Phi}$ means $S_c(\Phi)$ must still intersect $\partial \mathcal{K}$. To optimize the classification under these constraints, the network forms a ``tunnel''.}
    \label{fig:circledataset}
\end{minipage}
\hspace{2mm}
\end{figure}

In the context of classification problems (cf.~Definition~\ref{def:binary_classification}), Theorem~\ref{thm:constraints_nd} reveals a severe structural limitation for nested datasets. A prototypical example is the two-dimensional circle dataset (see Figure~\ref{fig:circledataset}), where an inner cluster of one class is completely surrounded by points of the other class. As discussed in Section~\ref{sec:topological_restrictions_phibar}, to perfectly classify this dataset, an ideal decision boundary would need to form a closed curve contained entirely within the interior of the compact domain $\mathcal{K}$. 

The classification problem defined by the circle dataset can be generalized to any dimension $n_{\textup{in}} \geq 2$ using a quadratic function
\begin{equation*}
    \Psi_z: \mathcal{K} \rightarrow \R, \quad \Psi_z(x) = \sum_{j = 1}^{\nin}(x_j-z_j)^2 \quad \text{with} \quad z\in \textup{int}(\mathcal{K}).
\end{equation*}
For an appropriate $c^\ast \in \R$, all data points in the sub-level set $S_{c^\ast}^{<}(\Psi_z)$ (an $\nin$-dimensional ball) are assigned the label $c_0$, while the super-level set $S_{c^\ast}^{>}(\Psi_z)$ defines the class $c_1$ surrounding the class $c_0$.

Theorem~\ref{thm:constraints_nd} implies that such datasets cannot be perfectly classified by a map $\Phi$ that is $\mu$-close to a reference map $\overline{\Phi}$ without critical points. Because $\overline{\Phi}$ lacks critical points, its decision boundary $S_{c^\ast}(\overline{\Phi})$ is forced to intersect the boundary $\partial \mathcal{K}$. Since the map $\Phi$ is trapped in a $\mu$-tube around $\overline{\Phi}$, its decision boundary $S_{c^\ast}(\Phi)$ inherits this property and is also forced to intersect $\partial \mathcal{K}$, as visualized in Figure~\ref{fig:circledataset}. In contrast to $\overline{\Phi}$, the map $\Phi$ can in principle have critical points in the interior of the domain~$\mathcal{K}$. Nevertheless, the map $\Phi$ cannot isolate the inner class $c_0$ with an $\nin$-dimensional hypersphere as a decision boundary. Instead, a natural geometric compromise occurs: the network forms a ``tunnel'', a continuous, narrow region of class $c_0$ extending to the boundary $\partial \mathcal{K}$. Effectively, this topology shifts the required critical point outside the observation domain $\mathcal{K}$, allowing the decision boundary $S_{c^\ast}(\Phi)$ to reach $\partial \mathcal{K}$ while still minimizing the empirical classification error.

\section{Case $0<\alpha \ll 1$: Close to Neural ODEs}
\label{sec:node}

In Section~\ref{sec:node_model}, we introduce neural ordinary differential equations (neural ODEs). Next, we motivate the study of neural ODEs in Section~\ref{sec:node_relationship} by showing that neural ODEs and ResNets with sufficiently small $\alpha \coloneqq \frac{\delta}{\eps}$ can be interpreted as discrete and continuous counterparts of the same underlying dynamics. Additionally, we calculate the approximation error between the neural ODE and ResNet dynamics arising from an explicit Euler discretization. Finally, we show in Section~\ref{sec:node_criticalpoints} that, in analogy to non-augmented neural ODEs, critical points cannot exist in non-augmented ResNets if $0 < \alpha \ll 1$.

\subsection{Neural ODEs}
\label{sec:node_model}

In this section, we introduce neural ODEs, which can be interpreted as the infinite-depth limit of ResNets. In analogy to ResNets, we include two additional transformations before and after the initial value problem, and we distinguish between non-augmented and augmented architectures.
We consider neural ODEs based on the solution $h:[0,T] \rightarrow \R^{\nhid}$ of an initial value problem 
\begin{equation} \label{eq:IVP} 
    \frac{\dd h}{\dd t} = f(h(t),\theta(t)), \qquad h(0) = a \in \mathcal{A} \subset \R^{\nhid}, 
\end{equation}
with a non-autonomous vector field $f : \R^{\nhid} \times \R^p \rightarrow \R^{\nhid}$, parameter function $\theta \in \Theta_{\text{NODE}} \subset C^0([0,T],\R^p)$, and a non-empty set of possible initial conditions $\mathcal{A} \subset \R^{\nhid}$. When required, the solution with initial condition $h(0) = a$ is denoted by $h_a(t)$ to emphasize the dependence on the initial condition. 

As with ResNets, we allow flexibility regarding the input and output dimensions of neural ODEs. We define \emph{neural ODEs} as the input-output maps given by
\begin{equation} \label{eq:NODE}
    \overline{\Phi}: \mathcal{X} \rightarrow \R^{\nout}, \qquad \overline{\Phi}(x) = \tilde{\lambda}(h_{\lambda(x)}(T)),
\end{equation} 
with $\mathcal{X}\subset \R^{\nin}$ open, input transformation $\lambda: \R^{\nin} \rightarrow \R^{\nhid}$, and output transformation $\tilde{\lambda}: \R^{\nhid}\rightarrow \R^{\nout}$. To more easily distinguish between neural ODE and ResNet architectures, we denote them in this section as $\overline{\Phi}$ and $\Phi$, respectively. As with ResNets, $\lambda$ and $\tilde{\lambda}$ are often chosen to be affine linear, but they can also be nonlinear maps. The architecture~\eqref{eq:NODE} depends on the time-$T$ map $h_{\lambda(x)}(T)$ of the initial value problem~\eqref{eq:IVP}. We assume that for every $a \in \mathcal{A}$, the solution of~\eqref{eq:IVP} exists on the entire time interval $[0,T]$. In the upcoming definition, we denote by $C^{i,j}(\mathcal{X}\times \mathcal{Y},\mathcal{Z})$ the space of functions $f: \mathcal{X}\times \mathcal{Y} \rightarrow \mathcal{Z}$ that are $i$-times continuously differentiable in the first variable and $j$-times continuously differentiable in the second variable.

\begin{definition}[Neural ODE] \label{def:node}
    For $k \geq 1$, we denote by $\NODE$ the set of all neural ODE architectures $\overline{\Phi}: \mathcal{X} \rightarrow \R^{\nout}$, $\mathcal{X}\subset \R^{\nin}$ open, defined by~\eqref{eq:NODE} with 
    \begin{itemize}
        \item input transformation $\lambda \in C^k(\R^{\nin},\R^{\nhid})$,
        \item output transformation $\tilde\lambda \in C^k(\R^{\nhid},\R^{\nout})$,
        \item parameter function $\theta \in \Theta_{\text{NODE}} \subset C^0([0,T],\R^p)$,
        \item vector field $f \in C^{k,0}(\R^{\nhid} \times \R^p, \R^{\nhid})$ with a set of possible initial conditions $\mathcal{A} \subset \R^{\nhid}$, such that there exists a unique solution of~\eqref{eq:IVP} on $[0,T]$ for every $a \in \mathcal{A}$.
    \end{itemize}
\end{definition}

The set of possible initial conditions of~\eqref{eq:IVP} is defined by $\mathcal{A} \coloneqq \lambda(\mathcal{X})$, as the input $x \in \mathcal{X} \subset \R^{\nin}$ of the neural ODE is mapped under the transformation $\lambda$ to $a = \lambda(x)$. The regularity of neural ODEs $\overline{\Phi} \in \NODE$ follows directly from the regularity of the vector field $f$, as explained in~\cite[Lemma 4.3]{kk2025}. We restrict Definition~\ref{def:node} to the case $k \geq 1$ to guarantee uniqueness of solution curves of~\eqref{eq:IVP} and hence well-posedness of the neural ODEs.

\begin{remark}
    The regularity assumptions on the vector field $f$ and the parameter function $\theta$ can be weakened by considering, for example, the Carathéodory conditions (cf.~\cite{Hale1977}), as discussed in \cite{Kuehn2023}. In this way, the parameter function $\theta$ does not need to be continuous, allowing typical choices such as piecewise constant parameter functions (cf.~\cite{Chen2018}).
\end{remark}

The following classification of non-augmented and augmented neural ODEs is independent of the choice of the vector field $f$ and applies to all neural ODEs $\overline{\Phi}\in \NODE$. The concept of non-augmented and augmented neural ODEs is illustrated in Figure~\ref{fig:node_architectures}. 

\begin{definition}[Neural ODE Classification]\label{def:node_classes}
    The class of neural ODEs denoted by $\NODE$, $k\geq 1$, $\mathcal{X}\subset \R^{\nin}$ open, is subdivided as follows:
    \begin{itemize}
        \item \emph{Non-augmented neural ODE} $\overline{\Phi} \in \textup{NODE}^k_\textup{N}(\mathcal{X},\R^{\nout})$: it holds $\nin \geq \nhid$.
        \item \emph{Augmented neural ODE} $\overline{\Phi} \in \textup{NODE}^k_\textup{A}(\mathcal{X},\R^{\nout})$: it holds $\nin < \nhid$.
    \end{itemize}
\end{definition}

\begin{figure}[h]
	\centering
	\begin{subfigure}{0.41\textwidth}
		\centering
		\begin{overpic}[scale = 0.23,tics=10]
			{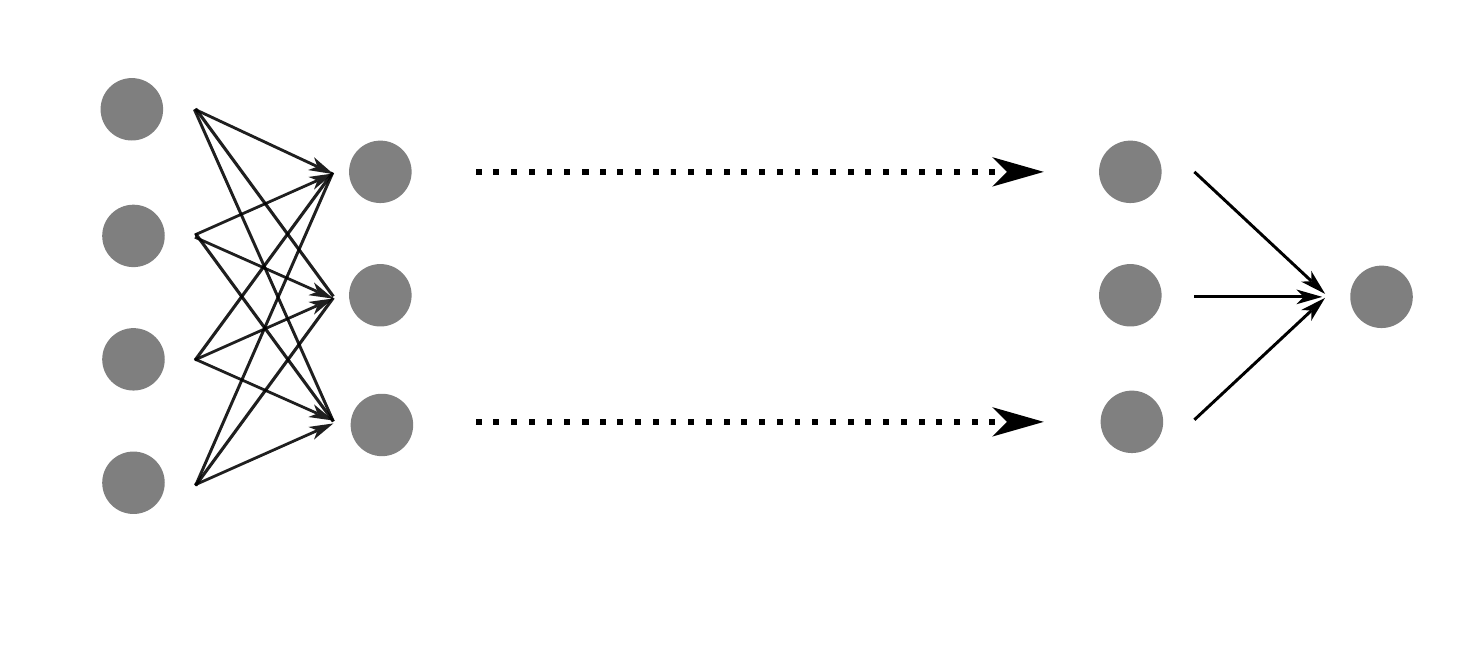}
			\put(7.5,1){\textcolor{black!70}{$x$}}
			\put(21,1){\textcolor{black!70}{$\lambda(x)$}}
			\put(71,1){\textcolor{black!70}{$h(T)$}}
			\put(89,1){\textcolor{black!70}{$\overline{\Phi}(x)$}}
			\put(44,22){\textcolor{black!70}{\normalsize{ODE}}}
		\end{overpic}
		\caption{Example of a non-augmented neural ODE $\overline{\Phi} \in \text{NODE}^k_{\text{N}}(\R^4,\R)$ with $\nhid = 3$.}
		\label{fig:node_arch_nonaugmented}
	\end{subfigure}
    \hspace{12mm}
	\begin{subfigure}{0.41\textwidth}
		\centering
		\begin{overpic}[scale = 0.23,tics=10]
			{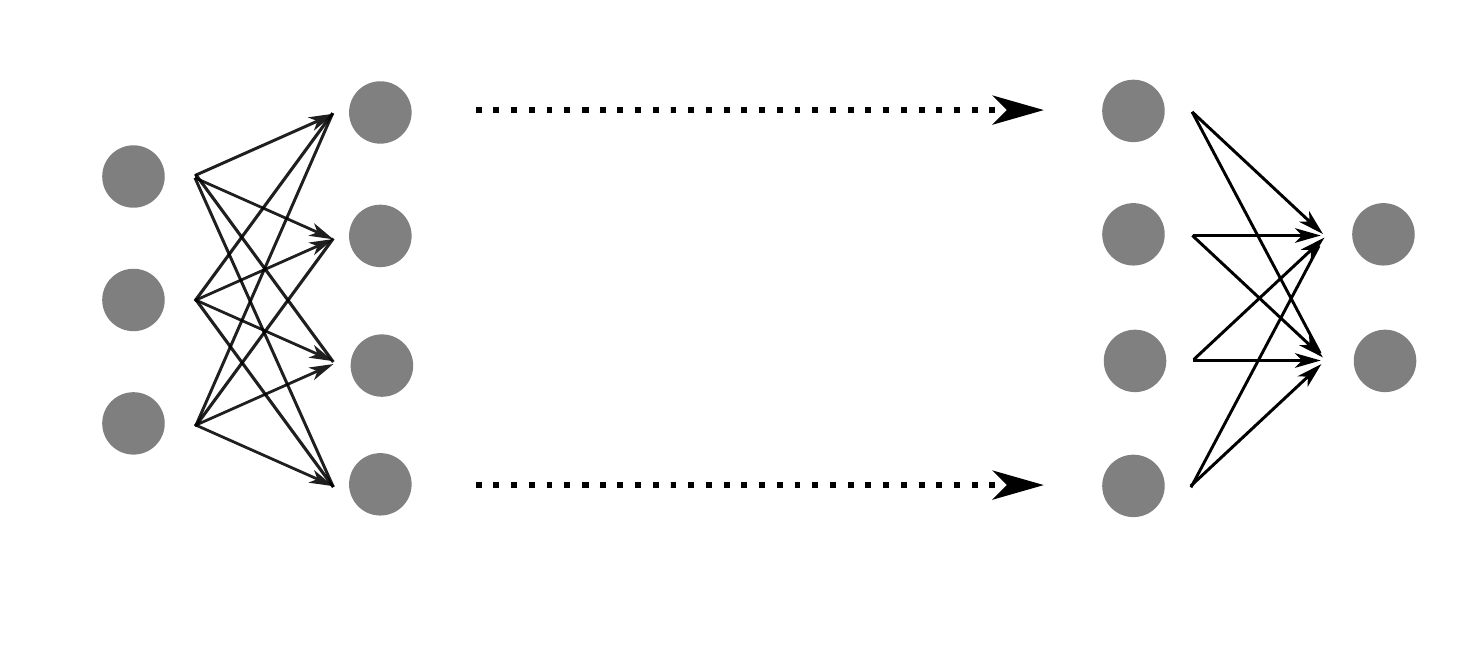}
			\put(7.5,1){\textcolor{black!70}{$x$}}
			\put(21,1){\textcolor{black!70}{$\lambda(x)$}}
			\put(71,1){\textcolor{black!70}{$h(T)$}}
			\put(89,1){\textcolor{black!70}{$\overline{\Phi}(x)$}}
			\put(44,22){\textcolor{black!70}{\normalsize{ODE}}}
		\end{overpic}
		\caption{Example of an augmented neural ODE $\overline{\Phi} \in \text{NODE}^k_{\text{A}}(\R^3,\R^2)$ with $\nhid = 4$.}
		\label{fig:node_arch_augmented}
	\end{subfigure}
	\caption{Classification of neural ODE architectures depending on the dimension of the input data and the vector field. Figure \ref{fig:node_architectures}\subref{fig:node_arch_nonaugmented}-\subref{fig:node_arch_augmented} is adapted from \cite[Figure 4.1(a)-(b)]{kk2025}.}
	\label{fig:node_architectures}
\end{figure}

\subsection{Relationship to ResNets}
\label{sec:node_relationship}

The following proposition demonstrates that, given our precise definitions, ResNets (cf.~Definition~\ref{def:resnet}) and neural ODEs (cf.~Definition~\ref{def:node}) can be understood as discrete and continuous counterparts of the same underlying dynamics. This is based on the well-established observation that ResNets can be interpreted as the explicit Euler discretization of neural ODEs \cite{Weinan2017}.

  \begin{proposition}[Relationship between ResNets and Neural ODEs]\label{prop:resnet_neuralODE}
    Let $\mathcal{X} \subset \R^{\nin}$ be open, $k \geq 1$.
    \begin{enumerate}[label=(\alph*), font=\normalfont]
        \item \label{prop:resnet_neuralODE_a} Given $L \in \mathbb{N}_{\geq 1}$ and $\overline{\Phi} \in \NODE$ defined over $[0,T]$, the explicit Euler discretization of $\overline{\Phi}$ with step size $\delta \coloneqq \frac{T}{L}$ results in a ResNet $\Phi \in \textup{RN}^k_{1,\delta}(\mathcal{X},\R^{\nout})$ with $L$ hidden layers.
        \item  \label{prop:resnet_neuralODE_b} Let $\Phi \in \RN$ be a ResNet with $L$ hidden layers. If the parameter dimensions are constant, the layer functions are identical, i.e., $p_l = p$ and $f_l(\cdot,\theta_l) \coloneqq f_{\textup{RN}}(\cdot,\theta_l)$ for all $l\in\{1,\ldots,L\}$, and $f_\textup{RN} \in C^{k,0}(\R^{\nhid}\times\R^p,\R^{\nhid})$, then there exists a neural ODE $\overline{\Phi} \in \NODE$ defined over $[0, L\delta]$ such that $\Phi$ is its explicit Euler discretization.
    \end{enumerate}
\end{proposition}

The proof of this proposition relies on standard arguments and mainly consists of translating the notation between ResNets and discretized neural ODEs. It is therefore proven in Appendix~\ref{app:proof_resnet_neuralODE} and included for completeness. The following remark argues that the assumptions of Proposition~\ref{prop:resnet_neuralODE}\ref{prop:resnet_neuralODE_b} are not restrictive in typical settings.

\begin{remark}
	The assumption $p_l = p$ that the layer parameters are of constant dimension for all $l \in \{1,\ldots, L\}$ can be established for any given $\Phi \in \RN$ by setting $p = \max\{p_1, \ldots, p_L\}$ and adding zero components to every strictly smaller $p_l$.
	
	Furthermore, the assumption $f_l(\cdot,\theta_l) \coloneqq f_\textup{RN}(\cdot,\theta_l)$ for $l\in \{1,\ldots,L\}$ is fulfilled if the residual function $f_l$ is of the typical form~\eqref{eq:resnet_typical_f} with~$\sigma_l = \sigma$ for every layer. The dependence on the layer $l$ is then only encoded in the parameters~$\theta_l \in \Theta_l \subset \R^{p_l}$.

    Finally, the additional assumption $f_\textup{RN} \in C^{k,0}(\R^{\nhid}\times\R^p,\R^{\nhid})$ in part~\ref{prop:resnet_neuralODE_b} requires joint continuity of the residual function in both the hidden state $h$ and the parameters $\theta$. This is not restrictive in practice: for canonical ResNets $\Phi \in \textup{RN}^k_{\eps,\delta,\sigma}(\mathcal{X},\R^{\nout})$ with activation function $\sigma \in C^k(\R,\R)$, the residual function~\eqref{eq:resnet_typical_f} is a composition of matrix-vector products and continuous nonlinearities, and hence automatically satisfies $f_\textup{RN} \in C^{k,0}(\R^{\nhid}\times\R^p,\R^{\nhid})$.
\end{remark}

In the following, we explicitly calculate the approximation error between the ResNet and neural ODE architectures considered, thereby adapting~\cite[Proposition 1]{Sander2022} to the more general neural ODE architecture of Definition~\ref{def:node}. The following theorem shows that the global error of the input-output map between neural ODEs and their explicit Euler discretization depends linearly on the step size or residual parameter $\delta$. Unless stated otherwise, we always consider Lipschitz constants with respect to the max-norm.

\begin{theorem}[Approximation Error between ResNets and Neural ODEs] \label{th:eulerdiscretization}
	Consider a neural ODE $\overline{\Phi} \in \textup{NODE}^1(\mathcal{X},\R^{\nout})$, $\mathcal{X} \subset \R^{\nin}$ open, based on the initial value problem 
	\begin{equation} \label{eq:IVP2}
		\frac{\dd h}{\dd t} = f(h(t),\theta(t)), \qquad h(0) = a \in \mathcal{A} \subset \R^{\nhid},
	\end{equation}
	on the time interval $t\in[0,T]$ with the following properties:
	\begin{itemize}
		\item The vector field $f\in C^{1,1}(\R^{\nhid} \times \R^p,\R^{\nhid})$ is continuously differentiable.
        \item The set of possible parameter functions fulfills $ \Theta_{\textup{NODE}}\subset  C^1([0,T],\R^p)$. 
		\item The output transformation $\tilde{\lambda}$ has global Lipschitz constant $K_{\tilde{\lambda}}$.
		\item The vector field $f$ has global Lipschitz constant~$K_\theta$ with respect to the first variable $h$, and the solution $h: [0,T] \rightarrow \R^{\nhid}$ has a bounded second derivative  $M_\theta \coloneqq \sup_{t \in [0,T]} \norm{h''(t)}_\infty$ uniformly for all $\theta \in \Theta_{\textup{NODE}}$. 
	\end{itemize}
	Fix $L \in \mathbb{N}_{\geq 1}$, then the ResNet $\Phi \in \textup{RN}^1_{1,\delta}(\mathcal{X},\R^{\nout})$ obtained by Euler's explicit method with step size $\delta \coloneqq \frac{T}{L}$ and the same transformations $\lambda$, $\tilde{\lambda}$, has the iterative update rule
	\begin{equation*}
		h_{l} = h_{l-1} + \delta \cdot f(h_{l-1},\theta_{l}), \qquad \theta_{l} \coloneqq \theta((l-1)\delta), \qquad l \in \{1,\ldots, L\},
	\end{equation*}
	with $h_0 = h(0)$. The global approximation error between the neural ODE and the corresponding ResNet is bounded by
	\begin{equation}  \label{eq:node_resnet_error}
		\norm{\Phi-\overline{\Phi}}_{\infty,\mathcal{X}}  \leq K_{\tilde{\lambda}} \cdot \frac{M_\theta \delta }{2K_\theta}\cdot \left(e^{K_\theta T}-1\right).
	\end{equation}
\end{theorem}

\begin{proof}
	By the regularity assumptions $f\in C^{1,1}(\R^{\nhid} \times \R^p,\R^{\nhid})$ and $\theta \in \Theta_{\textup{NODE}}\subset C^1([0,T],\R^p)$, the solution $h: [0,T] \rightarrow \R^{\nhid}$ is twice continuously differentiable. As $\Phi$ and $\overline{\Phi}$ are based on the same input transformation $\lambda$, it holds $h_0 = h(0)$ for every input $x\in \mathcal{X}$. According to~\cite[Theorem 6.3]{Atkinson1989}, the error between the solution of~\eqref{eq:IVP2} and the discrete solution $\{h_0,\ldots,h_L\}$ obtained by Euler's explicit method satisfies 
	\begin{equation*}
		\max_{l \in \{0, \ldots, L\}} \norm{h(t_l) - h_l}_\infty \leq \frac{M_\theta \delta }{2K_\theta}\cdot \left(e^{K_\theta T}-1\right)
	\end{equation*}
	with constants $K_\theta$, $M_\theta$ and step size $\delta =\frac{T}{L}$, uniformly in $x \in \mathcal{X}$. Consequently, it follows for the approximation error between the ResNet and the neural ODE that 
	\begin{equation*}
		\norm{\Phi-\overline{\Phi}}_{\infty,\mathcal{X}}  =   \sup_{x \in \mathcal{X}} \normm{\tilde{\lambda}(h_L) - \tilde{\lambda}(h(T))}_\infty \leq K_{\tilde{\lambda}} \cdot \frac{M_\theta \delta }{2K_\theta}\cdot \max_{l \in \{0, \ldots, L\}} \norm{h(t_l) - h_l}_\infty 
	\end{equation*}
	using the Lipschitz continuity of the output transformation $\tilde{\lambda}$. This completes the proof.
\end{proof}

\begin{remark}\label{rem:discretizations}
    While we focus on the explicit Euler discretization due to its direct structural relation to standard ResNets, higher-order numerical schemes such as Runge-Kutta methods can be used to obtain tighter error bounds. In the discrete setting, these higher-order methods correspond to multi-step ResNet architectures~\cite{haber2017,Lu2018}.
\end{remark}

We now specialize Theorem~\ref{th:eulerdiscretization} to canonical ResNets with the typical residual function~\eqref{eq:resnet_typical_f}. The proof consists of explicitly computing the constants $M_\theta$ and $K_\theta$ and is given in Appendix~\ref{app:cor_proof_resnet_neuralODE}.

\begin{corollary}[Approximation Error between canonical ResNets and Neural ODEs] \label{cor:node_resnet_error}
	Consider a neural ODE $\overline{\Phi} \in \textup{NODE}^1(\mathcal{X},\R^{\nout})$, $\mathcal{X} \subset \R^{\nin}$ open, based on the initial value problem
	\begin{equation*}
		\frac{\dd h}{\dd t} = \widetilde{W} \sigma(W h(t) +b) + \tilde{b}, \qquad h(0) = h_0 \in \R^{\nhid},
	\end{equation*}
	with parameters $(W,\widetilde{W},b,\tilde{b})\in \Theta \subset \R^{m \times \nhid} \times \R^{\nhid\times m} \times \R^m \times \R^{\nhid}$ on the time interval $t\in[0,T]$ with the following properties:
	\begin{itemize}
		\item The activation function $\sigma\in C^1(\R,\R)$ fulfills Assumptions~\ref{ass:A_activation_lipschitz} and~\ref{ass:A_activation_bounded}, i.e., there exist constants $S_\sigma, K_\sigma > 0$ with  $\norm{\sigma}_{\infty,\R} \leq S_\sigma$ and $\norm{\sigma'}_{\infty,\R} \leq K_\sigma$.
		\item The output transformation $\tilde{\lambda}$ has global Lipschitz constant $K_{\tilde{\lambda}}$.
        \item The weight matrices $W$, $\widetilde{W}$, and the biases $\tilde{b}$ fulfill Assumption~\ref{ass:B_weights_bounded} with respect to the max-norm, i.e., there exist constants $\omega_\infty, \widetilde{\omega}_\infty, \tilde{\beta}_\infty \geq 0$, such that $\norm{W}_\infty\leq \omega_\infty$, $\normm{\widetilde{W}}_\infty \leq \widetilde{\omega}_\infty $ and $\normm{\tilde{b}}_\infty \leq \tilde{\beta}_\infty$  for all $(W,\widetilde{W},b,\tilde{b})\in \Theta$.
	\end{itemize}
	Then, with the notation from Theorem~\ref{th:eulerdiscretization} it follows that
	\begin{equation}\label{eq:node_resnet_error_2}
		  \norm{\Phi-\overline{\Phi}}_{\infty,\mathcal{X}} \leq  K_{\tilde{\lambda}} \cdot \frac{\left(\widetilde{\omega}_\infty  S_\sigma + \tilde{\beta}_\infty\right) \cdot \delta }{2\omega_\infty}\cdot \left(e^{ K_{\sigma} \widetilde{\omega}_\infty\omega_\infty T}-1\right).
	\end{equation}
\end{corollary}

The preceding theorem and corollary show that, in the limit $\delta \rightarrow 0$, neural ODEs and the corresponding ResNets can be viewed as continuous and discrete realizations of the same underlying dynamics. In the setting of Theorem~\ref{th:eulerdiscretization}, this correspondence holds for ResNets with skip parameter $\eps = 1$. More generally, ResNets with arbitrary skip parameter $\eps \neq 1$ are also related to neural ODEs, as discussed in Proposition~\ref{prop:resnet_neuralODE}\ref{prop:resnet_neuralODE_b}. In this general case, taking the limit $\delta \rightarrow 0$ for fixed $\eps > 0$ means $\alpha \coloneqq \frac{\delta}{\eps} \rightarrow 0$, as also discussed in Sections~\ref{sec:resnet_gradient} and~\ref{sec:resnet_tunnel}. 

\subsection{Existence of Critical Points}     
\label{sec:node_criticalpoints}

In this section, we characterize the existence of critical points in neural ODEs and ResNets in the parameter regime $0<\alpha \ll 1$. For that purpose, we first extend the result of~\cite{kk2025} that non-augmented neural ODEs cannot have any critical points to our setting of general input and output transformations.

\begin{theorem}[No Critical Points in Non-Augmented Neural ODEs]\label{th:node_criticalpoints}
	Let $\overline{\Phi} \in \textup{NODE}^1_{\textup{N}}(\mathcal{X},\R)$, $\mathcal{X}\subset \R^{\nin}$ open, be a scalar non-augmented neural ODE and let the input and output transformations $\lambda$, $\tilde{\lambda}$ fulfill Assumption~\ref{ass:C_input_output}. Then $\overline{\Phi}$ cannot have any critical points, i.e., $\nabla_x \overline{\Phi}(x) \neq 0$ for all $x \in \mathcal{X}$.
\end{theorem}

\begin{proof}
    As $\overline{\Phi} \in C^1(\mathcal{X},\R)$, we can calculate the network input gradient with the multi-dimensional chain rule, which yields
	\begin{equation} \label{eq:node_gradient}
		\nabla_x\overline{\Phi}(x) = \left(\partial_x \lambda(x)\right)^\top \left(\partial_a h_a(T) \right)^\top \left(\partial_{h_a(T)}{\tilde{\lambda}}(h_a(T))\right)^\top \in \R^{\nin},
	\end{equation}
	where $\partial_x \lambda(x) \in \R^{\nhid \times \nin}$, $\partial_a h_a(T)  \in \R^{\nhid \times \nhid}$, and $\partial_{h_a(T)}{\tilde{\lambda}}(h_a(T)) \in \R^{1 \times \nhid}$ are Jacobian matrices.
	By~\cite[Proposition 4.10]{kk2025}, the Jacobian matrix $\partial_a h_a(T)$ of the time-$T$ map $h_a(T)$ with respect to the initial condition $a$ always has full rank~$\nhid$. 
    By Assumption~\ref{ass:C_input_output}, also the Jacobian matrices $\partial_x \lambda(x) \in \R^{\nhid \times \nin}$ and $\partial_{h_a(T)}{\tilde{\lambda}}(h_a(T)) \in \R^{1 \times \nhid}$ have both full rank.
	
    As $\overline{\Phi}$ is scalar and non-augmented, it holds that $\nin \geq \nhid \geq \nout = 1$, such that the dimensions in the matrix product~\eqref{eq:node_gradient} are monotonically decreasing. It follows that $\nabla_x \overline{\Phi}(x)$ always has full rank~$1$ (i.e., it is non-zero), uniformly in $x \in \mathcal{X}$, cf.~\cite[Lemma C.1]{kk2025}.
\end{proof}

In the following theorem, we show for non-augmented ResNets $\Phi \in \textup{RN}^1_{\eps,\delta,\textup{N}}(\mathcal{X},\R)$, $\mathcal{X}\subset \R^{\nin}$ open, that the property that no critical points exist, persists for $\alpha \coloneqq \frac{\delta}{\eps}$ sufficiently small. In the case of canonical ResNets we provide an explicit upper bound on the ratio~$\alpha$, below which $\Phi$ cannot have any critical points.

\begin{theorem}[No Critical Points in Non-Augmented ResNets with $0<\alpha \ll 1$] \label{th:resnet_criticalpoints_alphasmall}
	Consider a scalar non-augmented ResNet $\Phi \in \textup{RN}^1_{\eps,\delta,\textup{N}}(\mathcal{X},\R)$, $\mathcal{X}\subset \R^{\nin}$ open, with the following properties:
	\begin{itemize}
		\item The input and output transformations $\lambda$ and $\tilde{\lambda}$ fulfill Assumption~\ref{ass:C_input_output}.
		\item All residual functions $f_l(\cdot,\theta_l)\in C^1(\R^{\nhid},\R^{\nhid})$ are globally Lipschitz continuous with Lipschitz constant $K_f$ (w.r.t.~the max-norm), uniformly for all $\theta_l \in \Theta_l \subset \R^{p_l}$ and $l \in \{1,\ldots,L\}$.
	\end{itemize}
	Then, if $\alpha = \frac{\delta}{\eps} < \frac{1}{K_f}$,  $\Phi$ cannot have any critical points, i.e., $\nabla_x \Phi(x) \neq 0$ for all $x \in \mathcal{X}$. If the ResNet $\Phi \in \textup{RN}^1_{\eps,\delta,\sigma,\textup{N}}(\mathcal{X},\R)$ is canonical, the assumption on the residual functions $f_l$ can be replaced by the following:
	\begin{itemize}
		\item The component-wise applied activation functions $\sigma_l\in C^1(\R,\R)$ fulfill Assumption~\ref{ass:A_activation_lipschitz}, i.e., there exists a constant $K_\sigma > 0$ with $\norm{\sigma'_l}_{\infty,\R} \leq K_\sigma$ for all $l\in \{1,\ldots,L\}$.
		\item The largest singular value of the matrix product $W_l\widetilde{W}_{l}$ is uniformly bounded from above, i.e., there exists a constant $\nu_{\max}>0$ such that  $\sigma_{\max}(W_l\widetilde{W}_{l}) \leq \nu_{\max}$ for all $\theta_l = (W_l,\widetilde{W}_l,b_l,\tilde{b}_l)\in \Theta_l$ and $l \in \{1,\ldots,L\}$.
	\end{itemize}
	Then, if $\alpha < \frac{1}{\nu_{\max} \cdot K_{\sigma}}$,  $\Phi$ cannot have any critical points.
\end{theorem}

\begin{proof}
	We prove the theorem by applying Proposition~\ref{prop:resnet_nonaugmented}, which requires showing that for $\alpha <\frac{1}{K_{f}}$, it holds that $- \frac{1}{\alpha}$ is not an eigenvalue of the Jacobian matrix $ \partial_{h_{l-1}}f_l(h_{l-1},\theta_l)$ for all $l \in \{1,\ldots,L\}$, $h_{l-1} \in \R^{\nhid}$, and $\theta_l \in \Theta_l \subset \R^{p_l}$. For the eigenvalues $\lambda_i$, we can estimate
	\begin{equation*}
		\max_{i \in \{1,\ldots,\nhid\}} \abs{ \lambda_i( \partial_{h_{l-1}}f_l(h_{l-1},\theta_l)) } \leq \norm{\partial_{h_{l-1}}f_l(h_{l-1},\theta_l)}_\infty \leq K_f < \frac{1}{\alpha},
	\end{equation*} 
	where we used the fact that the maximal absolute value of the eigenvalues~$\lambda_i$ is bounded above by every induced matrix norm. By the mean value theorem, the global Lipschitz constant $K_{f}$ is an upper bound for the norm of the Jacobian matrix $\partial_{h_{l-1}}f_l(h_{l-1},\theta_l)$. Consequently, $-\frac{1}{\alpha}$ cannot be an eigenvalue of $\partial_{h_{l-1}}f_l(h_{l-1},\theta_l)$ for every $l \in \{1,\ldots,L\}$, such that the statement follows from Proposition~\ref{prop:resnet_nonaugmented}.
	
	In the case $\Phi \in \textup{RN}^1_{\eps,\delta,\sigma,\textup{N}}(\mathcal{X},\R)$, we have $\partial_{h_{l-1}}f_l(h_{l-1},\theta_l) =  \widetilde{W}_{l}\sigma'_{l}(a_{l})W_{l}$ with $a_l \coloneqq W_l h_{l-1} + b_l$, which is the explicit form of the Jacobian matrix determined in Proposition~\ref{prop:resnet_gradient}. As for two matrices $A\in \R^{a\times b}$ and $B \in \R^{b \times a}$, the matrix products $AB \in \R^{a \times a}$ and $BA \in \R^{b \times b}$ have the same non-zero eigenvalues (cf.~\cite{Horn2012}), we can estimate
	\begin{align*}
		&\max_{i \in \{1,\ldots,\nhid\}} \abs{ \lambda_i(\widetilde{W}_l\sigma'_{l}(a_{l})W_{l}) } = \max_{i \in \{1,\ldots,\nhid\}} \abs{ \lambda_i(W_{l} \widetilde{W}_l\sigma'_{l}(a_{l})) } \\ 
        \leq & \;\normm{W_{l} \widetilde{W}_l \sigma'_{l}(a_{l})}_2 \leq \normm{W_{l} \widetilde{W}_l}_2 \cdot \norm{\sigma'_{l}(a_{l})}_2 \leq \nu_{\max} \cdot K_{\sigma} <\frac{1}{\alpha}
	\end{align*}
	for all $x \in \mathcal{X}$. Here, we used that the maximum absolute eigenvalue is bounded by the sub-multiplicative Euclidean norm, that $\normm{W_{l} \widetilde{W}_l}_2 = \sigma_{\max}(W_{l} \widetilde{W}_l)$ and that $\norm{\sigma'_{l}(a_{l})}_2  = \norm{\sigma_l'}_{\infty,\R}$, as $\sigma'_{l}(a_{l})$ is a diagonal matrix. Since $-\frac{1}{\alpha}$ cannot be an eigenvalue of $\widetilde{W}_l\sigma'_{l}(a_{l})W_{l}$ for any $l \in \{1,\ldots,L\}$, the statement follows from Proposition~\ref{prop:resnet_nonaugmented}.
\end{proof}

\begin{remark}
    Under Assumption~\ref{ass:B_weights_bounded}, that all weight matrices are uniformly bounded, a uniform upper bound for the largest singular value of the matrix products $W_l \widetilde{W}_l $ can always be found. 
\end{remark}

\section{Case $\alpha \gg 1$: Close to Feed-Forward Neural Networks}
\label{sec:mlp}

In Section~\ref{sec:mlp_model}, we define general feed-forward neural networks (FNNs) and their canonical form, given by multilayer perceptrons (MLPs). In Section~\ref{sec:mlp_relationship}, we link ResNets to FNNs by taking the parameter limit $\alpha \coloneqq \frac{\delta}{\eps}\rightarrow \infty$, corresponding to a ResNet with dominating residual term. Additionally, we calculate the approximation error between FNNs and ResNets for large residual terms (large $\alpha$). Finally, in Section~\ref{sec:mlp_criticalpoints} we show that (in analogy to non-augmented MLPs) critical points cannot exist in non-augmented ResNets for $\alpha$ sufficiently large. 

\subsection{Feed-Forward Neural Networks}
\label{sec:mlp_model}

Classical feed-forward neural networks, such as MLPs, are structured in consecutive layers. In the following, we introduce the general FNN architecture we study, which, unlike ResNets, has no skip connection. We introduce FNNs with $L+2$ layers to align with the notation introduced for ResNets.
General FNNs are structured in layers $h_l \in \R^{n_l}$ defined by the iterative update rule
\begin{equation}
    h_l = \delta f_l(h_{l-1},\theta_l), \qquad l \in \{0,\ldots,L+1\},
\end{equation}
with input $h_{-1} \coloneqq x \in \mathcal{X}$, $\mathcal{X} \subset \R^{\nin}$, a (typically nonlinear) \emph{layer map} $f_l:\R^{n_{l-1}} \times \R^{p_{l}} \rightarrow \R^{n_l}$, and \emph{parameters} $\theta_{l} \in \Theta_l \subset \R^{p_{l}}$, where $ \Theta_l$ denotes the set of parameters of layer $l$. As in the definition of ResNets, we include the \emph{scaling parameter} $\delta>0$. In the case of MLPs, the layer map has the explicit form
\begin{equation} \label{eq:updaterule_mlp} 
	f_l(h_{l-1},\theta_{l}) \coloneqq \widetilde{W}_l\sigma_l(W_lh_{l-1}+b_l)+\tilde{b}_l = \widetilde{W}_l\sigma_l(a_l)+\tilde{b}_l
\end{equation}
where $\widetilde{W}_l \in \R^{n_{l} \times m_l}$, $W_l \in \R^{m_{l}\times n_{l-1}}$ are weight matrices, $b_l \in \R^{m_{l}}$, $\tilde{b}_l \in \R^{n_{l}}$ are biases, $h_{-1} \coloneqq x$ is the input data and $\sigma_l: \R\rightarrow \R$ is a component-wise applied activation function. As for ResNets, we write $\sigma_l$ for the component-wise extension $\sigma_l: \R^{m_{l}}\rightarrow \R^{m_{l}}$, where $m_l$ is the dimension of the pre-activated states $a_l \coloneqq W_l h_{l-1}+b_l$, $l \in \{0,\ldots,L+1\}$. The layer map in~\eqref{eq:updaterule_mlp} agrees with the residual functions of canonical ResNets defined in~\eqref{eq:resnet_typical_f} up to the layer-dependent dimension $n_l$. We denote the FNN input-output map by 
\begin{equation} \label{eq:mlp}
	\overline{\Phi}: \mathcal{X} \rightarrow \R^{\nout}, \; \mathcal{X} \subset \R^{\nin}, \qquad \overline{\Phi}(x) = h_{L+1}(x)
\end{equation}
with $\nout \coloneqq n_{L+1}$. To more easily distinguish general feed-forward architectures (including MLPs) from ResNet architectures, we denote them in this section by $\overline{\Phi}$ and $\Phi$, respectively. In contrast to ResNets, where the update rule~\eqref{eq:resnet_updaterule} requires constant layer width, the layer widths of FNNs can change. As the input and output dimensions $\nin \coloneqq n_{-1}$ and $\nout \coloneqq n_{L+1}$ are for FNNs not required to be equal, we do not include two additional transformations $\lambda$, $\tilde{\lambda}$ to the architecture as for ResNets in~\eqref{eq:resnet}. 

\begin{definition}[Feed-Forward Neural Network and Multilayer Perceptron] \label{def:feedforward}
	For $k \geq 0$ and $\mathcal{X}\subset \R^{\nin}$ open, the set $\FNN \subset C^k(\mathcal{X},\R^{\nout})$, $\delta>0$,  denotes all FNNs $\overline{\Phi}: \mathcal{X} \rightarrow \R^{\nout}$ as defined in~\eqref{eq:mlp} with layer map $f_l (\cdot,\theta_l)\in C^k(\R^{n_{l-1}},\R^{n_l})$ for each fixed $\theta_l \in \Theta_l \subset \R^{p_l}$, $l\in \{0,\ldots,L+1\}$. 

    In the canonical case, the set $\MLP \subset C^k(\mathcal{X},\R^{\nout})$, $\delta>0$, denotes all MLPs $\overline{\Phi}: \mathcal{X} \rightarrow \R^{\nout}$ with layer map as defined in~\eqref{eq:updaterule_mlp} satisfying $\sigma_l\in C^{k}(\R,\R)$ for $l \in \{0,\ldots,L+1\}$.
\end{definition}

In Section~\ref{sec:mlp_relationship}, we discuss how FNNs are related to ResNets with $\eps = 0$, and how MLPs connect to canonical ResNets with $\eps = 0$. \medskip

Depending on the dimensions $n_l$ of the layers $h_l \in \R^{n_l}$ and the dimensions~$m_l$ of the pre-activated states $a_l \in \R^{m_l}$, we define non-augmented FNNs and MLPs as follows.

\begin{definition}[Non-Augmented FNNs and MLPs] \label{def:mlp_classes}
	The classes $\FNN$ and $\MLP$, $k\geq 0$, $\mathcal{X}\subset \R^{\nin}$ open, are subdivided as follows:
	\begin{itemize}
		\item \emph{Non-Augmented FNN} $\overline{\Phi} \in \textup{FNN}^k_{\delta,\textup{N}}(\mathcal{X},\R^{\nout})$: it holds $n_{l-1} \geq n_{l} $ for $l \in \{0,\ldots, L+1\}$. 
        \item \emph{Non-Augmented MLP} $\overline{\Phi} \in \textup{MLP}^k_{\delta,\textup{N}}(\mathcal{X},\R^{\nout})$: it holds $n_{l-1} \geq m_{l} \geq n_{l} $ for $l \in \{0,\ldots, L+1\}$.
	\end{itemize}
\end{definition}

The concept of non-augmented FNNs and MLPs is visualized in Figure~\ref{fig:MLPclassification}.

\begin{figure}[h]
	\centering
     \hspace{3mm}
    \begin{subfigure}{0.38\textwidth}
		\centering
		\begin{overpic}[scale = 0.22,tics=10]
			{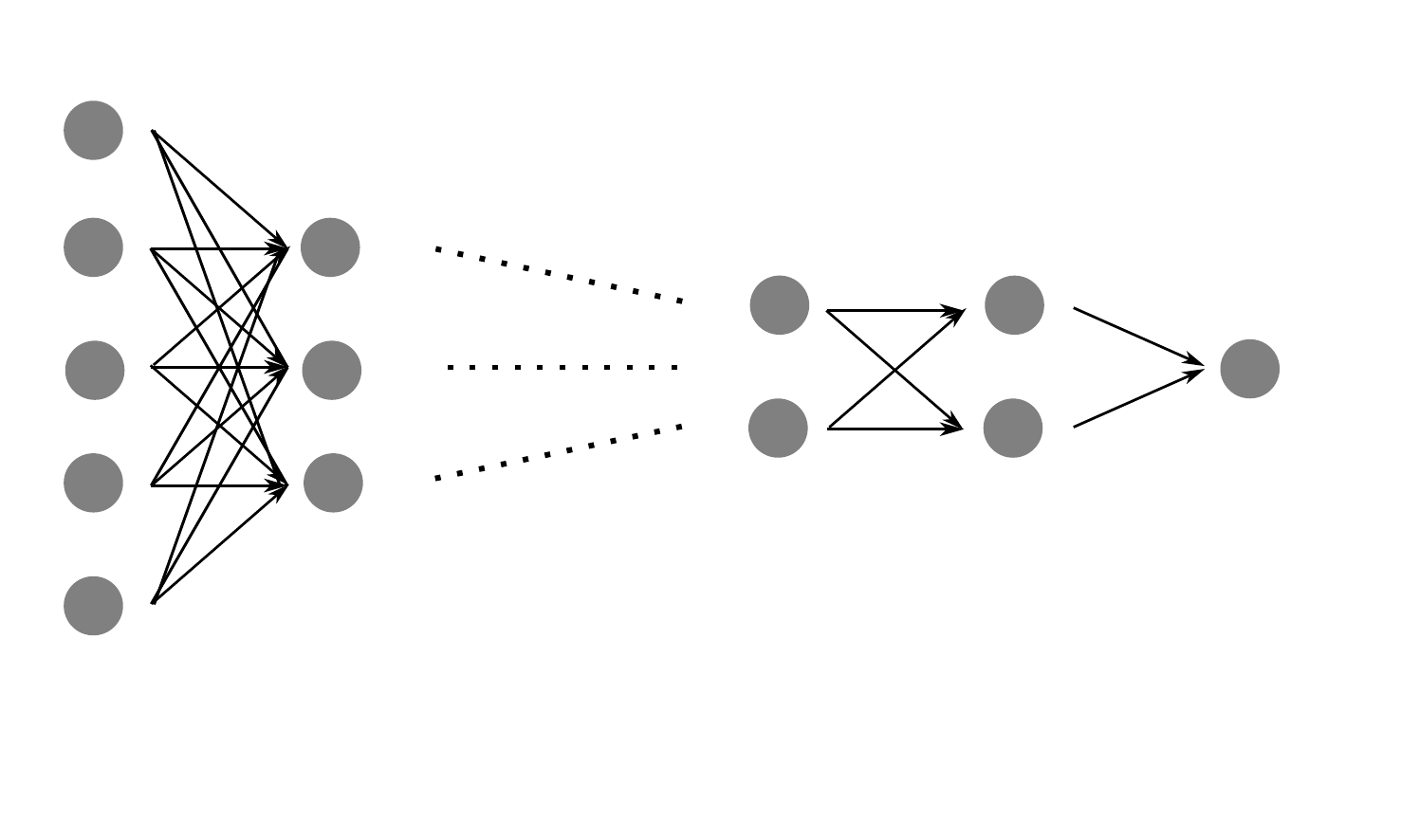}
			\put(5,2){\textcolor{black!70}{$x$}}
			\put(23,2){\textcolor{black!70}{$h_0$}}
            \put(51,2){\textcolor{black!70}{$h_{L-1}$}}
			\put(70,2){\textcolor{black!70}{$h_L$}}
			\put(86,2){\textcolor{black!70}{$h_{L+1}$}}
		\end{overpic}
		\caption{Example of a non-augmented FNN $\overline{\Phi} \in \textup{FNN}^k_{\delta,\textup{N}}(\R^5,\R)$ with layers $h_l \in \R^{n_l}$.}
		\label{fig:FNN_nonaug}
	\end{subfigure}
    \hspace{5mm}
	\begin{subfigure}{0.5\textwidth}
		\centering
		\begin{overpic}[scale = 0.22,tics=10]
			{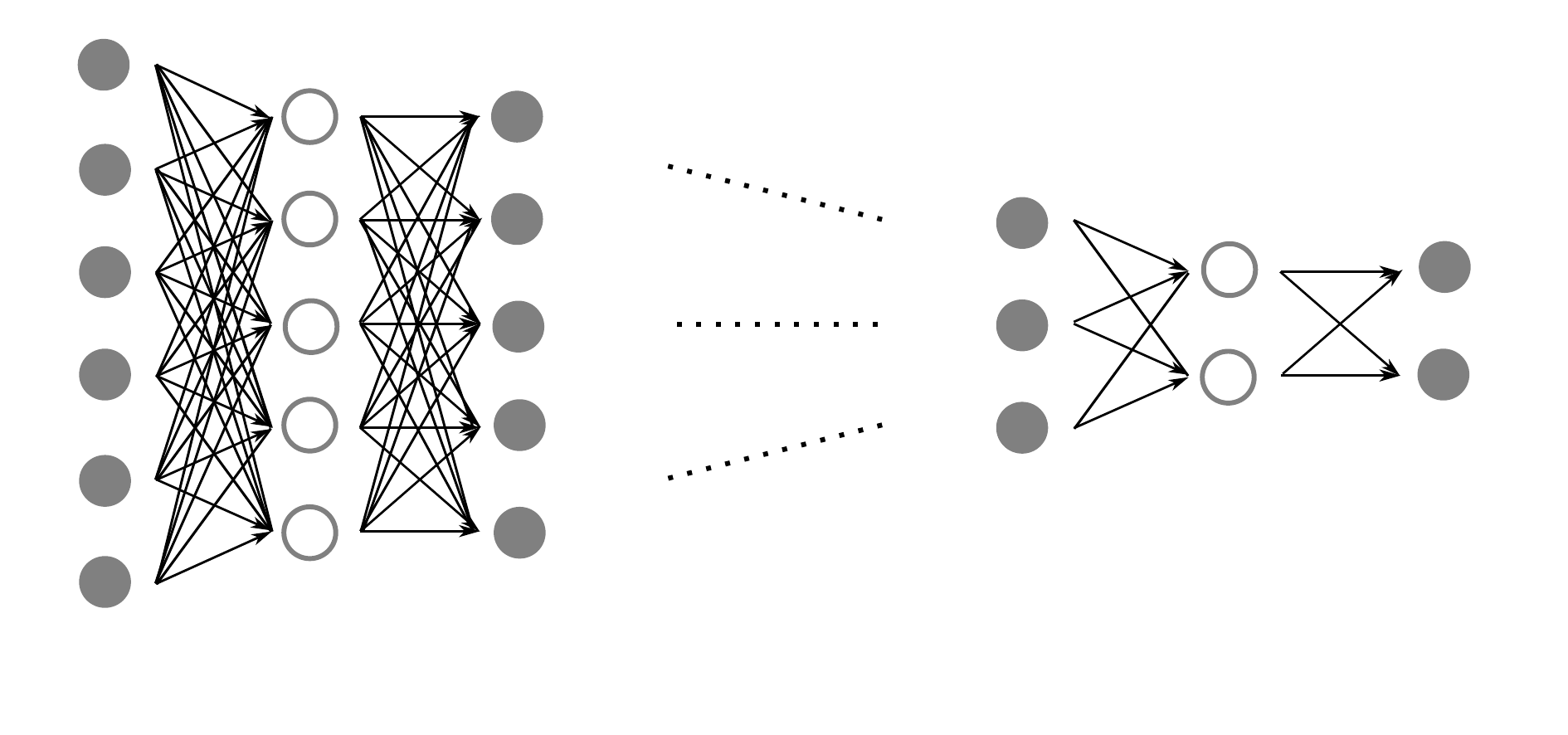}
			\put(5,2){\textcolor{black!70}{$x$}}
			\put(18,2){\textcolor{black!70}{$a_0$}}
			\put(31,2){\textcolor{black!70}{$h_0$}}
			\put(63,2){\textcolor{black!70}{$h_L$}}
			\put(75,2){\textcolor{black!70}{$a_{L+1}$}}
			\put(89,2){\textcolor{black!70}{$h_{L+1}$}}
		\end{overpic}
		\caption{Example of a non-augmented MLP $\overline{\Phi} \in \textup{MLP}^k_{\delta,\textup{N}}(\R^6,\R^2)$ with $h_l \in \R^{n_l}$ in gray and $a_l \in \R^{m_l}$ in white.}
		\label{fig:MLP_nonaug}
	\end{subfigure}
     \hspace{3mm}
	\caption{Visualization of the monotonically decreasing structure of non-augmented FNNs and MLPs.}
	\label{fig:MLPclassification}
\end{figure}

\begin{remark}
	The definition of a non-augmented FNN or MLP is more restrictive than that of a non-augmented ResNet: for MLPs, the dimensions of the input $x\in\R^{\nin}$, all layers $h_l \in \R^{n_l}$, and pre-activated states $a_l \in \R^{m_l}$ have to be monotonically decreasing. In the case of ResNets, it is only necessary that the input dimension $\nin$ is larger than or equal to the hidden dimension $\nhid$; the dimension $m_l$ of the pre-activated states is not relevant in our context.
\end{remark}

\begin{remark}
   Unlike ResNets, FNNs admit variable layer widths $n_l$. Architectures that violate the monotonicity requirements of Definition~\ref{def:mlp_classes} can be further classified as \emph{augmented}, \emph{bottleneck}, or \emph{degenerate}, as discussed in~\cite{kk2025}. Since our focus is on non-augmented architectures, we omit formal definitions of the alternative structures.
\end{remark}

		While FNNs allow for arbitrary layer maps, we now focus on the algebraic structure of MLPs. 

\begin{remark}
	The general update rule~\eqref{eq:updaterule_mlp} includes as special cases the following two typical MLP update rules: if $m_l = n_l$, $\widetilde{W}_l = \textup{Id}_{n_l}$ and $\tilde{b}_l = 0$, MLPs with an ``outer nonlinearity'' are obtained, i.e., $h_{l} = \delta \sigma_l(W_l h_{l-1}+b_l)$, and if $m_l = n_{l-1}$, $W_l = \textup{Id}_{n_{l-1}}$ and $b_l = 0$ for $l \in \{0,\ldots,L+1\}$, the update rule simplifies to MLPs with an ``inner nonlinearity'', i.e., $h_{l} = \delta (\widetilde{W}_l\sigma_l(h_{l-1})+\tilde{b}_l)$.
\end{remark}

\begin{remark}
    The general update rule~\eqref{eq:updaterule_mlp} includes two affine transformations that can always be reduced to a single one per layer without changing the input-output map: the transformation $\sigma_l(a_l) \mapsto a_{l+1}$ amounts to
	\begin{equation}\label{eq:activatedstates}
		a_{l+1} = W_{l+1} h_l + b_{l+1} = W_{l+1}\left(\delta\left(\widetilde{W}_l \sigma_l(a_l) + \tilde{b}_l\right)\right) + b_{l+1} = \delta W_{l+1}\widetilde{W}_l \sigma_l(a_l) + \delta W_{l+1}\tilde{b}_l + b_{l+1},
	\end{equation}
	which is a single affine linear transformation of the activated state $\sigma_l(a_l)$. Nevertheless, we choose the update rule~\eqref{eq:updaterule_mlp}, as it allows us to express the structure of $\overline{\Phi}$ with only full-rank weight matrices, see Theorem~\ref{th:mlp_normalform}.
\end{remark}

Given an MLP, the following theorem allows us to always suppose that Assumption~\ref{ass:B_weights_fullrank} holds, i.e., that all weight matrices have full rank. 

\begin{theorem}[MLP Normal Form  \cite{kk2025}] \label{th:mlp_normalform}
	Let $\overline{\Phi}\in \MLP$, $k \geq 0$, $\mathcal{X}\subset \R^{\nin}$ open, be an MLP, where the weight matrices $W_0$ and $\widetilde{W}_{L+1}$ have full rank. Then $\overline{\Phi}$ is equivalent to an MLP $\widetilde{\Phi}\in \textup{MLP}^k(\mathcal{X},\R^{\nout})$ with only full rank weight matrices, i.e., $\overline{\Phi}(x) = \widetilde{\Phi}(x)$ for all $x \in \mathcal{X}$. $\widetilde{\Phi}$ is called the \emph{normal form} of $\overline{\Phi}$.
\end{theorem}

\begin{proof}
	The result is a special case of~\cite[Theorem 3.10]{kk2025} under the additional assumption that $W_0$ and $\widetilde{W}_{L+1}$ have full rank.
\end{proof}

The assumption that $W_0$ and $\widetilde{W}_{L+1}$ have full rank is analogous to Assumption~\ref{ass:C_input_output} that the Jacobian matrices of the input and output transformation have full rank. The proof of \cite[Theorem 3.10]{kk2025} replaces the matrix product $W_{l+1}\widetilde{W}_l$ in~\eqref{eq:activatedstates} by a product of full rank matrices of possibly smaller dimensions. For ResNet architectures, it is not possible to define a normal form in the same way. 

\begin{remark}
    It is important to note that the classes (non-augmented, augmented, or bottleneck) of an MLP $\overline{\Phi} \in \textup{MLP}^k(\mathcal{X},\R^{\nout})$ and its normal form $\widetilde{\Phi}$ (cf.~Theorem~\ref{th:mlp_normalform}) are not necessarily the same. A non-augmented or augmented MLP with singular weight matrices can be equivalent to an MLP in normal form with a bottleneck. The MLP results of Section~\ref{sec:mlp_criticalpoints} are stated for MLPs in normal form, where all weight matrices have full rank and hence fulfill Assumption~\ref{ass:B_weights_fullrank}.
\end{remark}

\subsection{Relationship to ResNets}
\label{sec:mlp_relationship} 

In the following proposition, we show that given a ResNet $\Phi \in \textup{RN}_{0,\delta}^k(\mathcal{X},\R^{\nout})$, $\mathcal{X}\subset \R^{\nin}$ with skip parameter $\eps = 0$, there exists an FNN $\overline{\Phi} \in \FNN$, that has the same input-output map as~$\Phi$. Conversely, the statement also holds if the FNN's hidden layers have a constant width. In the previous section, we introduced FNNs with $L+2$ layers to maintain the analogy with ResNets, which have, in addition to the $L$ hidden layers, one input and one output transformation.

\begin{proposition}[Relationship between ResNets and FNNs] \label{prop:resnet_mlp}
	Let $\mathcal{X}\subset \R^{\nin}$ open, $k\geq 0$ and $\delta>0$. 
	\begin{enumerate}[label=(\alph*), font=\normalfont]
		\item \label{prop:resnet_mlp_a} For every ResNet $\Phi \in \textup{RN}_{0,\delta}^k(\mathcal{X},\R^{\nout})$, there exists an FNN $\overline{\Phi} \in \FNN$ with the same input-output map as $\Phi$, i.e., $\Phi(x) = \overline{\Phi}(x)$ for all $x \in \mathcal{X}$.
		\item \label{prop:resnet_mlp_b} For every FNN $\overline{\Phi} \in \FNN$, where all hidden layers $h_l$ have the same dimension $n_l = \nhid $ for $l \in \{0,1,\ldots,L\}$, there exists a ResNet $\Phi \in \textup{RN}_{0,\delta}^k(\mathcal{X},\R^{\nout})$ with the same input-output map, i.e., $\Phi(x) = \overline{\Phi}(x)$ for all $x \in \mathcal{X}$.
	\end{enumerate}
    If the input and output transformations $\lambda$, $\tilde{\lambda}$ have the typical form~\eqref{eq:lambda_typical}, the statements holds as a special case for canonical ResNets $\Phi \in \textup{RN}_{0,\delta,\sigma}^k(\mathcal{X},\R^{\nout})$ and MLPs $\overline{\Phi}\in\MLP$.
\end{proposition}

The proof of this proposition only relies on the translation of the notation between ResNets and FNNs. It is therefore proven in Appendix~\ref{app:proof_resnet_mlp} and included for completeness. Proposition~\ref{prop:resnet_mlp} shows that the ResNet and FNN dynamics coincide in the case $\eps = 0$, which relates to the limit $\alpha \coloneqq \frac{\delta}{\eps} \rightarrow \infty$, as $\eps \rightarrow 0$ for fixed $\delta>0$, as also discussed in Section~\ref{sec:resnet_gradient} and Section~\ref{sec:resnet_tunnel}. \medskip

In the following, we aim to quantify the distance between FNNs and ResNets with a small skip parameter $0< \eps < 1$. For this purpose, we first calculate the explicit input-output map of ResNets. Although Definition~\ref{def:resnet} assumes $\eps,\delta>0$, the following lemma includes the limiting cases $\eps = 0$ and $\delta = 0$.

\begin{lemma}[ResNet Input-Output Map]\label{lem:resnet_inputoutput}
	For a ResNet $\Phi \in \textup{RN}^0_{\eps,\delta}(\mathcal{X},\R^{\nout})$, $\mathcal{X}\subset \R^{\nin}$ open, with $\eps,\delta \geq 0$, for the $l$-th layer it holds that
    \begin{equation}\label{eq:inputoutput}
        h_l = \eps^l \lambda(x) + \delta \sum_{j = 1}^{l} \eps^{l-j}f_j(h_{j-1},\theta_j), \qquad l \in \{1,\ldots,L\}.
    \end{equation}
\end{lemma}

\begin{proof}
    The statement follows by induction on the update rule~\eqref{eq:resnet_updaterule}. For $l=1$, it holds $$h_1 = \eps h_0 + \delta f_1(h_0,\theta_1) = \eps \lambda(x) + \delta f_1(\lambda(x),\theta_1), $$ which matches~\eqref{eq:inputoutput}. For $l \in \{2,\ldots,L\}$, assume that formula~\eqref{eq:inputoutput} holds for $l-1$, then we conclude for layer $l$:
    \begin{align*}
        h_l &= \eps h_{l-1} + \delta f_l(h_{l-1}, \theta_l) \\
        &= \eps \left( \eps^{l-1} \lambda(x) + \delta \sum_{j = 1}^{l-1} \eps^{l-1-j}f_j(h_{j-1}, \theta_j)\right) + \delta f_l(h_{l-1}, \theta_l) \\
        &= \eps^l \lambda(x) + \delta \sum_{j = 1}^{l} \eps^{l-j}f_j(h_{j-1}, \theta_j). \qedhere
    \end{align*}
\end{proof}

\begin{remark} \label{rem:scaling_asymmetry}
    Lemma~\ref{lem:resnet_inputoutput} reveals a fundamental asymmetry in how the parameters $\eps$ and $\delta$ influence the ResNet input-output map. The residual parameter $\delta$ acts as a uniform global multiplier for every residual function $f_l$. In contrast, the skip parameter $\eps$ induces a depth-dependent exponential scaling: the contribution of the $l$-th hidden layer is scaled by $\eps^{L-l}$, and the input $\lambda(x)$ is scaled by $\eps^L$. 
    Consequently, for $0 \leq \eps < 1$, the contribution from earlier layers decays exponentially as it propagates toward the output.
    This phenomenon is illustrated for low-dimensional examples in Section~\ref{subsubsec:1d2l}.
\end{remark}

The following theorem shows that the global error of the input-output map between ResNets and FNNs scales linearly in the skip parameter $\eps$ asymptotically for $\eps \rightarrow 0$ and fixed $\delta>0$ and $L>0$. 

\begin{theorem}[Approximation Error between ResNets and FNNs] \label{th:mlp_resnet_error}   
	Let $0<\eps<1$ and $\delta>0$ and consider a ResNet $\Phi \in \textup{RN}^0_{\eps,\delta}(\mathcal{X},\R^{\nout})$, $\mathcal{X}\subset \R^{\nin}$ open, with update rule
	\begin{equation*}
		h_l = \eps h_{l-1} + \delta f_l(h_{l-1},\theta_l), \qquad l \in \{1,\ldots,L\},
	\end{equation*}
    parameters $\theta_l\in \Theta_l \subset \R^{p_l}$, input transformation $\lambda:\R^{\nin} \rightarrow \R^{\nhid}$ and output transformation $\tilde{\lambda}: \R^{\nhid} \rightarrow \R^{\nout}$.
	Furthermore, consider the corresponding FNN $\overline{\Phi} \in \textup{FNN}^0_{\delta}(\mathcal{X},\R^{\nout})$ with update rule
	\begin{equation*}
		h_l = \delta f_l(h_{l-1},\theta_l), \qquad l \in \{1,\ldots,L\},
	\end{equation*}
	and $h_0 = \lambda(x)$, $h_{L+1} = \tilde{\lambda}(h_L)$. Assume that:
	\begin{itemize}
		\item The residual functions $f_l$ are globally bounded, i.e., there exists a constant $S_f>0$, such that $\norm{f_l(h_{l-1},\theta_l)}_\infty \leq S_f$ for all $h_{l-1} \in \R^{n_{l-1}}$, $\theta_l \in \Theta_l$ and $l \in \{1,\ldots,L\}$. 
        \item The residual functions $f_l$ are globally Lipschitz continuous with Lipschitz constant $K_f$ for all $\theta_l\in\Theta_l$ and $l \in \{1,\ldots,L\}$.  
        \item The input transformation $\lambda$ is globally bounded by the constant $S_\lambda>0$, i.e., $\norm{\lambda}_{\infty,\mathcal{X}}\leq S_\lambda$.
        \item The output transformation $\tilde{\lambda}$ is globally Lipschitz continuous with Lipschitz constant $K_{\tilde{\lambda}}$.
	\end{itemize}
	Then it holds for the global approximation error between the ResNet $\Phi$ and the FNN $\overline{\Phi}$ that
    \begin{equation}\label{eq:error_resnet_mlp}
	   \norm{\Phi-\overline{\Phi}}_{\infty,\mathcal{X}} \leq \eps \cdot K_{\tilde{\lambda}} \left(  (\delta K_f)^{L-1} S_\lambda + (S_\lambda +\delta S_f)  \sum_{j = 0}^{L-2} (\delta K_f)^j\right)  + \mathcal{O}(\eps^2)
    \end{equation}
    as $\eps \rightarrow 0$.
\end{theorem}

\begin{proof}
	Let $x\in \mathcal{X}$ be arbitrary. To estimate $\normk{\Phi-\overline{\Phi}}_{\infty,\mathcal{X}}$, we first bound the distance between the hidden states of the ResNet, denoted by $h_l^{\textup{RN}}$, and the hidden states of the FNN, denoted by $h_l^{\textup{FNN}}$ for $l \in \{1,\ldots,L\}$. The pointwise error at layer $l$ is defined to be $\Delta_l \coloneqq \norm{h_l^{\textup{RN}}-h_l^{\textup{FNN}}}_\infty$. 
    For the first hidden layer it holds 
    \begin{equation*}
        \Delta_1 = \norm{h_1^{\textup{RN}}-h_1^{\textup{FNN}}}_\infty = \norm{\eps \lambda(x) + \delta f_1(\lambda(x),\theta_1) -\delta f_1(\lambda(x),\theta_1)}_\infty = \norm{\eps \lambda(x)}_\infty \leq \eps S_\lambda.
    \end{equation*}
    To determine $\Delta_l$ for $l \in \{2, \ldots, L\}$, we first estimate with Lemma~\ref{lem:resnet_inputoutput} :
    \begin{align*}
        \norm{h_l^{\textup{RN}}}_\infty &= \norm{\eps^l \lambda(x) + \delta \sum_{j = 1}^{l} \eps^{l-j}f_j(h_{j-1}^{\textup{RN}},\theta_j)}_\infty \leq \eps^l \norm{\lambda(x)}_\infty + \delta \sum_{j=1}^{l} \eps^{l-j} \norm{f_j(h_{j-1}^{\textup{RN}},\theta_j)}_\infty \\
        &\leq \eps^l S_\lambda  + \delta S_f  \sum_{j=0}^{l-1} \eps^j 
        \leq S_\lambda  + \frac{\delta S_f}{1-\eps} \eqqcolon H,
    \end{align*}
    where we used the bound of the geometric series for $0 < \eps < 1$. It follows for $l \in \{2,\ldots,L\}$:
    \begin{align*}
        \Delta_l &= \norm{\eps h_{l-1}^{\textup{RN}} + \delta f_l(h_{l-1}^{\textup{RN}}, \theta_l) - \delta f_l(h_{l-1}^{\textup{FNN}}, \theta_l)}_\infty \\
        &\leq \eps \norm{h_{l-1}^{\textup{RN}}}_\infty + \delta \norm{f_l(h_{l-1}^{\textup{RN}}, \theta_l) - f_l(h_{l-1}^{\textup{FNN}}, \theta_l)}_\infty \leq \eps H + \delta K_f \Delta_{l-1},
    \end{align*}
    where we used the Lipschitz continuity of the residual functions $f_l$. As the bound of $\Delta_l$ depends linearly on $\Delta_{l-1}$, we can estimate
    \begin{equation*}
        \Delta_L \leq \eps H + \delta K_f \Delta_{L-1} \leq \ldots \leq (\delta K_f)^{L-1}\Delta_1 + \eps H \sum_{j = 0}^{L-2} (\delta K_f)^j.
    \end{equation*}
    Finally, we use the Lipschitz continuity of the output transformation $\tilde{\lambda}$ to obtain
    \begin{align}
        \norm{\Phi - \overline{\Phi}}_{\infty,\mathcal{X}} &= \sup_{x \in \mathcal{X}}\norm{\tilde{\lambda}(h_L^{\textup{RN}}) - \tilde{\lambda}(h_L^{\textup{FNN}})}_\infty 
        \leq K_{\tilde{\lambda}} \Delta_L \nonumber \\
        &\leq \eps \cdot K_{\tilde{\lambda}}\left( (\delta K_f)^{L-1} S_\lambda +  \left(S_\lambda  + \frac{\delta S_f}{1-\eps}\right) \sum_{j = 0}^{L-2} (\delta K_f)^j\right),  \label{eq:explicit_bound} 
    \end{align}
    such that the statement follows after expanding $\frac{1}{1-\eps}$ in $\eps$.
\end{proof}

\begin{remark} \label{rem:scaling_linear}
    The asymptotic bound in Theorem~\ref{th:mlp_resnet_error} is sharp to leading order in $0<\eps<1$. An explicit, non-asymptotic upper bound valid for all $0<\eps<1$ is given by equation~\eqref{eq:explicit_bound}. Furthermore, the asymptotic bound indicates a regular perturbation problem~\cite{KuehnBook} as $\varepsilon\rightarrow 0$, which is reasonable considering the Lipschitz assumptions on the nonlinear terms. Yet, even if these Lipschitz assumptions would be dropped, one could attempt to find an asymptotic bound using techniques for singularly perturbed iterated maps~\cite{JelbartKuehn}. This would be technically an extremely challenging extension. 
\end{remark}

\begin{remark}
    The global approximation error~\eqref{eq:error_resnet_mlp} depends on the product $\delta K_f$ and the depth $L$. We observe that for $\delta K_f < 1$, the error remains bounded, even for very deep networks with $L \gg 1$, as the sum $\sum_{j=0}^{L-2} (\delta K_f)^j$ converges. In this case, the ResNet is an approximation of the FNN with stable approximation error with respect to the depth. In the case $\delta K_f > 1$, the global approximation error grows exponentially with depth $L$.  This highlights that while the approximation is $\mathcal{O}(\eps)$ for any fixed $L$, the error grows as the network becomes deeper unless the residual functions are sufficiently contractive.
\end{remark}

We now consider Theorem~\ref{th:mlp_resnet_error} for canonical ResNets with the typical residual function~\eqref{eq:resnet_typical_f}. The proof consists of explicitly computing the constants $S_f$, $K_f$, $S_{\lambda}$ and $K_{\tilde{\lambda}}$ and is given in Appendix~\ref{app:cor_proof_resnet_mlp}.

\begin{corollary}[Approximation Error between Canonical ResNets and MLPs] \label{cor:mlp_resnet_error}   
	Let $0<\eps<1$ and $\delta>0$ and consider a canonical ResNet $\Phi \in \textup{RN}^0_{\eps,\delta,\sigma}(\mathcal{X},\R^{\nout})$, $\mathcal{X}\subset \R^{\nin}$ open, with update rule
    \begin{equation*}\label{eq:resnet_updaterule_distance}
		h_l = \eps h_{l-1} + \delta \left(\widetilde{W}_l \sigma_l(W_l h_{l-1} + b_l) + \tilde{b}_l\right), \qquad l \in \{1,\ldots,L\},
	\end{equation*}
	as defined in~\eqref{eq:resnet_typical_f}.
	Let the input and output transformations have the typical form \eqref{eq:lambda_typical} and let
    $(W_l,\widetilde{W}_l,b_l,\tilde{b}_l)\in \Theta_l \subset \R^{\nhid\times m_l} \times \R^{m_l \times \nhid} \times \R^{m_l} \times \R^{\nhid}$ for $l \in \{0,\ldots,L+1\}$.
	Furthermore, consider the corresponding $\overline{\Phi} \in \textup{MLP}^0_{\delta}(\mathcal{X},\R^{\nout})$ with update rule
	\begin{equation*}
		h_l = \delta\left(\widetilde{W}_l \sigma_l(W_l h_{l-1} + b_l) + \tilde{b}_l \right), \qquad l \in \{1,\ldots,L\},
	\end{equation*}
	and $h_0 = \lambda(x)$, $h_{L+1} = \tilde{\lambda}(h_L)$. Assume that:
	\begin{itemize}
		\item The activation functions $\sigma_{l}\in C^0(\R,\R)$, $l\in\{1,\ldots,L+1\}$ fulfill Assumption~\ref{ass:A_activation_lipschitz}, i.e., all $\sigma_{l}$, $l\in\{1,\ldots,L+1\}$ are Lipschitz continuous with Lipschitz constant $K_\sigma$.
		\item The activation functions $\sigma_{l}\in C^0(\R,\R)$, $l\in \{0,\ldots,L\}$ fulfill Assumption~\ref{ass:A_activation_bounded}, i.e., there exists a constant $S_\sigma > 0$ such that $\norm{\sigma_l}_{\infty,\R} \leq S_\sigma$ for all $l \in \{0,\ldots,L\}$.
		\item  The weight matrices $W_l$, $\widetilde{W}_l$, and the biases $\tilde{b}_l$ fulfill Assumption~\ref{ass:B_weights_bounded} with respect to the max-norm, i.e., there exist constants $\omega_\infty, \widetilde{\omega}_\infty, \tilde{\beta}_\infty \geq 0$, such that $\norm{W_l}_\infty\leq \omega_\infty$, $\normm{\widetilde{W}_l}_\infty \leq \widetilde{\omega}_\infty $ and $\normm{\tilde{b}_l}_\infty \leq \tilde{\beta}_\infty$ for all $(W_l,\widetilde{W}_l,b_l,\tilde{b}_l)\in \Theta_l$ and $l\in \{0,\ldots,L+1\}$.
	\end{itemize}
	Then the global approximation error~\eqref{eq:error_resnet_mlp} with $S_f = S_\lambda = \widetilde{\omega}_\infty S_\sigma + \tilde{\beta}_\infty$ and  $K_f = K_{\tilde\lambda} =  \widetilde{\omega }_\infty  K_{\sigma} \omega_\infty$ is given by
    \begin{equation*}
		\norm{\Phi-\overline{\Phi}}_{\infty,\mathcal{X}} \leq \eps \cdot  \left(\widetilde{\omega}_\infty S_\sigma + \tilde{\beta}_\infty\right) \cdot \widetilde{\omega }_\infty  K_{\sigma} \omega_\infty \cdot\left( (\delta \widetilde{\omega }_\infty  K_{\sigma} \omega_\infty)^{L-1}  +  (1+\delta) \cdot \sum_{j = 0}^{L-2} (\delta \widetilde{\omega }_\infty  K_{\sigma} \omega_\infty)^j\right)+\mathcal{O}(\eps^2)
	\end{equation*}
    as $\eps \rightarrow 0$.
\end{corollary}

\subsection{Existence of Critical Points}
\label{sec:mlp_criticalpoints}

In this section, we characterize the existence of critical points in FNNs and ResNets in the parameter regime $\alpha \gg 1$. For that purpose, we first extend a result of~\cite{kk2025}, that non-augmented MLPs in normal form cannot have any critical points, to our setting of general layer maps.

\begin{theorem}[No Critical Points in Non-Augmented FNNs]\label{th:mlp_criticalpoints}
    Let $\overline{\Phi} \in \textup{FNN}^1_{\delta,\textup{N}}(\mathcal{X},\R)$, $\mathcal{X}\subset \R^{\nin}$ open, be a scalar non-augmented FNN with the following properties:
    \begin{itemize}
        \item All Jacobian matrices $\partial_{h_{l-1}}f_l(h_{l-1},\theta_l)\in\R^{n_l\times n_{l-1}}$ have full rank $n_{l}$ for all $h_{l-1}\in\R^{n_{l-1}}$, $\theta_l \in \Theta_l\subset \R^{p_l}$ and $l \in \{0,\ldots,L+1\}$.
    \end{itemize}
    Then $\overline{\Phi}$ cannot have any critical points, i.e., $\nabla_x \overline{\Phi}(x) \neq 0$ for all $x \in \mathcal{X}$.
    
	In the case that the feed-forward neural network is a non-augmented MLP $\overline{\Phi} \in \textup{MLP}^1_{\delta,\textup{N}}(\mathcal{X},\R)$ with $\mathcal{X}\subset \R^{\nin}$ open, the assumption can be replaced by the following conditions:
	\begin{itemize}
		\item The component-wise applied activation functions $\sigma_l\in C^1(\R,\R)$, $l \in \{0,\ldots,L+1\}$ fulfill Assumption~\ref{ass:A_activation_monotone}, i.e., they are strictly monotone and it holds $ \lvert \sigma_l'(y) \rvert > 0$ for every $y \in \R$ and all $l \in \{0,\ldots,L+1\}$.
		\item All weight matrices $(W_0, \widetilde{W}_0, \ldots, W_{L+1}, \widetilde{W}_{L+1})\in \Theta \subset \R^p$, have full rank, cf.~Assumption~\ref{ass:B_weights_fullrank}.
	\end{itemize}
\end{theorem}

\begin{proof}
    For the first part of the statement, we calculate the FNN input gradient, given by
    \begin{equation} \label{eq:FNNgradient}
		\nabla_x \overline{\Phi}(x) = \Bigl[ \delta \cdot \partial_{h_{L}} f_{L+1}(h_{L},\theta_{L+1}) \cdots \delta \cdot\partial_{h_{0}} f_1(h_{0},\theta_1) \cdot \delta\cdot \partial_xf_0(x,\theta_0) \Bigr]^\top \quad \in \R^{\nin}.
	\end{equation}
    As the FNN $\overline{\Phi}$ is scalar and non-augmented, it holds $\nin = n_{-1} \geq  \ldots \geq \nout = n_{L+1} = 1$, such that the dimensions in the matrix product~\eqref{eq:FNNgradient} are monotonically decreasing. As all Jacobian matrices have full rank, the gradient $\nabla_x \overline{\Phi}(x)$ always has full rank~$1$, uniformly in $x \in \mathcal{X}$.

    In the case of non-augmented MLPs, the theorem follows directly from~\cite[Theorem~3.16]{kk2025}: the definition of MLPs in~\cite{kk2025} only differs by the number of layers and the standing assumption on the strict monotonicity of the activation functions, which we assume here additionally. 
\end{proof}

In the following theorem, we show for non-augmented ResNets $\Phi \in \textup{RN}^1_{\eps,\delta,\textup{N}}(\mathcal{X},\R)$, $\mathcal{X}\subset \R^{\nin}$ open, that the absence of critical points persists for $\alpha \coloneqq \frac{\delta}{\eps}$ sufficiently large. In the case of canonical ResNets, we provide an explicit lower bound on the ratio~$\alpha$ above which $\Phi$ cannot have any critical points. 

\begin{theorem}[No Critical Points in Non-Augmented ResNets with $\alpha \gg 1$] \label{th:resnet_criticalpoints_alphalarge}
	Consider a scalar non-augmented ResNet $\Phi \in \textup{RN}^1_{\eps,\delta,\textup{N}}(\mathcal{X},\R)$, $\mathcal{X}\subset \R^{\nin}$ open, with the following properties:
	\begin{itemize}
		\item The input and output transformations $\lambda$ and $\tilde{\lambda}$ fulfill Assumption~\ref{ass:C_input_output}.
		\item Every residual function $f_l(\cdot,\theta_l)\in C^1(\R^{\nhid},\R^{\nhid})$ fulfills the lower Lipschitz condition
		\begin{equation*}
			\norm{f_l(y_1,\theta_l)-f_l(y_2,\theta_l)}_2\geq k_f \norm{y_1-y_2}_2
		\end{equation*}
		for some $k_f >0$ for all $y_1,y_2 \in\R^{\nhid}$, $l \in \{1,\ldots,L\}$ and all $\theta_l \in \Theta_l \subset \R^{p_l}$.
	\end{itemize}
	Then, if $\alpha \coloneqq \frac{\delta}{\eps} > \frac{1}{k_{f} }$,  $\Phi$ cannot have any critical points, i.e., $\nabla_x \Phi(x) \neq 0$ for all $x \in \mathcal{X}$. If the ResNet $\Phi \in \textup{RN}^1_{\eps,\delta,\sigma,\textup{N}}(\mathcal{X},\R)$ is canonical, the assumption on the residual functions $f_l$ can be replaced by the following:
	\begin{itemize}
		\item The activation functions $\sigma_l\in C^1(\R,\R)$ fulfill Assumption~\ref{ass:A_activation_monotone}, i.e., they are strictly monotone and it holds $ \abs{\sigma_l'(y)} > 0$ for every $y \in \R$ and all $l \in \{1,\ldots,L\}$. Additionally we assume that there exists a constant $k_\sigma >0$ such that $\inf_{x \in \mathcal{X}} \vert \sigma_l'([a_l]_i)\vert \geq k_\sigma$ with $a_l \coloneqq W_l h_{l-1} + b_l$ and $h_0 = x$ for all $\theta_l = (W_l,\widetilde{W}_l,b_l,\tilde{b}_l)\in \Theta_l$ with $l \in \{1,\ldots,L\}$, $i\in\{1,\ldots,m_l\}$.
		\item The smallest singular value of the matrix products $W_l\widetilde{W}_{l}$ is uniformly bounded from below, i.e., there exists a constant $\nu_{\min}>0$ such that  $\sigma_{\min}(W_l\widetilde{W}_{l}) \geq \nu_{\min}$ for all $\theta_l = (W_l,\widetilde{W}_l,b_l,\tilde{b}_l)\in \Theta_l$ and $l \in \{1,\ldots,L\}$.
	\end{itemize}
	Then, if $\alpha > \frac{1}{\nu_{\min}\cdot k_\sigma}$, $\Phi$ cannot have any critical points.
\end{theorem}

\begin{proof}
	We aim to prove the theorem by applying Proposition~\ref{prop:resnet_nonaugmented}, which requires showing that for $\alpha > \frac{1}{k_f}$, it holds that $- \frac{1}{\alpha}$ is not an eigenvalue of the Jacobian matrix $\partial_{h_{l-1}}f_l(h_{l-1},\theta_l)$ for all $l \in \{1,\ldots,L\}$, $h_{l-1} \in \R^{\nhid}$, and $\theta_l \in \Theta_l \subset \R^{p_l}$. By the assumed lower Lipschitz condition on the residual functions $f_l$, it holds for $\mu>0$ and $v \in \R^{\nhid}$:
	\begin{equation} \label{eq:directional_derivative}
		\frac{\norm{f_l(h_{l-1} + \mu \cdot v,\theta_l) - f_l(h_{l-1},\theta_l)}_2}{\abs{\mu}} \geq \frac{k_f \cdot \norm{\mu \cdot v}_2}{\abs{\mu}} = k_f \cdot \norm{v}_2.
	\end{equation}
	As $f_l(\cdot,\theta_l)\in C^1(\R^{\nhid},\R^{\nhid})$, we can take the limit $\mu \rightarrow 0$ in~\eqref{eq:directional_derivative}, which yields
	\begin{equation} \label{eq:directional_derivative_2}
		\norm{\partial_{h_{l-1}}f_l(h_{l-1},\theta_l) \cdot v}_2 
        \geq k_f \cdot \norm{v}_2.
	\end{equation}
	By the min-max characterization of singular values (cf.~\cite{Horn2012}), it holds for a matrix $A \in \R^{a \times b}$ that
	\begin{equation}\label{eq:singularvalues}
		\sigma_{\min}(A) =  \min_{v \in \R^{b} \setminus \{0\}} \frac{\norm{Av}_2}{\norm{v}_2}.
	\end{equation}
	Hence, we can conclude for the smallest singular value of the Jacobian matrix  $\partial_{h_{l-1}}f_l(h_{l-1},\theta_l)$: 
	\begin{equation*}
		\sigma_{\min}(\partial_{h_{l-1}}f_l(h_{l-1},\theta_l)) =  \min_{v \in \R^{\nhid},v\neq 0} \frac{\norm{	\partial_{h_{l-1}}f_l(h_{l-1},\theta_l) \cdot v}_2}{\norm{v}_2} \geq k_f.
	\end{equation*}
	As the smallest singular value is a lower bound for the absolute value of all eigenvalues (cf.~\cite{Horn2012}), it holds  
	\begin{equation*}
		\min_{i \in \{1,\ldots,\nhid\}} \vert \lambda_i(\partial_{h_{l-1}}f_l(h_{l-1},\theta_l)) \vert  \geq  \sigma_{\min}(\partial_{h_{l-1}}f_l(h_{l-1},\theta_l)) \geq k_f > \frac{1}{\alpha}.
	\end{equation*}
	Consequently, $-\frac{1}{\alpha}$ cannot be an eigenvalue of $\partial_{h_{l-1}}f_l(h_{l-1},\theta_l)$ for any $l \in \{1,\ldots,L\}$, such that the first statement follows from Proposition~\ref{prop:resnet_nonaugmented}.
	
	In the case $\Phi \in \textup{RN}^1_{\eps,\delta,\sigma,\textup{N}}(\mathcal{X},\R)$, we have $\partial_{h_{l-1}}f_l(h_{l-1},\theta_l) =  \widetilde{W}_{l}\sigma'_{l}(a_{l})W_{l}$ with $a_l \coloneqq W_l h_{l-1} + b_l$, which is the explicit form of the Jacobian matrix determined in Proposition~\ref{prop:resnet_gradient}. As for two matrices $A\in \R^{a\times b}$ and $B \in \R^{b \times a}$, the matrix products $AB \in \R^{a \times a}$ and $BA \in \R^{b \times b}$ have the same non-zero eigenvalues (cf.~\cite{Horn2012}), it holds
    \begin{equation*}
         \min_{i: \lambda_i \neq 0} \vert \lambda_i(\widetilde{W}_l\sigma'_{l}(a_{l})W_l) \vert = \min_{i: \lambda_i \neq 0} \vert \lambda_i(W_l\widetilde{W}_l\sigma'_{l}(a_{l})) \vert.
    \end{equation*}
    Let $\lambda^*$ be an arbitrary non-zero eigenvalue of the Jacobian $\widetilde{W}_l\sigma'_l(a_l)W_l$, so it is also  a non-zero eigenvalue of the matrix product $W_l\widetilde{W}_l\sigma'_l(a_l)$. Let $v^* \neq 0$ be the corresponding eigenvector such that
    \begin{equation*}
        W_l\widetilde{W}_l\sigma'_l(a_l)v^* = \lambda^*v^*.
    \end{equation*}
    Taking the Euclidean norm on both sides yields
    \begin{equation*}
        \abs{\lambda^*}\cdot \norm{v^*}_2 = \normm{W_l\widetilde{W}_l\sigma'_l(a_l)v^*}_2 \geq \sigma_{\min}(W_l\widetilde{W}_l) \norm{\sigma'_l(a_l)v^*}_2 \geq \nu_{\min} \cdot k_\sigma \cdot \norm{v^*}_2,
    \end{equation*}
    where we used the minimal singular value characterization from~\eqref{eq:singularvalues}, the given assumptions on $\sigma_{\min}(W_l\widetilde{W}_l)$, the fact that $\sigma'_l(a_l)$ is a diagonal matrix and the assumed lower bound on the activation functions. Dividing by $\norm{v^*}_2 >0$ results in
    \begin{equation*}
        \abs{\lambda^*} \geq \nu_{\min} \cdot k_\sigma > \frac{1}{\alpha}.
    \end{equation*}
    This shows that any non-zero eigenvalue of the Jacobian matrix $\widetilde{W}_l\sigma'_{l}(a_{l})W_l$ has an absolute value strictly greater than $\frac{1}{\alpha}$. Since $-\frac{1}{\alpha} \neq 0$, potential zero eigenvalues do not play a role in the relevant analysis. Consequently, $-\frac{1}{\alpha}$ cannot be an eigenvalue of $\partial_{h_{l-1}}f_l(h_{l-1},\theta_l)$ for any $l \in \{1,\ldots,L\}$, such that the second statement follows from Proposition~\ref{prop:resnet_nonaugmented}.
\end{proof}

\begin{remark}
    The assumption for canonical ResNets of Theorem~\ref{th:resnet_criticalpoints_alphalarge} that there exists a constant $k_\sigma >0$ with $\inf_{x \in \mathcal{X}} \vert \sigma_l'([a_l]_i)\vert \geq k_\sigma$ with $a_l \coloneqq W_l h_{l-1} + b_l$ and $h_0 = x$ for all $l \in \{0,\ldots,L\}$ and $i\in\{1,\ldots,m_l\}$ is standard in applications. Indeed, it is satisfied for typical activation functions such as $\tanh$ or sigmoidal functions whenever the input domain $\mathcal{X}\subset \R^{\nin}$ and all parameters $\theta \in \Theta = \Theta_1 \times \cdots \times \Theta_L$ are bounded (cf.~Assumption~\ref{ass:B_weights_bounded}). If the domain $\mathcal{X}$ is unbounded, no such $k_\sigma > 0$ exists, since $\vert \sigma_l'(y)\vert \to 0$ as $\vert y \vert \to \infty$ for these activation functions.
\end{remark}

\begin{remark}
    Assumption~\ref{ass:B_weights_fullrank} guarantees that for a given parameter regime $\Theta$ all weight matrices contained in $\Theta$ have full rank and hence strictly positive singular values. However, this does not prevent the infimum of all singular values of $\Theta$ from being zero. Theorem~\ref{th:resnet_criticalpoints_alphalarge} requires the stronger $\Theta$-uniform lower bound for all matrix products $W_l \widetilde{W}_{l}$.
    
    The assumption in Theorem~\ref{th:resnet_criticalpoints_alphalarge} that the smallest singular value of $W_l \widetilde{W}_l$ is uniformly bounded from below imposes an architectural constraint: Because $W_l \in \R^{m_l \times \nhid}$ and $\widetilde{W}_l \in \R^{\nhid \times m_l}$, their product is an $m_l \times m_l$ matrix. By the properties of matrix rank, it holds $\operatorname{rank}(W_l \widetilde{W}_l) \leq \min(m_l, \nhid)$. For $\sigma_{\min}(W_l \widetilde{W}_l)$ to be strictly positive, the matrix must have full rank $m_l$, which implies $\nhid \geq m_l$ for all $l \in \{1,\ldots,L\}$.
    
    This restriction aligns with the definition of non-augmented architectures: In classical MLPs, non-augmented networks are restricted to monotonically decreasing or constant layer dimensions (cf.~Definition~\ref{def:mlp_classes}). Similarly, requiring $\nhid \geq m_l$ ensures that no intermediate augmentations of pre-activated states $a_l \in \R^{m_l}$ exist. The case $m_l> \nhid$ is mathematically prohibited for $\alpha \gg 1$ because the resulting rank-deficiency of $W_l \widetilde{W}_l$ would introduce zero eigenvalues, which are close to the critical eigenvalue $-\frac{1}{\alpha}$. It is worth noting, however, that the dimension constraint is an artifact of the usage of uniform singular value bounds.
\end{remark}

\section{Examples}
 \label{sec:numerics}

We start with a thorough analysis of the simplest one-dimensional ResNets, as they illustrate the embedding restrictions present for the various regimes of $\alpha$ in their fundamental form. This has two purposes. First, to present the application of the derived expressivity results of Section~\ref{sec:resnets} to Section~\ref{sec:mlp} to explicit examples. Second, to discuss the impact of the parameter initialization and training on the embedding restrictions. We further implement two-dimensional classification tasks that illustrate the connection of the approximation constraints to the ``tunnel effect'' proven in Theorem~\ref{thm:constraints_nd}. The code used to generate the models and resulting plots of this section can be found at \url{https://github.com/twoehrer/Narrow_ResNet_Constraints.git}.

\subsection{One-Dimensional ResNets}\label{subsec:1dexamples}
For this entire subsection, our goal is to use one-dimensional ResNets to approximate the function $f(x) = x^2$ for $x\in (-1,1)$ as it has a critical point at $x = 0$. \medskip

Before analyzing this one-dimensional setting in detail, let us state the implementations' \emph{numerical specifications}: The models are trained on 300 uniformly distributed points, $x_i \sim \mathcal{U}(-1,1)$, with labels $y_i = x_i^2$ for $i = 1,\ldots, 300$. A mean-squared error loss is used to train the networks to approximate the target function. We use standard Xavier initialization (a standard normal distribution in the one-dimensional case) for the model weights. The model parameters are optimized via the stochastic gradient descent based Adam algorithm with learning rate 0.01. 

\subsubsection{One-Layer ResNets}
  Consider a one-layer ResNet $\Phi \in \textup{RN}^1_{\eps,\delta,\sigma}((-1,1),\R)$ from Definition~\ref{def:resnet} of the canonical form
    \begin{equation}
        \Phi(x) = \widetilde{W}_2 \big[ \eps x + \delta \widetilde{W}_1\tanh(W_1x+b_1) + \bt_1 \big]+\tilde{b}_2 \label{eq:onelayermodel}
    \end{equation}
    with scalars $\Wt_2,\Wt_1, W_1, b_1, \bt_1, \bt_2 \in \R$, and activation function $\sigma = \tanh$ which is monotone, bounded, and smooth.
    This corresponds to one layer $L=1$ in \eqref{eq:resnet}, with input transformation $\lambda(x) = x$, and an affine output transformation $\tilde{\lambda}(y) = \widetilde{W}_2y + \tilde{b}_2$. As the input space is not augmented, it is classified as $\Phi \in \textup{RN}^1_{\eps,\delta, \sigma, \textup{N}}((-1,1),\R)$ according to Definition~\ref{def:resnet_classes}.

Now, we examine the embedding restrictions within the model's parameter space for a fixed ratio $\alpha := \frac\delta\epsilon$. From an implementation standpoint, this determines whether a chosen architecture and parameter regime are able to efficiently approximate the target function.\\

We start with the two limit cases:
\begin{itemize}
  \item \emph{Case $\alpha = 0$, $\delta = 0$:} In this case \eqref{eq:onelayermodel} is affine linear and it follows that for all parameters no non-degenerate critical point can exist. I.e.\ critical points only exists for the trivial case $\Phi \equiv 0$ which implies every point $x$ is a degenerate critical point ($\Phi''(x) = 0$).
  
    \item  \emph{Case $\alpha = \infty$, $\eps = 0$; Figure~\ref{fig:1d1lresnet}\subref{subfig:eps0}:} Setting $\eps = 0$ leads to $\Phi\in \textup{MLP}^1_{\delta, \textup{N}}((-1,1),\R)$. As the input derivative is $\Phi'(x) = \Wt_2\delta \Wt_1 \tanh'(W_1x + b_1) W_1$ and $|\tanh'|> 0$, a critical point only exists for the trivial MLP $\Phi(x) \equiv 0$ and hence for all model parameters no non-degenerate critical point can exist, in accordance with Theorem~\ref{th:mlp_criticalpoints}. See Figure~\ref{fig:1d1lresnet}\subref{subfig:eps0} for a trained approximation which fails to globally approximate the target function well due to the inability to express critical points. We refer to the detailed analysis of the resulting restrictions of Theorem~\ref{thm:constraints_1d}.
   \end{itemize}
    
    For the remaining cases, let us calculate the input derivative of \eqref{eq:onelayermodel} explicitly as
    \begin{equation}\label{eq:phiprime}
        \Phi'(x) = \Wt_2\delta\Big[\frac1\alpha + \Wt_1 \tanh'(W_1 x + b_1) W_1\Big], \quad x\in (-1,1),
    \end{equation}
    and assume $\Wt_2, \Wt_1, W_1 \neq 0$ as well as $\eps, \delta, \alpha > 0$. Then it directly follows
    that (cf. Proposition~\ref{prop:resnet_nonaugmented})
    \begin{equation}\label{eq:directcrit}
    \Phi'(x) \neq 0 \iff \alpha \neq \frac{1}{-\Wt_1 W_1 \tanh'(W_1 x + b_1)}.
    \end{equation}
For the following case distinction, denote the model parameter set of the ResNets considered as 
\begin{equation}
    \Theta := \{(\Wt_2,\Wt_1, W_1, b_1, \bt_1, \bt_2) \in \R^6 \mid \Wt_2\neq 0, \Wt_1\neq 0, W_1 \neq 0  \}.
\end{equation}

\begin{itemize}
    \item \emph{Case $0 < \alpha \ll 1$:} In this case, Theorem~\ref{th:resnet_criticalpoints_alphasmall} yields that no critical point exists for $\Phi$ provided  $0 < \alpha < \frac{1}{ |\Wt_1 W_1|}$ where $K_\sigma = 1$ since $\tanh' \leq 1$.
    As $\tanh' > 0$, \eqref{eq:directcrit} additionally tells us that no critical point exists if $\Wt_1 W_1 > 0$ is satisfied. In total, no critical point of \eqref{eq:onelayermodel} can exist if $\Wt_1 W_1 > 0$ or $\frac{-1}{\alpha}<\Wt_1 W_1 < 0$. For fixed $0 < \alpha \ll 1$ it holds that all ResNets $\Phi$  do not have critical points if the model parameters satisfy
    \begin{equation}\label{eq:thetaalphasmall}
     \theta \in \Theta^{\alpha \ll 1} := \Big\{ \theta \in \Theta \mid \Wt_1 W_1  > -\frac{1}{\alpha} \Big\}. 
    \end{equation}

    \item \emph{Case $\alpha \gg 1$:} For $\alpha$ large enough, Theorem~\ref{th:resnet_criticalpoints_alphalarge} yields that $\Phi$ is not able to have a critical point as long as $\frac{1}{\alpha k_\sigma} <|\Wt_1 W_1| $ with $k_\sigma:=\inf_{x\in (-1,1)} \tanh'(W_1x + b_1)  = \tanh'(|W_1| + b_1)> 0$. The direct calculation in \eqref{eq:directcrit} tells us more precisely that no critical point exists if $\Wt_1 W_1 > 0$ or $\Wt_1W_1 < -\frac{1}{\alpha k_\sigma} < 0$. So far, this condition is formulated for fixed $\Wt_1W_1$ as $k_\sigma$ is model parameter-dependent. As $k_\sigma$ satisfies $k_\sigma(z) \to 0$ for $|z| \to \infty$, we need to bound the parameters to obtain a uniform estimate for the whole parameter regime:
    For sufficiently large $\alpha \gg 1$ it holds that all ResNets $\Phi$  do not have critical points if  the model parameters satisfy
    \begin{equation}\label{eq:thetaalphalarge}
        \theta \in \Theta^{\alpha \gg 1} := \Big\{ \theta \in \Theta :\, |W_1| < \omega_\infty, |b_1| < \beta_\infty, \Wt_1 W_1 \in \Big(-\infty, -\frac{1}{\alpha k_{\omega_\infty, \beta_\infty}}\Big)\cup (0, +\infty)\Big\},
    \end{equation}
    where the uniform bound is defined as
    \begin{equation}\label{eq:lowerk}
    k_{\omega_\infty, \beta_\infty}:= \inf_{|W_1| \le \omega_\infty, |b_1|\le \beta_\infty} \inf_{x\in(-1,1)} \tanh'(W_1x + b_1) = \tanh'(\omega_\infty + \beta_\infty).
    \end{equation}

    \item \emph{Case $\alpha = 1$:} In this case, we are ``in-between'' the regimes of $\alpha$ the previous sections are mainly concerned with. We analyze it, as \eqref{eq:directcrit} provides an explicit condition on the model parameters for a critical point. 
For any fixed $x\in (-1,1)$, the critical point condition
    \begin{equation}\label{eq:critcond1l}
        -\Wt_1 W_1 \tanh'(W_1x + b_1) = 1.
    \end{equation}
    can be satisfied by choosing $\Wt_1 W_1 \leq  -1$ with $\Wt_1 \neq 0$ arbitrary, $W_1 = \frac{-1}{\Wt_1}$ and $b_1 = \frac{-1}{\Wt_1 x}$ if $x\neq0$ and $b_1 = 0$ otherwise. Hence, ResNets $\Phi$ of the form \eqref{eq:onelayermodel} with $\alpha = 1$ cannot have critical points if the model parameters satisfy $\Wt_1W_1> -1$, which aligns with \eqref{eq:thetaalphasmall} for $\alpha = 1$. We can further exclude parameters by evoking Theorem~\ref{th:resnet_criticalpoints_alphalarge}. As for the case $\alpha \gg 1$, we require upper bounds on $W_1$ and $b_1$ to bound $|\Wt_1 W_1|$ uniformly from below. Then, in analogy to \eqref{eq:thetaalphalarge}, ResNets with $\alpha = 1$ do not have critical points if the model parameters satisfy
    \begin{equation}
        \theta \in \Theta^{\alpha = 1} := \Big\{ \theta \in \Theta :\, |W_1| < \omega_\infty, |b_1| < \beta_\infty, \Wt_1 W_1 \in \left(-\infty, -\frac{1}{k_{\omega_\infty, \beta_\infty}}\right)\cup (0, +\infty
        )\Big\},
    \end{equation}
    where $k_{\omega_\infty, \beta_\infty}$ is defined in \eqref{eq:lowerk}.

\begin{figure}
    \centering
    \begin{subfigure}{0.45\linewidth}
        \centering
        \includegraphics[width=\linewidth]{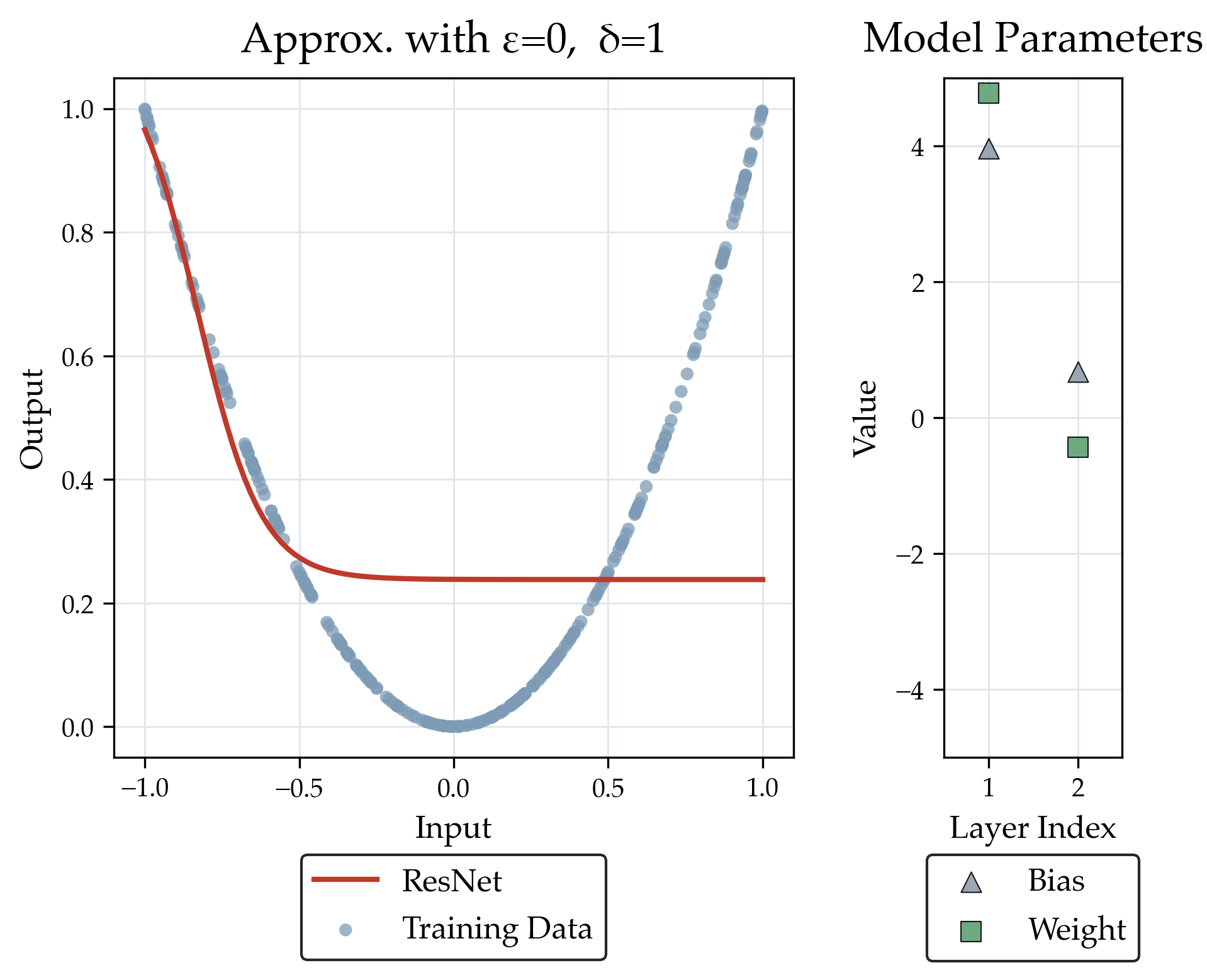}
        \caption{\small MLP, $\varepsilon = 0$, $\delta = 1$.}
        \label{subfig:eps0}
    \end{subfigure}
    \begin{subfigure}{0.45\linewidth}
        \centering
        \includegraphics[width=\linewidth]{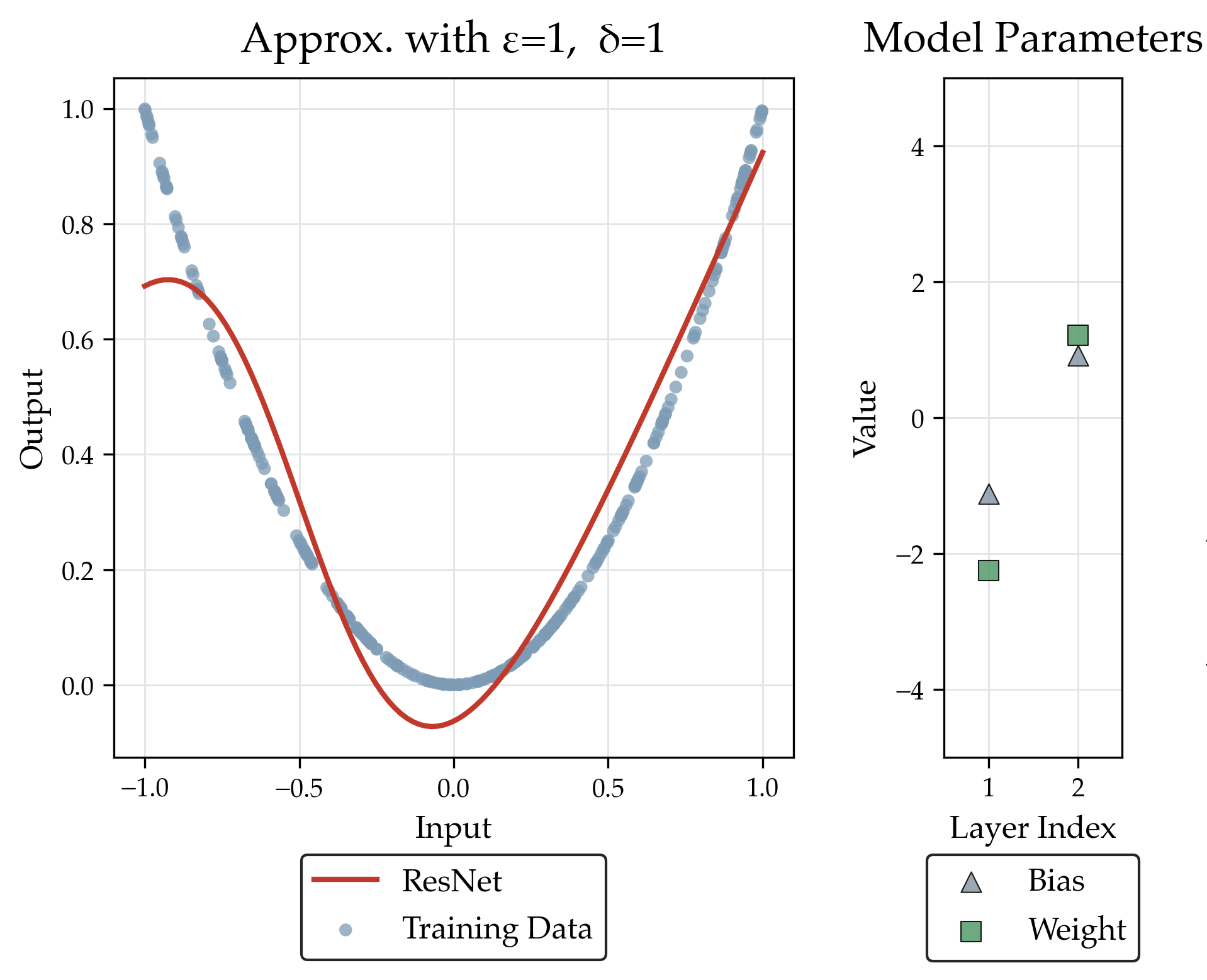}
        \caption{\small ResNet, $\varepsilon = \delta = 1$.}
        \label{subfig:eps1}
    \end{subfigure}
   
    \caption{\small Approximation of $f(x) = x^2$ by a one-dimensional one-layer ResNet \eqref{eq:onelayermodel}, $\Wt_1 = 1$, $\bt_1 = 0$ and the other weights and biases as plotted. \label{fig:1d1lresnet}}
\end{figure}

\item \emph{Case $\alpha = 1, \Wt_1 = 1$, $\bt_1 = 0$; Figure~\ref{fig:1d1lresnet}\subref{subfig:eps1}:} If we simplify the model and fix $\Wt_1 = 1$ and $\bt = 0$, Theorem~\ref{th:resnet_criticalpoints_alphalarge} does not provide any further model parameter bounds (as in the case $\alpha \gg 1$) beyond $W_1 > 0$ and $W_1 < -1$. We show that the embedding restriction criterion $1 > \frac{1}{ k_\sigma \nu_{\min}}$ of Theorem~\ref{th:resnet_criticalpoints_alphalarge} is never fulfilled. The reason is that, naturally, we require the lower bound on the weights ($\nu_{\min}$ in Theorem~\ref{th:resnet_criticalpoints_alphalarge}) to be smaller than the upper bound on the weights (necessary for a positive $k_{\omega_\infty, \beta_\infty}$). I.e., $\nu_{\min} \leq \omega_\infty$ needs to be fulfilled. It follows (with $k_{\omega_\infty, \beta_\infty}$ from \eqref{eq:lowerk}) that
    \begin{equation}
    \nu_{\min} k_{\omega_\infty, \beta_\infty}\leq  \omega_\infty k_{\omega_\infty, \beta_\infty}  = \omega_\infty \tanh'(\omega_\infty + \beta_\infty)  \leq 1
    \end{equation}
    for any $\omega_\infty \geq 1$, where we used the boundedness of $y\mapsto y \tanh'(y + b)$ on $\R$ uniform in $|b| \leq \beta_\infty$. Hence, we cannot find any parameter regime $\Theta$ satisfying the assumptions of Theorem~\ref{th:resnet_criticalpoints_alphalarge} such that the condition $1 > \frac{1}{ k_\sigma \nu_s}$ is fulfilled uniformly in $\theta\in \Theta$. This means, in this case, ResNets $\Phi$ do not have critical points if 
    \begin{equation}
     \theta \in \tilde{\Theta}^{\alpha = 1} := \{  \theta \in \Theta \mid \Wt_1 = 1, \bt_1 = 0 , W_1 \in ( -1, 0) \cup (0, +\infty)\}.
        \label{eq:thetaalphaone}
    \end{equation}
\end{itemize}
We have collected the conditions on the weights for each regime of $\alpha$ such that no critical points can be embedded. Figure~\ref{fig:1d1lresnet} compares the approximation performance depending on the embedding capabilities. Plotted are two trained models $\Phi$ of the form \eqref{eq:onelayermodel} with $\Wt_1 = 1$. Figure~\ref{fig:1d1lresnet}\subref{subfig:eps0} shows that, as MLPs cannot embed critical points, the trained network fails the approximation and confirms the result of Theorem~\ref{thm:constraints_1d}. Figure~\ref{fig:1d1lresnet}\subref{subfig:eps1} shows that the ResNet is able to express critical points and achieves adequate approximations of the target function in relation to the model's simple form.

  \subsubsection{Multi-Layer ResNets}
  We now extend the analysis from the one-layer case to networks of arbitrary depth $L > 0$.
 
 Consider $\Phi \in \textup{RN}^1_{\eps,\delta,\sigma,\textup{N}}((-1,1),\R)$ with $\eps,\delta > 0$.
    To simplify the model (as often done in applications), we only look at layers with ``outer'' nonlinearities
    \begin{equation} \label{eq:ResNetGeneral1D}
        h_{l} = \eps h_{l-1} + \delta \tanh(W_l h_{l-1} +b_l)   = \eps h_{l-1} + \delta \sigma(a_l),
    \end{equation}
    where $h_0\coloneqq x \in (-1,1)$, $a_l \coloneqq W_l h_{l-1} + b_l$ with  $W_l, b_l \in \R$ for all $l\in \{1,\ldots,L\}$. With the specific choice of $f_l$ from \eqref{eq:resnet_typical_f}, this corresponds to taking $\widetilde{W}_l = 1$ and $\tilde{b}_l = 0$ for all layers $l = 1, \ldots, L$ in \eqref{eq:resnet}. As in the one-layer example, we consider an identity input transformation and an affine output transformation, such that the input-output map is
    \begin{equation}\label{eq:inputtooutput}
        \Phi(x) = \widetilde{W}_{L+1} h_L(x) + \tilde{b}, \quad \widetilde{W}_{L+1},\tilde{b}_{L+1}  \in \R,\, x\in (-1,1).
    \end{equation}

The criterion that an input $x\in (-1,1)$ is not a critical point naturally extends from the one-layer case of \eqref{eq:directcrit} to $L$ layers as (cf.\ Proposition~\ref{prop:resnet_nonaugmented})
        \begin{equation}\label{eq:directcritmulti}
    \Phi'(x) \neq 0 \iff \alpha \neq \frac{1}{- W_l \tanh'(W_l h_{l-1}(x) + b_l)} \quad \text{for all } l\in \{1,\ldots, L\}.
    \end{equation}
Due to $\tanh'(y)\in (0,1]$ for all $y\in\R$, a condition on the parameters such that no critical point for $\Phi$ exists, is $ W_l > \frac{-1}{\alpha}$ for all $l \in \{1, \ldots, L\}$. 
    
    In fact, the remaining parameter regime can embed critical points. We see this by generalizing the explicit solution of \eqref{eq:critcond1l} from the one-layer case with $\alpha = 1$ above to $L$ layers and arbitrary fixed ratio $\alpha>0$.
\begin{lemma}\label{lem:b0condition}
      Let $\Phi$ be a non-augmented ResNet of type \eqref{eq:ResNetGeneral1D}--\eqref{eq:inputtooutput} with $\eps, \delta > 0$, $L\ge 1$. Then an arbitrary input $x\in (-1,1)$ is a critical point if and only if there exists an $l \in \{1, \ldots, L\}$ such that the parameters $\alpha = \frac{\delta}{\eps} > 0, W_l \in \R\setminus\{0\}, b_l\in \R$ satisfy
\begin{equation}
W_l  \leq -\frac{1}{\alpha},\quad  b_l = \pm(\tanh')^{-1}\left(-\frac{1}{\alpha W_l}\right) - W_l h_{l-1}(x).
\label{eq:blcondition}
\end{equation}
\end{lemma}
\begin{proof}
   The function $\tanh'(y) = \cosh^{-2}(y)$ is an even function and bijective on $\R_{+}:=[0,+\infty)\to (0,1]$ and $(-\infty,0] \to
    (0,1]$. Then, with the notation $a_l(x)=W_l h_l(x) + b_l$, and assuming $W_l \leq \frac{-1}{\alpha}$, we get from directly differentiating \eqref{eq:ResNetGeneral1D} (cf.\ Lemma~\ref{lem:resnet_Dl_fullrank}) that
        \begin{align*}
          &\frac{\partial h_{l}(x)}{\partial h_{l-1}(x)} =  \eps + \delta W_l\tanh'(a_l(x)) = 0  \numberthis \label{eq:hlprime}\\
          &\iff  a_l(x) = (\tanh'_{\R_+})^{-1}\left(-\frac{1}{\alpha W_l}\right) \quad \text{or} \quad a_l(x) = -(\tanh'_{\R_+})^{-1}\left(-\frac{1}{\alpha W_l}\right) \\
         &\iff   h_{l-1}(x) =\frac{1}{W_l} \left(\pm (\tanh'_{\R_+})^{-1}\left(-\frac{1}{\alpha W_l}\right) - b_l \right),\numberthis \label{eq:hlcondition}
    \end{align*}
    where $\pm$ is to be interpreted as positive \emph{or} negative sign.
\end{proof}

\begin{remark}~\label{rem:tanheven}
   \begin{itemize}
       \item Lemma~\ref{lem:b0condition} confirms that for $\alpha \approx 1$, critical points can be embedded with reasonable bounds on the magnitude of the model parameters.
       
       \item The crucial restriction in Lemma~\ref{lem:b0condition} is the necessity of a negative layer parameter, namely $W_l \leq \frac{-1}{\alpha} < 0$. The following implementations will show that the trained model approximations strongly depend on the sign of the different layer weights at initialization. The reason is that in our setting, the SGD iterations rarely change the weight signs during training. This relates to the influence of zero \emph{input} gradients on the vanishing \emph{parameter} gradient problem. These observations raise further questions regarding the implicit regularization of SGD in this context, which remain beyond the scope of this work.

       \item Equation~\eqref{eq:hlcondition} also shows that, as $\tanh'$ is an even function, non-degenerate zeros of the layer-wise derivative appear in pairs, when considering $h_{l-1} \in \R$ and $a_l \neq 0$. This can have the unwanted effect that the number of generated critical points ``multiplies'' from layer to layer (cf.~Figure~\ref{fig:1d2lresnet20}).
   \end{itemize} 
\end{remark}

\subsubsection{Implementation of Two-Layer ResNets}\label{subsubsec:1d2l}
We consider one-dimensional ResNets $\Phi \in \textup{RN}^1_{\eps,\delta, \sigma, \textup{N}}((-1,1),\R)$ of the form \eqref{eq:inputtooutput} with two layers $L =2$. Recalling that $a_1: = W_1x+b_1$ and $a_2: = W_2 h_1+b_2$, the two-layer model is explicitly given as
    \begin{align*}
        \Phi(x) &= \widetilde{W}_3 \big[ \eps h_1 + \delta \tanh(a_2)\big]+\tilde{b}_3 \\
        &= \widetilde{W}_3 \big[ \eps^2 x + \eps \delta \tanh(W_1 x + b_1) + \delta \tanh(W_2 h_1 + b_2)\big]+\tilde{b}_3 \\
         &= \widetilde{W}_3 \big[ \eps^2 x + \eps \delta \tanh(W_1 x + b_1) + \delta \tanh(W_2 (\eps x + \delta \tanh(W_1 x + b_1)) + b_2)\big]+\tilde{b}_3.\numberthis \label{eq:resnet1d2l}
    \end{align*}
    While the two-layer model remains structurally simple, it reveals several expressivity features that depend on the ratio between skip and residual channels. Additionally, these cases show that the ability to embed critical points does not, by itself, guarantee optimal model performance. We further point out that the trained models are deliberately chosen to highlight certain phenomena that \emph{may} occur during training for certain parameter initializations. These examples represent possible qualitative behaviors rather than a general statistical characterization of the training outcome. \medskip

\noindent\emph{Plots:} The presented images for the different ResNet realizations below each contain three subplots. The left plot is the target function approximation. The middle plot depicts the trained model parameter values of each layer where the index $3$ denotes the parameters of the affine output layer. The plot on the right-hand side shows the layer-wise derivative $\frac{\partial h_l}{\partial h_{l -1}}$ for $l = 1,2,3$ with the same notation as in Section~\ref{subsec:1dexamples}. Each non-degenerate zero of the layer-wise derivative, which are marked with a red circle in the right-hand side plots, generates at least one critical point of $\Phi$ on $\R$. Specifically, as the layer-wise derivative is even, see \eqref{eq:hlprime}, non-degenerate zeros of the layer-wise derivative always appear in pairs.

\begin{itemize}
   \item   \emph{Case $\eps = 1, \delta = 1, \alpha = 1$; Figure~\ref{fig:1d2lresnet}}: We observe that the trained ResNet with $\alpha = 1$ is able to approximate the target function well, with small approximation error in relation to the low model complexity. Both layer weights $W_1$ and $W_2$ have negative signs. According to Lemma~\ref{lem:b0condition} only the second layer generates a critical point, as $W_2 < -1$. The right-hand side plot further shows two critical points are generated in the second layer of which only one is ``active'' in the interval $(-1,1)$ of the full input-output map $\Phi$.

 \item \emph{Case $\eps = 20, \delta = 20, \alpha = 1$; Figure~\ref{fig:1d2lresnet20}}: This ResNet has channel ratio $\alpha = 1$ as in Figure~\ref{fig:1d2lresnet}. The comparison highlights that while the ratio of skip and residual channel determines the embedding capability of critical points, it is not the only aspect that influences the expressivity of the model. This becomes obvious when examining $\Phi$'s explicit form \eqref{eq:resnet1d2l}, or recalling Lemma~\ref{lem:resnet_inputoutput} for the general case. In the presented case, the trained model $\Phi$ displays catastrophic approximation results with too many critical points (cf. the last item in Remark~\ref{rem:tanheven}). Both the first layer and the second layer generate a pair of critical points. In the final output function, the combined layers lead to a total of eight critical points due to the layered structure of $\Phi$. This behavior stems from the gradient of $\Phi$ in \eqref{eq:resnet_gradient}. Since each layer appears as a factor, the critical points of individual layers aggregate. Specifically, the two critical points of the second layer are each attained at two distinct input values. Let $h^*$ be one of the two critical points of the map $h \mapsto h_2(h)$, such that $\frac{\partial h_2}{\partial h_1}(h^*) = 0$. Then, there exist two inputs $x_{1,2} \in (-1,1)$ satisfying $h_1(x_{1,2}) = h^*$. Consequently, the factor in $\nabla_x \Phi$ corresponding to the second layer generates four critical points $x_{1,2,3,4} \in (-1,1)$ via:
\begin{equation}
\frac{\partial h_2}{\partial h_1}(h_1(x_{i})) = 0, \quad i \in \{1, \dots, 4\}.
 \end{equation}

    \begin{figure}[t]
        \centering
        \includegraphics[width=0.7\linewidth, trim=0 0 0 0, clip]{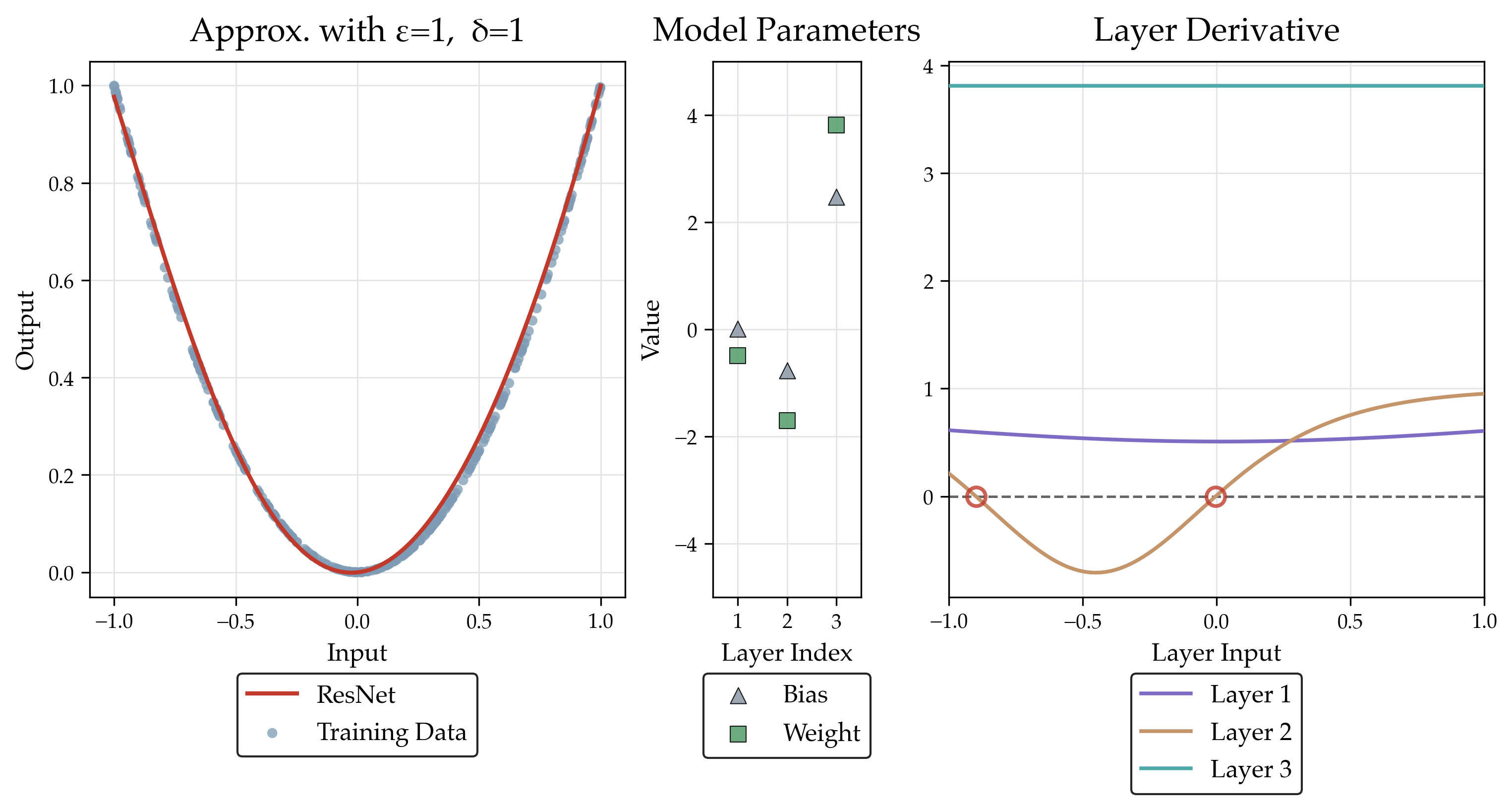}
        \caption{Trained ResNet of the form \eqref{eq:resnet1d2l} with $\eps = \delta = \alpha = 1$.\label{fig:1d2lresnet}}
    \end{figure}
    \begin{figure}
        \centering
        \includegraphics[width=0.7\linewidth, trim=0 0 0 0, clip]{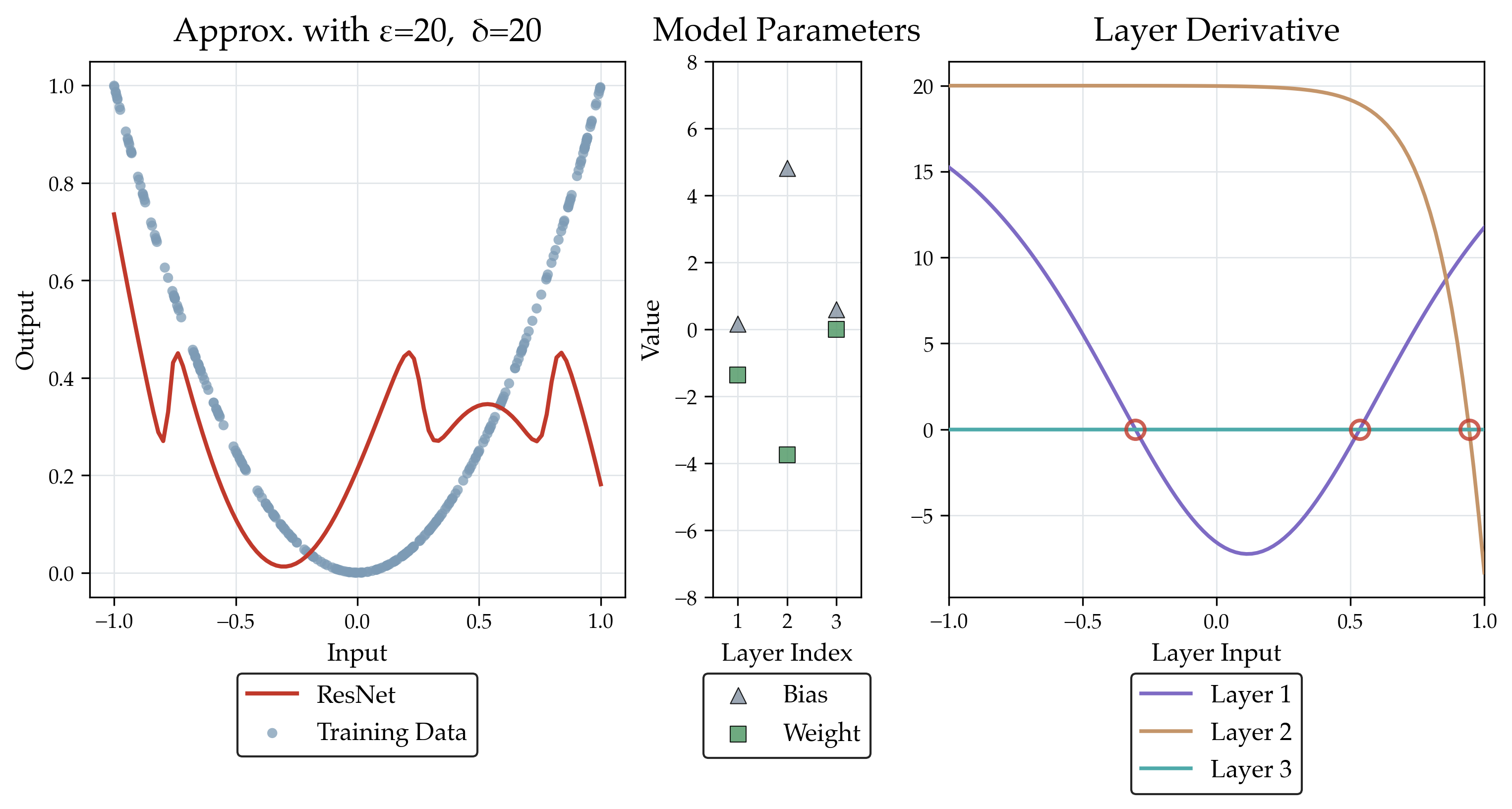}
        \caption{\small Trained ResNet \eqref{eq:resnet1d2l} with equal channel ratio, $\alpha = 1$, but $\eps=20$ and $\delta = 20$.\label{fig:1d2lresnet20}}
    \end{figure}

    \item \emph{Case $\eps = 1, \delta = 0.3, \alpha = 0.3$; Figure~\ref{fig:1d2lresneteps1deltap3}:} This case is representative of the regime $\alpha \ll 1$. In accordance with the assumptions of Theorem~\ref{th:resnet_criticalpoints_alphasmall}, we observe large weight magnitudes that compensate for the small channel ratio $\alpha = 0.3$. Nonetheless, the model fails to approximate the function well. The error is large even though critical points are generated in the first layer. This highlights that expressivity depends on more than just the ability to express critical points. Looking at \eqref{eq:resnet1d2l} directly, reveals that the linear term is dominant, as each other nonlinear term is bounded by $|\delta|$. This leads to $\Phi$ being almost piecewise-linear and hence unable to approximate $f(x) = x^2$ accurately.

      
  \item \emph{Case $\eps = 0.01, \delta = 1, \alpha = 100$;     Figure~\ref{fig:1d2lmlp}:} This case represents the regime $\alpha \gg 1$, where the ResNet's structure is close to that of non-augmented MLPs. Indeed, the plotted $\Phi$ resembles the one-layer MLP from above with $\eps = 0$, depicted in Figure~\ref{fig:1d1lresnet}\subref{subfig:eps0}. However, in contrast to MLPs, the case $\eps = 0.01$ still allows critical points. The right-hand side plot illustrates this, as the derivative of the second layer has a zero, in line with Lemma~\ref{lem:b0condition}. This does not have a significant impact on the expressivity of the ResNet. Because this ResNet closely approximates the MLP architecture, it shares a nearly identical expressive profile.

Let us summarize the additional observations regarding ResNet expressivity that we gathered from the two-layer case study:
\begin{itemize}
    \item \emph{Which layer in the model hierarchy generates a critical point matters.} Whether a critical point is generated in the first or second layer can lead to qualitatively very different outcomes in the input-output map $\Phi$. Specifically, due to the (anti-)symmetry of the activation function $\sigma = \tanh$, critical points present in the first layer can be ``multiplied'' by the second layer, specifically for a dominant skip parameter $\eps$.
    \item \emph{Not only $\alpha$ but also the values of $\eps$ and $\delta$ matter.} The channel ratio $\alpha = \frac{\delta}{\eps}$ governs the structural capacity to embed critical points, but the absolute scales of $\eps$ and $\delta$ are equally relevant for the model's expressivity (cf.~Remark~\ref{rem:scaling_asymmetry}). A small $\delta$ forces $\Phi$ toward near piecewise linear behavior (regardless of $\alpha$). Large values of $\eps$ (regardless of $\alpha$) can lead to catastrophic approximation.

    \item \emph{Initialization matters.} In our one-dimensional setting the parameter training keeps the sign of the weights from initialization unchanged. This leads to SGD convergence to sign-based local minima. There are various design choices that can circumvent this problem, such as increasing the parameter dimension $m_l$ in \eqref{eq:resnet_typical_f}. Nonetheless, it shows that the parameter regime at initialization strongly influences the trained final outcome due to implicit regularization effects of SGD based algorithms \cite{MWSB24}.
\end{itemize}

      \begin{figure}[t]
        \centering
        \includegraphics[width=0.7\linewidth, trim=0 0 0 0, clip]{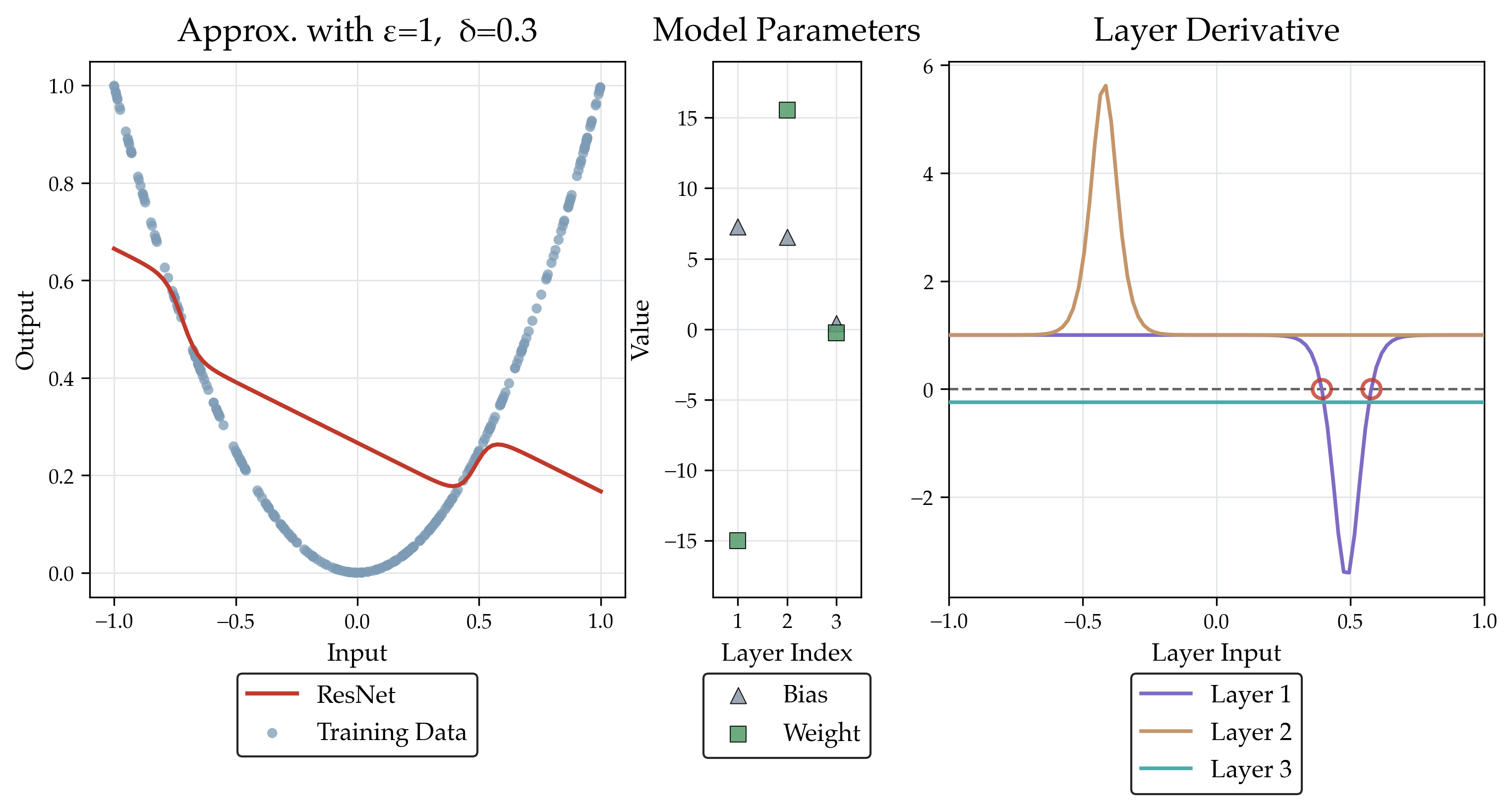}
        \caption{Trained ResNet \eqref{eq:resnet1d2l} with small residual term, $\eps = 1$ and $\delta = 0.3$. \label{fig:1d2lresneteps1deltap3}}
    \end{figure}

            \begin{figure}
        \centering
        \includegraphics[width=0.7\linewidth, trim=0 0 0 0, clip]{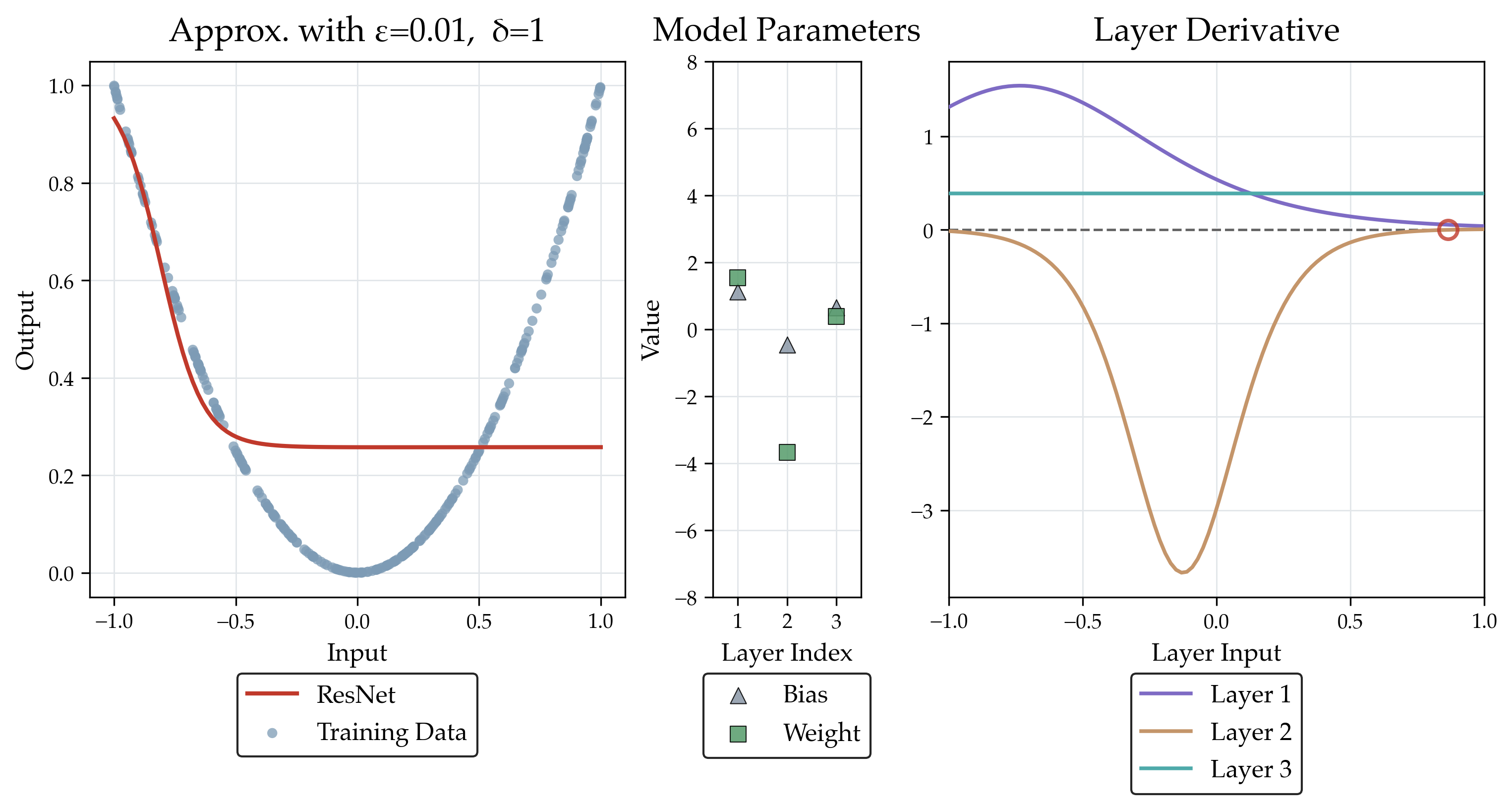}
        \caption{Trained ResNet \eqref{eq:resnet1d2l} close to MLPs with $\eps = 0.01$, $\delta = 1$. \label{fig:1d2lmlp}}
    \end{figure}
\end{itemize}

\subsection{Two-Dimensional ResNets}

In this subsection, we consider ResNets to solve classification tasks for two-dimensional toy datasets. This illustrates how the universal approximation constraints of non-augmented ResNets manifest in SGD-trained networks. Specifically, the results demonstrate the ``tunnel effect'' in various regimes analyzed in the previous sections.

These two-dimensional experiments correspond to the approximation of the prototypical functions $f: \R^2 \to \R$ that exhibit non-degenerate critical points:
    \begin{equation}\label{eq:Psi_prototypical}
        \Psi_\text{circ}(x_1,x_2) = x_1^2+x_2^2-0.5\quad \text{and}\quad  \Psi_\text{xor}(x_1,x_2) = x_2^2-x_1^2-0.5.
    \end{equation}
    Both have a single critical point at $x_1 = x_2 = 0$, the function $\Psi_\text{circ}$ a minimum, and $\Psi_\text{xor}$ a saddle point. These two functions correspond to the prototypical toy datasets of \texttt{Circle} and \texttt{XOR} datasets via their level sets as shown in Figure~\ref{fig:datasets}. For any model to successfully classify the dataset for an arbitrary number of points (Definition~\ref{def:binary_classification}), it must effectively approximate these functions and hence embed the corresponding critical points.

    \begin{figure}[h]
	\centering
	\begin{subfigure}{0.45\textwidth}
		\centering
		\begin{overpic}[scale = 0.5,tics=10]
			{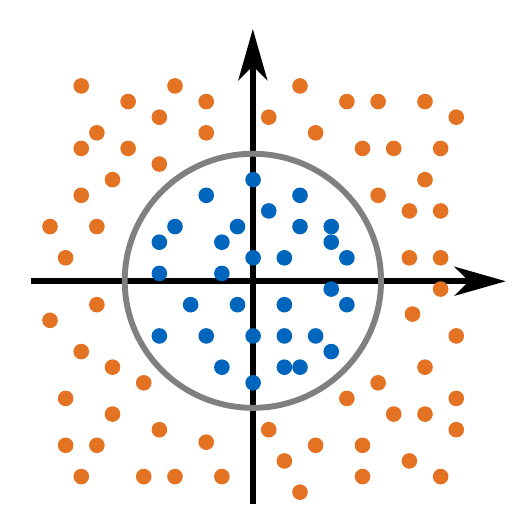}
			\put(97,45.5){\textcolor{black}{$x_1$}}
            \put(44,97){\textcolor{black}{$x_2$}}
		\end{overpic}
		\caption{\texttt{Circle} dataset with $\Psi_\text{circ}(x_1,x_2) = x_1^2+x_2^2-0.5$.}
        \label{fig:datasets_a}
	\end{subfigure}
    \hspace{3mm}
	\begin{subfigure}{0.45\textwidth}
		\centering
		\begin{overpic}[scale = 0.5,tics=10]
			{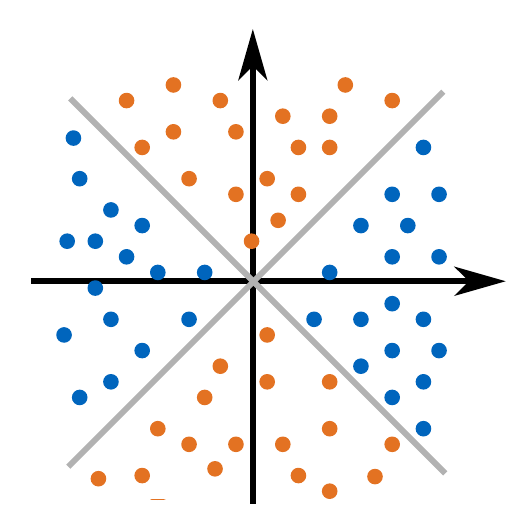}
			\put(97,45.5){\textcolor{black}{$x_1$}}
            \put(44,97){\textcolor{black}{$x_2$}}
		\end{overpic}
		\caption{\texttt{XOR} dataset with $\Psi_\text{xor}(x_1,x_2) = x_2^2-x_1^2-0.5$.}
        \label{fig:datasets_b}
	\end{subfigure}
	\caption{Two-dimensional datasets defined as sub- and super-level sets of the prototypical functions~\eqref{eq:Psi_prototypical}: Points with $\Psi_i(x)<0.5$ have a blue label and points with $\Psi_i(x)>0.5$ have an orange label, $i \in \{\text{circ},\text{xor}\}$. Figure \ref{fig:datasets}\subref{fig:datasets_a}-\subref{fig:datasets_b} is adapted from \cite[Figure 4(a)-(b)]{Chemnitz2025}.}
	 \label{fig:datasets}
\end{figure}

\subsubsection{Details on the Implementations}

\noindent \emph{Parameter Training:} 1400 training points are drawn from the \texttt{Circle} or \texttt{XOR} datasets as illustrated in Figure~\ref{fig:datasets}. Parameter training is initialized with Xavier uniform distribution and optimized with cross-entropy loss and the stochastic gradient descent based Adam algorithm with batch size 128 and learning rate $0.01$. We use batch normalization during training to improve trainability, which is particularly helpful in the MLP regimes. Note that the added batch normalization is consistent with the general model structure of Definition~\ref{def:resnet} at the model evaluation phase. The chosen examples are rather best picks than average training outcomes, as we focus on possible expressivity while keeping the number of parameters to a minimum.\\

\noindent \emph{ResNet Structure:} The trained models are all based on the class $\textup{RN}_{\eps,\delta, \sigma}^1((-2.5,2.5)^2, (0,1))$ from Definition~\ref{def:resnet}. Specifically, we consider 
\begin{equation}\label{eq:2dphi}
        \Phi(x) = \tilde{\lambda}(h_L(\lambda(x))), \quad x\in(-2.5,2.5)^2,
    \end{equation}
    with residual functions
    \begin{equation}\label{eq:2dresfun}
        f_l(h_{l-1}, \theta_l) = \tanh(W_l h_{l-1} + b_l), \quad l\in \{1,\ldots, L\}.
    \end{equation}
    This means we only consider simplified residual functions with $\Wt_ l = \textup{Id}_{2}$, $\bt_l = 0$ from which it follows that $W_l \in \R^{\nhid \times \nhid}$, and $b_l \in \R^{\nhid}$. The input dimension $\nin = 2$ and output dimension $\nout = 1$ are fixed for all examples. The input layer is nonlinear of the form $\lambda(x) = \tanh(W_0 x + b_0)$, with $W_0 \in \R^{\nhid \times 2}$, $b_0 \in \R^{\nhid}$ and the output layer $\tilde{\lambda}(y) = \operatorname{sigmoid} (\Wt_{L+1} y + \bt_{L+1})\in (0,1)$ with $\Wt_{L+1} \in \R^{1\times \nhid}$, $\bt_{L+1} \in \R$ normalizing the final outputs to classification probabilities.
    \\

    \noindent \emph{Plots}: The provided plots depict the prediction probability level sets, i.e., the output value of the map $\Phi :\R^{2} \to (0,1)$ for each input in $x\in (-2.5, 2.5)^2$. The left-hand side plots additionally include test data drawn independently from the same data distribution as the training points. The right-hand side reduces the plots to the contours of the level sets to emphasize the (non-)existence of critical points corresponding to the (non-)existence of bounded level sets.

\subsubsection{Discussion of Numerical Examples}
\begin{itemize}
    \item \emph{Case $\alpha \ll 1$, $\eps = 1, \delta = 0.1$; Figure~\ref{fig:2dnode}}: Trained is a non-augmented ResNet with structure $\Phi \in \textup{RN}_{1,0.1,\sigma, \textup{N}}((-2.5,2.5)^2, (0,1))$ with $\nhid = 2$ on the \texttt{Circle} dataset. This model is approaching the neural ODE regime. The ResNet has 20 hidden layers, $L = 20$, of type \eqref{eq:2dresfun} with skip parameter $\eps = 1$ and residual parameter $\delta = \frac2L = 0.1$. This corresponds to an Euler discretized neural ODE on the time interval $[0,2]$ and step size $0.1$ (cf.~Proposition~\ref{prop:resnet_neuralODE}). As analyzed in Section~\ref{sec:node}, neural ODEs are unable to embed critical points and hence their input-output map is unable to express the desired blue disk around the origin. We further infer from Theorem~\ref{thm:constraints_nd} that this results in a blue ``tunnel'' towards the domain boundary (cf.~\cite{Dupont2019}).  Despite the $\Phi$ being a rather coarse discretization of the continuous neural ODE dynamics (with potentially large approximation error in sup-norm), the plotted ResNet behaves qualitatively very similarly to a neural ODE. It also forms a blue tunnel which results in the undesired misclassifications of orange data points.

    Theorem~\ref{th:resnet_criticalpoints_alphasmall} provides a deeper understanding of why the ResNet also forms a ``tunnel''. The trained ResNet is theoretically able to embed critical points by compensating the small channel ratio $\alpha = 0.1$ with weights of large magnitude (see Lemma~\ref{lem:resnet_Dl_fullrank}). However, the initialized parameters of moderate size stay in a moderate regime throughout the SGD-based training.  It follows that the condition $\alpha < \frac{1}{\nu_{\max} K_\sigma}$ of Theorem~\ref{th:resnet_criticalpoints_alphasmall}, which excludes the existence of critical points, stays relevant throughout training. This holds despite the possibly significant dynamical difference between the ResNet and its continuous neural ODE limit.

    \begin{figure}[t]
    	\centering 
    	\begin{subfigure}{0.4\textwidth}
    		\includegraphics[width=\textwidth, trim={0 2em 0 0},clip]{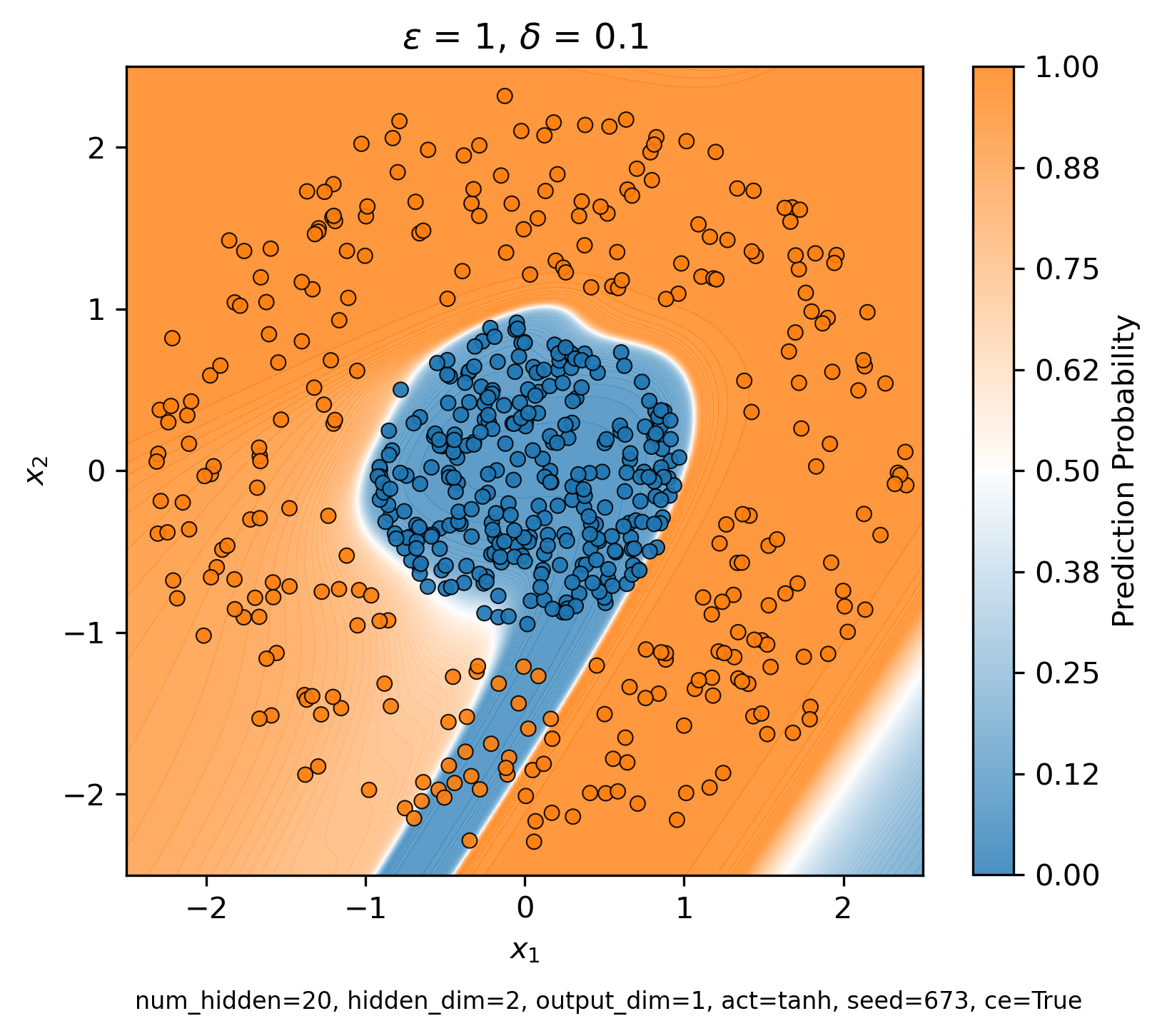}
    	\end{subfigure}
    	\begin{subfigure}{0.4\textwidth}
    		\includegraphics[width = \textwidth, trim={0 2em 0em 0},clip]{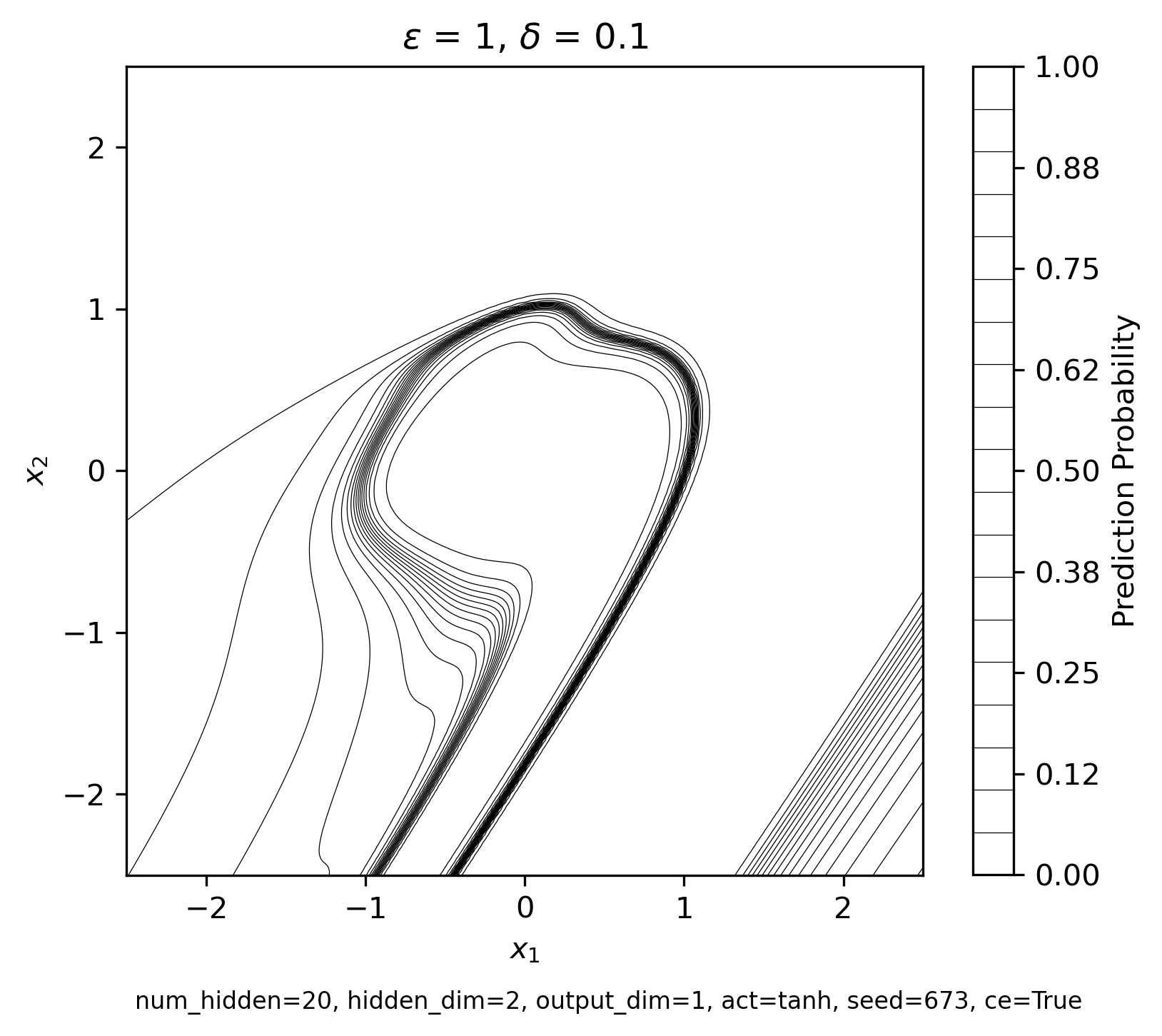}
    	\end{subfigure}
    	\caption{\small Prediction level sets of a non-augmented ResNet of the form \eqref{eq:2dphi} with $\eps = 1$ and $\delta = 0.1$ in the neural ODE regime. The model is unable to express critical points which leads to the blue ``tunnel''. \label{fig:2dnode}}
    \end{figure}

    \begin{figure}
    	\centering 
    	\begin{subfigure}{0.4\textwidth}
    		\includegraphics[width=\textwidth, trim={0 2em 0 0},clip]{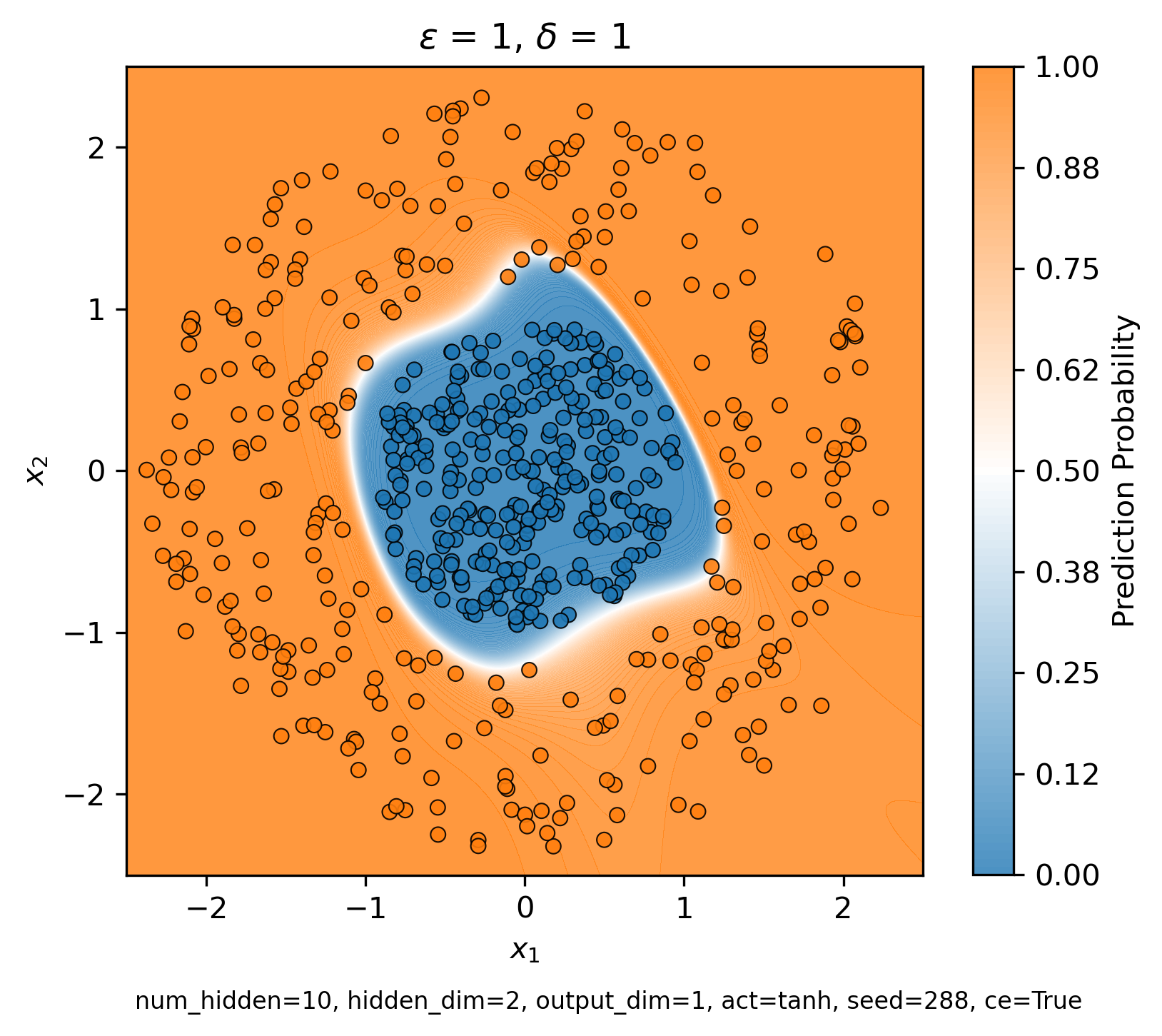}
    	\end{subfigure}
    	\begin{subfigure}{0.4\textwidth}
    		\includegraphics[width = \textwidth, trim={0 2em 0em 0},clip]{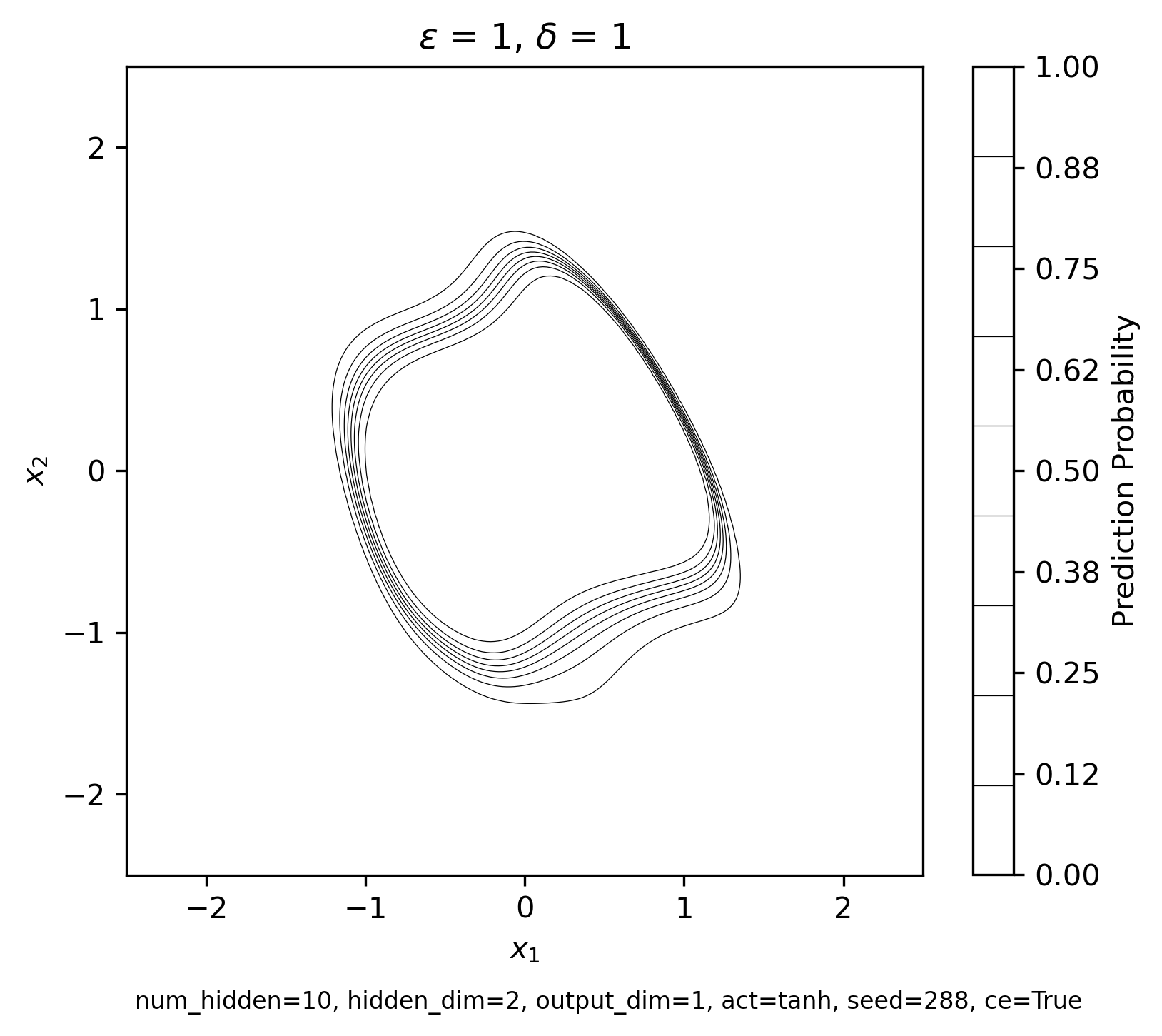}
    	\end{subfigure}
    	\caption{\small Prediction level sets of a non-augmented ResNet of the form \eqref{eq:2dphi} with balanced skip and residual channels $\eps = \delta = 1$. A critical point around the origin is present.  \label{fig:2drescirc}}
    \end{figure}

 \item \emph{Case $\alpha = 1$, \texttt{Circle} dataset; Figure~\ref{fig:2drescirc}:} Trained is a non-augmented ResNet with standard channel parameters $\eps = \delta = 1$, $L = 10$ and $\nhid = 2$ on the \texttt{Circle} dataset. The plots show that the model is able to express bounded blue level sets around the origin without the appearance of a ``tunnel'' and achieves high classification accuracy. As the input-output map is continuous, the bounded level sets further imply the existence of a critical point. This confirms that non-augmented ResNets with balanced channels are able to embed critical points in contrast to non-augmented MLPs (cf.~Figure~\ref{fig:intro}) and neural ODEs (cf.~Figure~\ref{fig:1d2lresnet20}). This represents the case ``in-between'' the neural ODE regime (Theorem~\ref{th:resnet_criticalpoints_alphasmall}) and the MLP regime (Theorem~\ref{th:resnet_criticalpoints_alphalarge}) where neither of both conditions on $\alpha$ to exclude critical points are satisfied (cf.\ the case $\alpha = 1$ of Section~\ref{subsec:1dexamples}).

      \item \emph{Case $\alpha \gg 1$, $\eps = 0.1, \delta = 1$; Figure~\ref{fig:2dmlp}:} This non-augmented ResNet structure is approaching the MLP regime with a dominant residual channel as $\alpha = 10$. The model with $L = 6$ and $\nhid = 2$ is trained on the \texttt{XOR} dataset (cf.\  Figure~\ref{fig:intro} for the \texttt{Circle} dataset case). The plot shows misclassifications close to the origin, as the model fails to approximate $f(x) = x^2_2 - x^2_1$ there, which requires a critical point in form of a saddle point. Instead, the orange level sets ``tunnel'' through the blue level sets. The outcome aligns with the conditions of Theorem~\ref{th:resnet_criticalpoints_alphalarge} that exclude critical points of $\Phi$ for large $\alpha$. The \texttt{XOR} case of a saddle point is not covered by Theorem~\ref{thm:constraints_nd}, as optimal level sets naturally intersect with the domain boundary. Nonetheless, the outcome is rather similar: the inability to embed the critical point leads to a tunnel and as a result degradation in accuracy. 

            \begin{figure}[t]
    	\centering 
    	\begin{subfigure}{0.4\textwidth}
    		\includegraphics[width=\textwidth, trim={0 2em 0 0},clip]{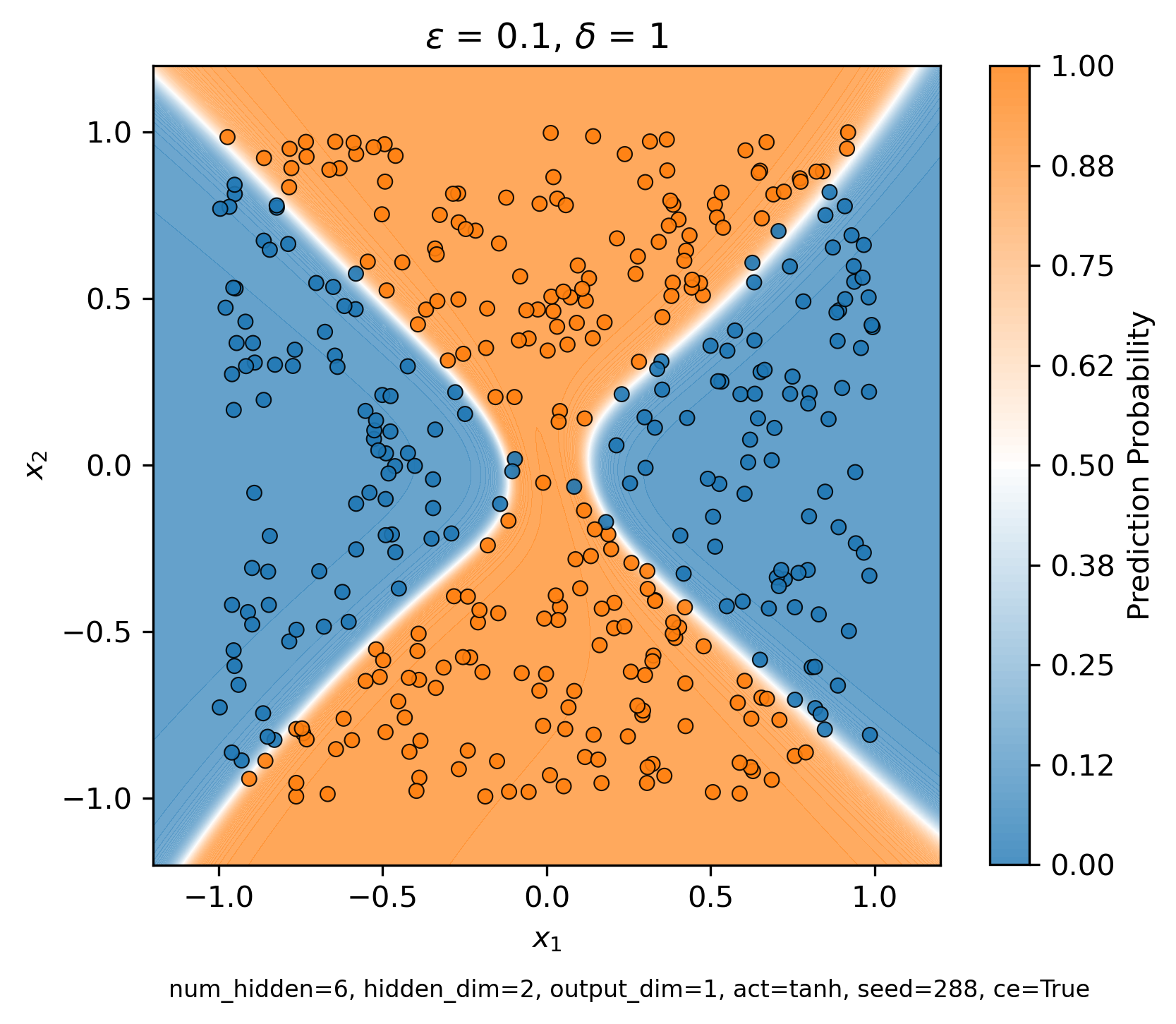}
    	\end{subfigure}
    	\begin{subfigure}{0.4\textwidth}
    		\includegraphics[width = \textwidth, trim={0 2em 0em 0},clip]{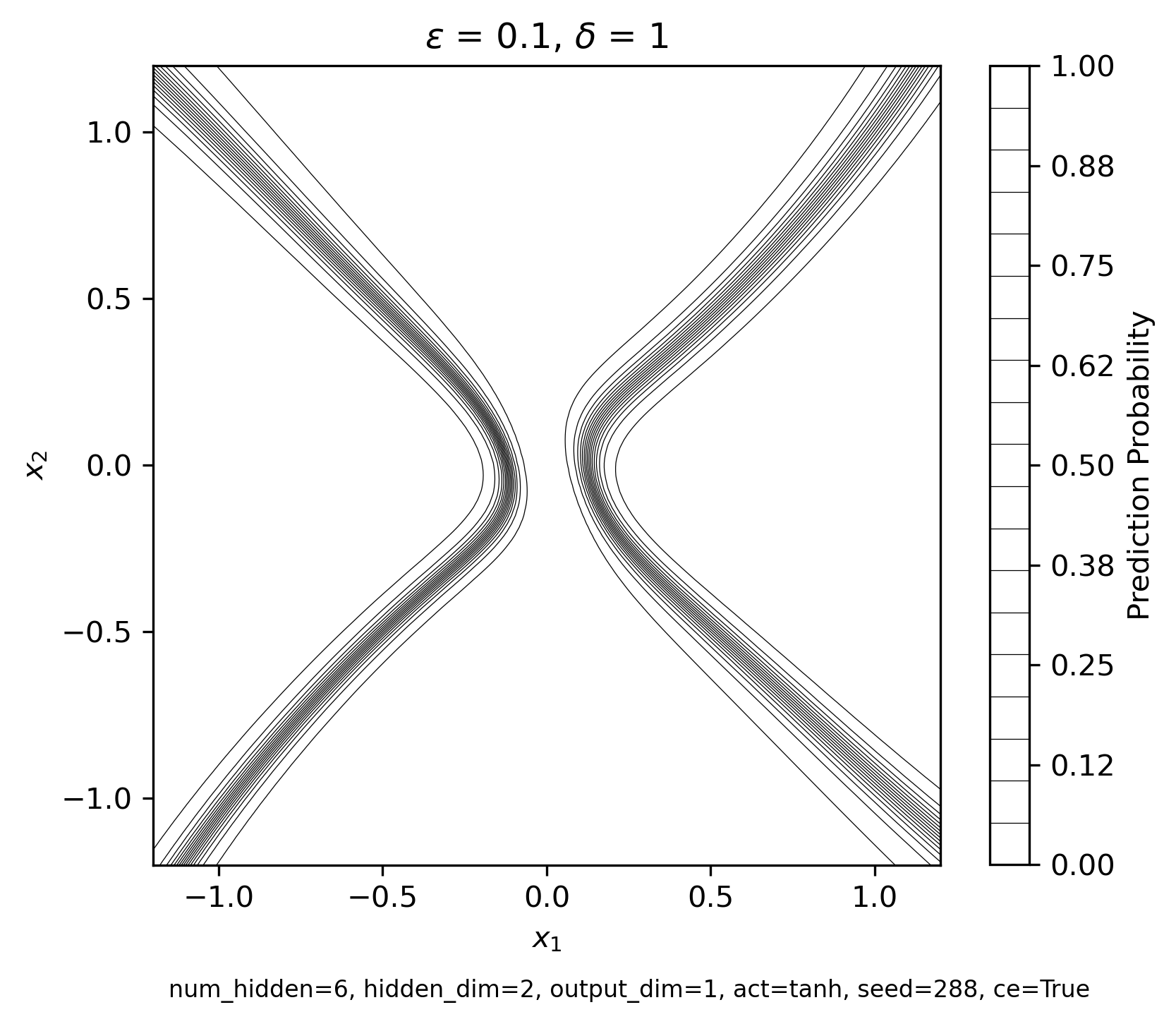}
    	\end{subfigure}
    	\caption{\small Prediction level sets of a non-augmented ResNet of the form \eqref{eq:2dphi} in the MLP regime, $\eps = 0.1$ and $\delta = 1$ which shows an undesired orange ``tunnel'' at the origin due to the absence of a critical point. \label{fig:2dmlp}}
    \end{figure}

 \begin{figure}
    	\centering 
    	\begin{subfigure}{0.4\textwidth}
    		\includegraphics[width=\textwidth, trim={0 2em 0 0},clip]{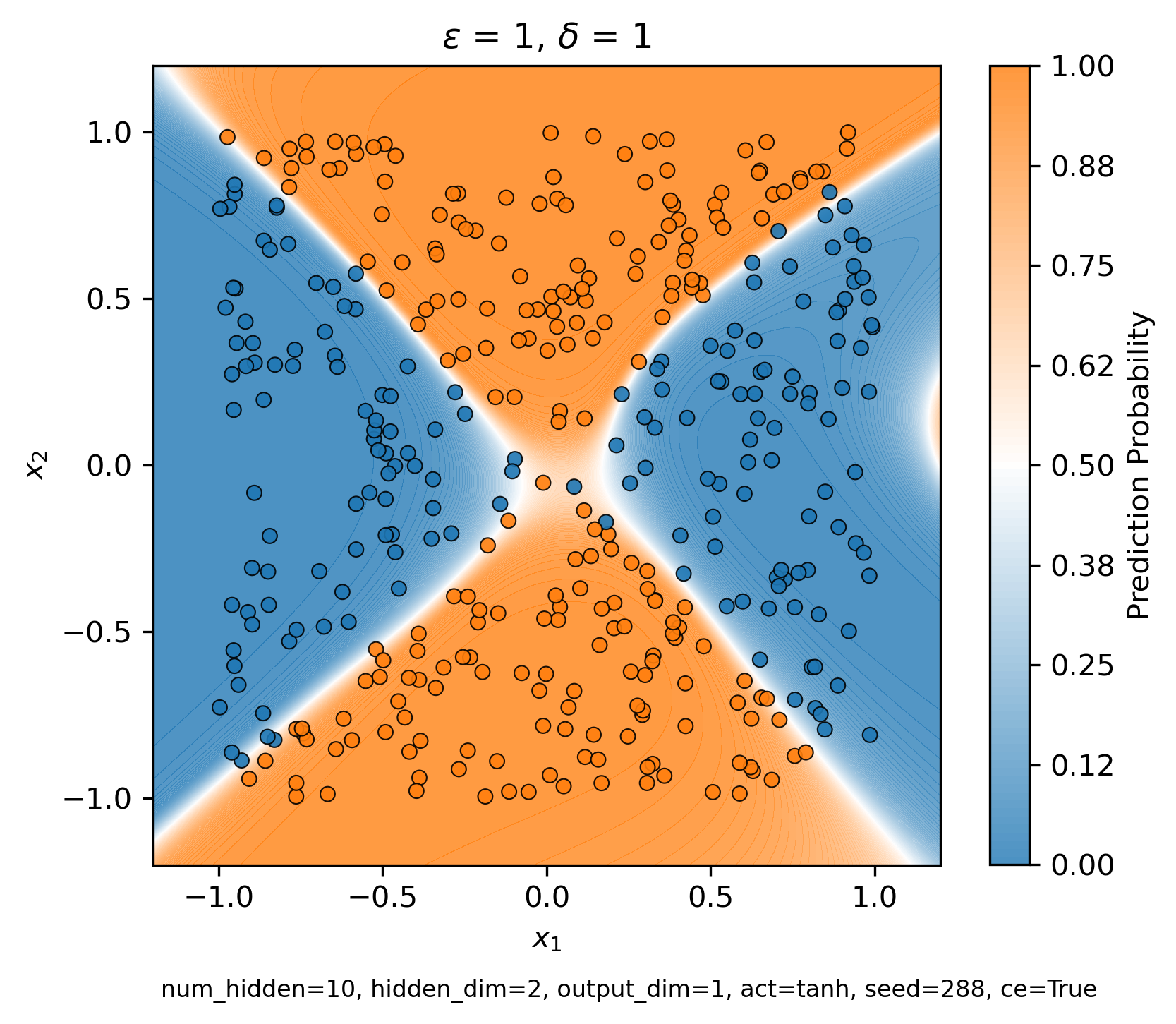}
    	\end{subfigure}
    	\begin{subfigure}{0.4\textwidth}
    		\includegraphics[width = \textwidth, trim={0 2em 0em 0},clip]{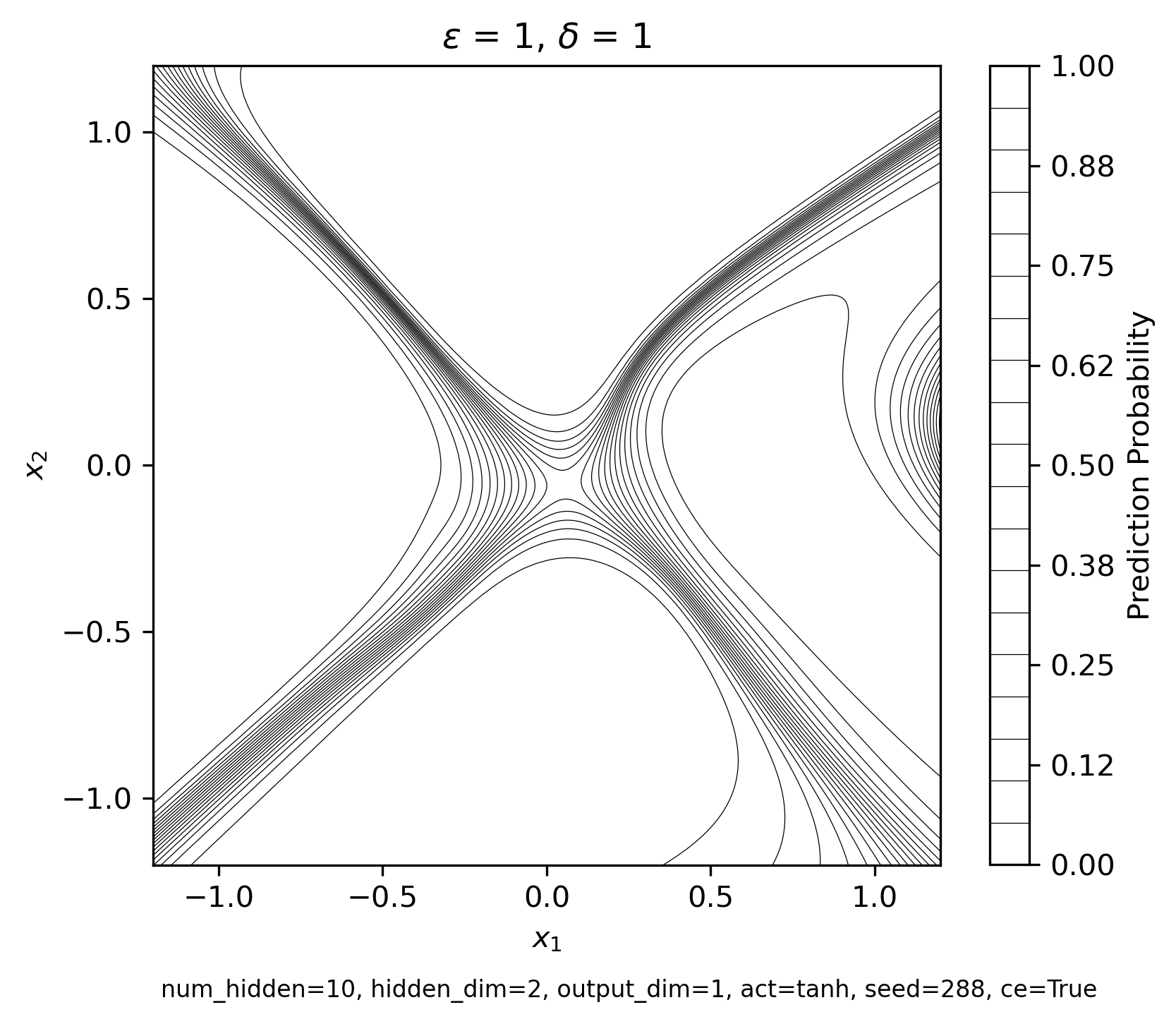}
    	\end{subfigure}
    	\caption{\small Prediction level sets of a non-augmented ResNet of the form \eqref{eq:2dphi} with balanced skip and residual channels $\eps = \delta = 1$. A critical point in form of a saddle point around the origin is present.  \label{fig:xorres}}
    \end{figure}

\item \emph{Case $\alpha = 1$, \texttt{XOR} dataset; Figure~\ref{fig:xorres}:} The non-augmented ResNet is trained on the \texttt{XOR} dataset with balanced channels $\eps = \delta = 1$ and $L=10$. As for the \texttt{Circle} dataset case above, the model embeds a critical point, in this case a saddle point, close to the origin.  As a result, it achieves superior classification accuracy compared to the case $\alpha = 10$ above and learns the desired topological structure

     \item \emph{Case augmented MLP, $\eps = 0, \delta = 1$; Figure~\ref{fig:2daugcirc}:} The final example depicts an augmented MLP, specifically $\Phi \in \textup{RN}_{0,1,\sigma, \textup{A}}((-2.5,2.5)^2,(0,1))$ with $\nhid = 3$ and $L = 1$. The model is trained on the \texttt{Circle} dataset and achieves high accuracy for the classification task by embedding a critical point close to the origin and generating bounded blue level sets without ``tunnel''. In this work, we do not analyze the augmented structure any further and refer to \cite{kk2025}. Compared to the ResNet cases with $\alpha = 1$, which manage to also embed a critical point, this augmented model requires far less parameters to achieve high accuracy results and is more robust across independent parameter initializations.

\end{itemize}

This numerical case study illustrates the existence of critical points of trained ResNets within the different channel ratio regimes. Notably, it confirms that our analysis of the $\alpha$ parameter regime holds even for small deviations from the standard ResNet case, where $\alpha = 1$. The experiments further demonstrate how the proven ``tunnel effect'' of Theorem~\ref{thm:constraints_nd} is reflected in the actual implementations. For the \texttt{Circle} dataset the absence of a critical point embedding corresponds to the blue level sets forming a tunnel towards the domain boundary with misclassifications. The \texttt{XOR} case is not directly covered by Theorem~\ref{thm:constraints_nd}, as the optimal level sets naturally intersect with the boundary. The implementations show that ResNets without a saddle point will form a tunnel through the origin.

     \begin{figure}
    	\centering 
    	\begin{subfigure}{0.4\textwidth}
    		\includegraphics[width=\textwidth, trim={0 2em 0 0},clip]{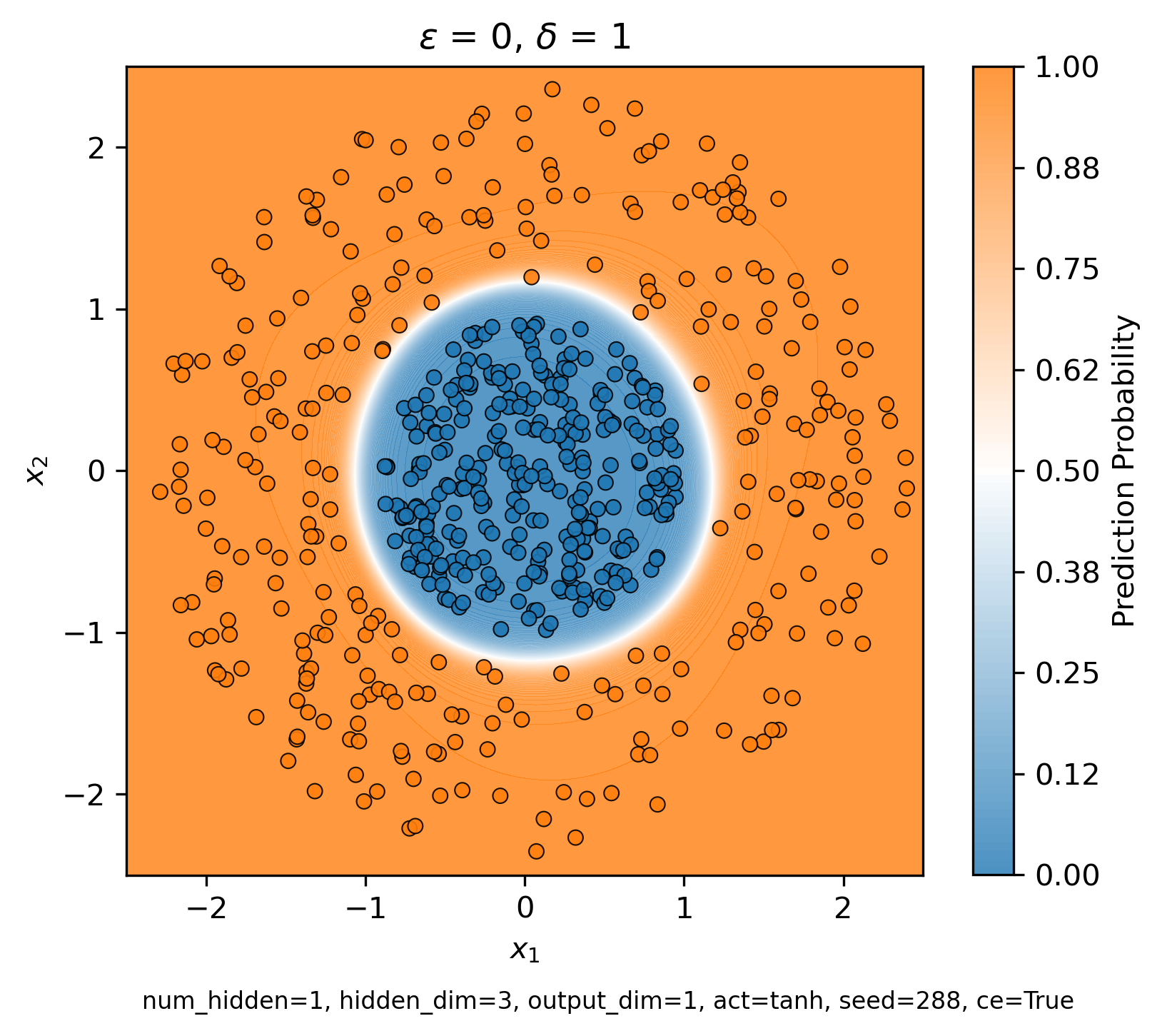}
    	\end{subfigure}
    	\begin{subfigure}{0.4\textwidth}
    		\includegraphics[width = \textwidth, trim={0 2em 0em 0},clip]{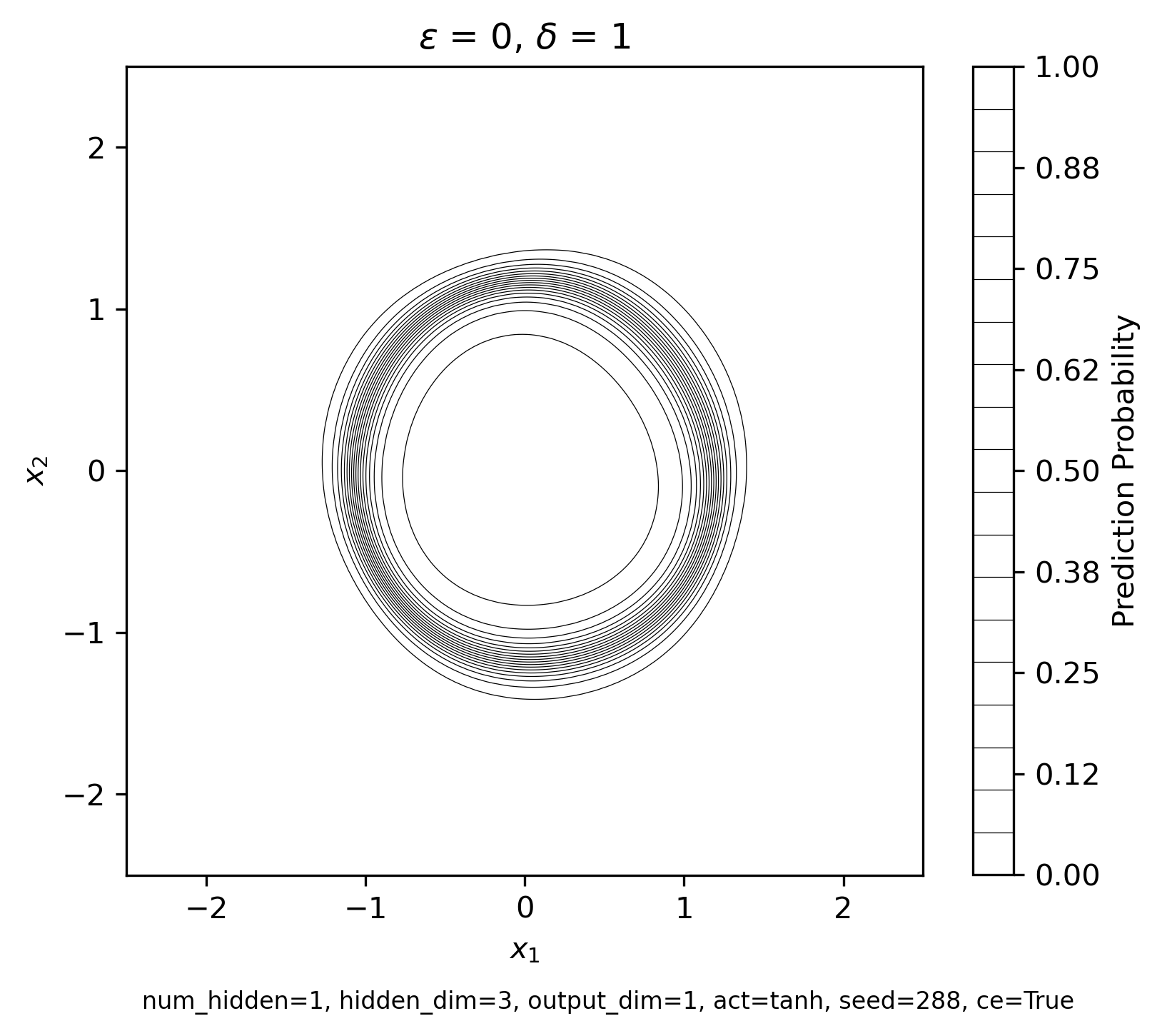}
    	\end{subfigure}
    	\caption{\small Prediction level sets of an augmented MLP of the form \eqref{eq:2dphi} with $\nhid = 3$, $\eps = 0$ and $\delta = 1$. The embedding of a critical point in form of a minimum around the origin is achieved. \label{fig:2daugcirc}}
    \end{figure}

\section{Conclusion}
In this work, we investigated the universal approximation limitations of non-augmented (or narrow) ResNets, focusing on their ability to express critical points. By parameterizing the layer update rule with a skip parameter and a residual parameter, we introduced the channel ratio $\alpha$. This framing established a unified mathematical perspective on narrow ResNet architectures, and allowed us to formally bridge the gap between continuous neural ordinary differential equations and standard feed-forward neural networks. The inability to express critical points in the input-output map fundamentally restricts a network's global approximation capabilities. We proved that the lacking expressivity forces the network to shift required critical points outside the observation domain, which leads to a ``tunnel effect''. In classification tasks with nested data, this means the decision boundary always has to intersect with the domain boundary which leads to undesired misclassifications. 

We systematically categorized the embedding capabilities of narrow ResNets across three distinct regimes of the channel ratio $\alpha$: 
\begin{itemize}
    \item The neural ODE regime ($0 < \alpha \ll 1$): When the residual channel is small relative to the skip connection, ResNets act similar to continuous neural ODEs. We provided explicit upper bounds on $\alpha$ below which the network strictly inherits the topological limitations of non-augmented neural ODEs, and hence is unable to embed critical points.

\item The MLP regime ($\alpha \gg 1$): When the residual channel dominates, the architecture reduces to a perturbed standard feed-forward neural network. We established lower bounds on $\alpha$ above which the ResNet is equally incapable of expressing critical points, mirroring the constraints of standard narrow feed-forward networks.

\item The intermediate regime ($\alpha \approx 1$): In this balanced regime, typical of standard ResNets, the embedding of critical points is generally possible. The numerical examples demonstrated that the ability to successfully embed a critical point remains sensitive to the precise channel ratio and magnitude of the initialized layer weights.
\end{itemize} 

While non-augmented ResNets can theoretically resolve the topological restrictions of their MLP and neural ODE counterparts, this structural advantage is fragile. Because of implicit regularization effects of SGD, models initialized in sub-optimal parameter regions may converge to solutions that exhibit the tunnel effect despite having the theoretical capacity to avoid it.

\vspace{5mm}

\noindent\textbf{Acknowledgments:} CK and SVK thank the DFG for partial support via the SPP2298 `Theoretical Foundations of Deep Learning'. CK thanks the VolkswagenStiftung for support via a Lichtenberg Professorship. SVK thanks the Munich Data Science Institute (MDSI) for partial support via a Linde doctoral fellowship. TW is supported by the Austrian Science Fund (FWF) 10.55776/J4681.
	
\vspace{3mm}

    \addcontentsline{toc}{section}{References}
	\bibliographystyle{abbrvurl}
	\bibliography{bibliography}

    \appendix
    \addtocontents{toc}{\protect\setcounter{tocdepth}{1}}

    \section{Relationship between ResNets, Neural ODEs and FNNs}
    \label{app:relationship}

    This appendix collects the proofs of Proposition~\ref{prop:resnet_neuralODE} and Proposition~\ref{prop:resnet_mlp}, which concern the relationship between ResNets and neural ODEs, and between ResNets and FNNs, respectively.

    \subsection{Proof of Proposition~\ref{prop:resnet_neuralODE}}
    \label{app:proof_resnet_neuralODE}

\begin{proof}
    To prove part~\ref{prop:resnet_neuralODE_a}, let $\overline{\Phi} \in \NODE$ be a neural ODE based on the initial value problem~\eqref{eq:IVP}. For the corresponding ResNet $\Phi$ we choose the transformations $\lambda$, $\tilde{\lambda}$ to be the same as for the neural ODE $\overline{\Phi}$. An explicit Euler discretization of the initial value problem~\eqref{eq:IVP} over the time interval $[0,T]$ with step size $\delta \coloneqq \frac{T}{L}$ results in
    \begin{equation}\label{eq:Euler1}
        h(t + \delta) = h(t) + \delta \cdot f(h(t),\theta(t)).
    \end{equation}
    By defining $t_l \coloneqq l \delta$ for $l \in \{0,\ldots,L\}$, $h_l \coloneqq h(t_l)$ for $l \in \{0,\ldots, L\}$, and $\theta_l \coloneqq \theta(t_{l-1})$ for $l \in \{1,\ldots, L\}$, the discretization~\eqref{eq:Euler1} simplifies at the time points $t_l$, $l \in \{0,\ldots,L-1\}$ to
    \begin{equation} \label{eq:Euler2} 
        h_{l+1} = h_{l} + \delta \cdot f(h_{l},\theta_{l+1}).
    \end{equation}
    After the index shift $l \mapsto l-1$, the update rule~\eqref{eq:Euler2} agrees with the ResNet update rule~\eqref{eq:resnet_updaterule} with $\eps = 1$ and $f_l(\cdot,\cdot) = f(\cdot,\cdot)$ for all $l \in \{1,\ldots,L\}$. As $f\in C^{k,0}(\R^{\nhid}\times \R^p,\R^{\nhid})$, it follows that $f(\cdot,\theta_l)\in C^{k}(\R^{\nhid},\R^{\nhid})$ for each fixed $\theta_l\in \R^{p_l}$, and hence assertion~\ref{prop:resnet_neuralODE_a} follows.
    
    For part~\ref{prop:resnet_neuralODE_b}, we consider a ResNet $\Phi \in \RN$ with $L$ hidden layers and update rule
    \begin{equation}  \label{eq:Euler3} 
        h_{l} = \eps h_{l-1} + \delta f_\textup{RN}(h_{l-1},\theta_{l})
    \end{equation}
    for $l \in \{1, \ldots, L\}$, where the parameter dimensions $p_l = p$ and the residual functions $f_l(\cdot,\theta_l) \coloneqq f_\textup{RN}(\cdot,\theta_l)$ are independent of the layer index $l$. To define a corresponding neural ODE, let $\theta\in C^\infty(\R,\R^p)$ be a smooth interpolation of the parameters with $\theta_l \coloneqq \theta(t_{l-1})$ for $l \in \{1,\ldots,L\}$, where $t_l \coloneqq l \delta$ for $l \in \{0,\ldots,L\}$. Such a smooth interpolation always exists, for example by using Lagrange polynomials~\cite{Humpherys2020}. Let $\overline{\Phi} \in \NODE$ be a neural ODE with $T \coloneqq \delta L$ defined by the two transformations $\lambda$, $\tilde{\lambda}$ of the ResNet $\Phi$ and based on the initial value problem 
    \begin{equation}\label{eq:Euler4} 
        \frac{\dd h}{\dd t} = f(h(t),\theta(t)) \coloneqq \frac{\eps-1}{\delta}\cdot h(t) + f_\textup{RN}(h(t),\theta(t)).
    \end{equation}
    Since $f_\textup{RN}\in C^{k,0}(\R^{\nhid}\times\R^p,\R^{\nhid})$ by assumption and the linear term $\frac{\eps-1}{\delta}\cdot h$ is smooth in both $h$ and $\theta$, the vector field $f$ defined in~\eqref{eq:Euler4} satisfies $f\in C^{k,0}(\R^{\nhid}\times\R^p,\R^{\nhid})$. Together with the smoothness of the parameter function $\theta$, this confirms $\overline{\Phi}\in\NODE$.
    
    Analogously to part~\ref{prop:resnet_neuralODE_a}, an explicit Euler discretization of the initial value problem~\eqref{eq:Euler4} over the time interval $[0,T]$ with step size $\delta$ yields
    \begin{equation}\label{eq:Euler5}
        h(t + \delta) = h(t) + (\eps-1)\cdot h(t) + \delta \cdot f_\textup{RN}(h(t),\theta(t)) = \eps h(t) + \delta \cdot f_\textup{RN}(h(t),\theta(t)).
    \end{equation}
    Using $t_l \coloneqq l \delta$, $h_l \coloneqq h(t_l)$ for $l \in \{0,\ldots,L\}$, and $\theta_l \coloneqq \theta(t_{l-1})$ for $l \in \{1,\ldots, L\}$, the discretization~\eqref{eq:Euler5} simplifies at the time points $t_l$, $l \in \{0,\ldots,L-1\}$ after the index shift $l \mapsto l-1$ to the ResNet update rule~\eqref{eq:Euler3}, such that assertion~\ref{prop:resnet_neuralODE_b} follows. 
\end{proof}

    \subsection{Proof of Proposition~\ref{prop:resnet_mlp}}
    \label{app:proof_resnet_mlp}

\begin{proof}
    To prove part~\ref{prop:resnet_mlp_a}, given a ResNet $\Phi \in \textup{RN}_{0,\delta}^k(\mathcal{X},\R^{\nout})$ with the notation introduced in Section~\ref{sec:resnet_model}, we construct an equivalent FNN $\overline{\Phi} \in \FNN$ as follows. We denote the layers and layer maps of the FNN with a dash to distinguish them from those of the ResNet.
    \begin{itemize}
        \item As the ResNet $\Phi$ consists of an input transformation $\lambda$, hidden layers $h_0,\ldots,h_L$, and an output transformation $\tilde{\lambda}$, we consider an FNN with $L+2$ layers as introduced in Section~\ref{sec:mlp_model}. The input of the FNN is $\bar{h}_{-1} = x \in \mathcal{X}\subset \R^{\nin}$, the hidden layers are $\bar{h}_0, \ldots, \bar{h}_L$ with $\bar{h}_l \in \R^{\nhid}$, and the output is $\bar{h}_{L+1}\in\R^{\nout}$, such that the FNN is a map $\overline{\Phi}: x = \bar{h}_{-1}\mapsto \bar{h}_{L+1}$.
        \item The FNN layer dimensions are chosen to be $\bar{n}_l = \nhid$ for all $l \in \{0,\ldots,L\}$.
        \item For the hidden layers $\bar{h}_1, \ldots, \bar{h}_L$, the layer maps and parameters are chosen to agree with those of the ResNet, i.e., $\bar{f}_l(\cdot,\bar{\theta}_l) \coloneqq f_l(\cdot,\theta_l)$ and $\bar{\theta}_l \coloneqq \theta_l$ for all $l \in \{1,\ldots,L\}$.
        \item The layer map of the first FNN layer $\bar{h}_0$ is chosen as $\bar{f}_0(x,\bar{\theta}_0) \coloneqq \frac{1}{\delta}\lambda(x)$, such that $\bar{h}_0 = \delta \bar{f}_0(x,\bar{\theta}_0) = \lambda(x)$.
        \item The layer map of the output layer $\bar{h}_{L+1}$ is chosen as $\bar{f}_{L+1}(y,\bar{\theta}_{L+1}) \coloneqq \frac{1}{\delta}\tilde{\lambda}(y)$, such that $\bar{h}_{L+1} = \delta \bar{f}_{L+1}(\bar{h}_L,\bar{\theta}_{L+1}) = \tilde{\lambda}(\bar{h}_L)$.
    \end{itemize}
    Since $\lambda \in C^k(\R^{\nin},\R^{\nhid})$m $\tilde{\lambda} \in C^k(\R^{\nhid},\R^{\nout})$ and $\delta > 0$, the constructed layer maps $\bar{f}_0$ and $\bar{f}_{L+1}$ inherit the required regularity, i.e., $\bar{f}_0(\cdot,\bar{\theta}_0) \in C^k(\R^{\nin},\R^{\nhid})$ and $\bar{f}_{L+1}(\cdot,\bar{\theta}_{L+1}) \in C^k(\R^{\nhid},\R^{\nout})$.
    
    By construction it holds $\bar{h}_{-1} = x$, $\bar{h}_0 = \lambda(x)$, $\bar{h}_l = h_l$ for all $l \in \{1,\ldots,L\}$, and $\bar{h}_{L+1} = \tilde{\lambda}(h_L)$. Hence the input-output maps of the ResNet $\Phi$ and the FNN $\overline{\Phi}$ agree:
    \begin{equation*}
        \Phi(x) = \tilde{\lambda}(h_L(\lambda(x))) = \bar{h}_{L+1}(x) = \overline{\Phi}(x) \qquad \text{for all } x \in \mathcal{X},
    \end{equation*}
    and part~\ref{prop:resnet_mlp_a} follows.
    
    For part~\ref{prop:resnet_mlp_b}, the construction above can be reversed: given an FNN $\overline{\Phi} \in \FNN$ whose hidden layers all have the same dimension $\bar{n}_l = \nhid$ for $l \in \{0,\ldots,L\}$, we define a ResNet $\Phi \in \textup{RN}_{0,\delta}^k(\mathcal{X},\R^{\nout})$ by setting the input transformation $\lambda(x) \coloneqq \delta \bar{f}_0(x,\bar{\theta}_0)$, the residual functions $f_l(\cdot,\theta_l) \coloneqq \bar{f}_l(\cdot,\bar{\theta}_l)$ for $l \in \{1,\ldots,L\}$, and the output transformation $\tilde{\lambda}(y) \coloneqq \delta \bar{f}_{L+1}(y,\bar{\theta}_{L+1})$. By the same argumentation as above, the resulting ResNet $\Phi$ has the same input-output map as $\overline{\Phi}$, and part~\ref{prop:resnet_mlp_b} follows.

    In particular, if the input and output transformations $\lambda$, $\tilde{\lambda}$ have the typical form~\eqref{eq:lambda_typical}, the statements hold for canonical ResNets $\Phi \in \textup{RN}_{0,\delta,\sigma}^k(\mathcal{X},\R^{\nout})$ and MLPs $\overline{\Phi}\in\MLP$ as a special case, since the layer maps of canonical ResNets and MLPs are of the form~\eqref{eq:resnet_typical_f} and hence fulfill the regularity requirements of Definition~\ref{def:feedforward}.
\end{proof}

    \section{Distance between ResNets, Neural ODEs and FNNs}
    \label{app:distance}

    This appendix collects the proofs of Corollary~\ref{cor:node_resnet_error} and Corollary~\ref{cor:mlp_resnet_error}, which determine the distance between canonical ResNets and neural ODEs, and between canonical ResNets and MLPs, respectively.

    \subsection{Proof of Corollary~\ref{cor:node_resnet_error}}
    \label{app:cor_proof_resnet_neuralODE}
    
    \begin{proof}
	To apply Theorem~\ref{th:eulerdiscretization} to the given neural ODE, we calculate the constants $M_\theta$ and $K_\theta$. The Lipschitz constant of the vector field of the neural ODE with respect to the variable~$h$ is given by $K_\theta \coloneqq K_{\sigma} \widetilde{\omega}_\infty  \omega_\infty$, as
	\begin{align*}
		\normm{\widetilde{W} &\sigma(W h_1(t) +b) + \tilde{b} - (\widetilde{W} \sigma(W h_2(t) +b) + \tilde{b})}_\infty
		\leq \; \widetilde{\omega}_\infty  \cdot \normm{\sigma(W h_1(t) +b) -\sigma(W h_2(t) +b) }_\infty \\
		\leq &\; \widetilde{\omega}_\infty   K_\sigma \cdot\norm{Wh_1(t)+b-(Wh_2(t)+b)}_\infty \leq \; \widetilde{\omega}_\infty K_\sigma \omega_\infty \cdot \norm{h_1(t) -h_2(t)}_\infty,
	\end{align*}
	where the upper bound $\norm{\sigma'}_{\infty,\R} \leq K_\sigma$ is, by the mean value theorem, a Lipschitz constant of the activation function. 
	To calculate an upper bound $M_\theta$ for the second derivative of the sup-norm of the solution $h:[0,T] \rightarrow \R^{\nhid}$, we estimate
	\begin{align*}
		\normm{h''(t)}_\infty &= \norm{\frac{\partial}{\partial t} h'(t)}_\infty =  \norm{ \frac{\partial}{\partial t} \left(\widetilde{W} \sigma(W h(t) +b) + \tilde{b}\right)}_\infty \\
		&= \normm{\widetilde{W} \sigma'(Wh(t) + b) \cdot \left(\widetilde{W} \sigma(W h(t) +b) + \tilde{b}\right) }_\infty \\
		&\leq \widetilde{\omega}_\infty  K_{\sigma} \left( \widetilde{\omega}_\infty  S_\sigma + \tilde{\beta}_\infty\right) \eqqcolon M_\theta
	\end{align*}
	for every $t \in [0,T]$. Here, we used that for the component-wise applied activation function it holds $\norm{\sigma}_{\infty,\R} \leq S_\sigma$ and $\norm{\sigma'}_{\infty,\R} \leq K_\sigma$.
\end{proof}

    \subsection{Proof of Corollary~\ref{cor:mlp_resnet_error}}
    \label{app:cor_proof_resnet_mlp}

    \begin{proof}
    To apply Theorem~\ref{th:mlp_resnet_error} to canonical ResNets and MLPs, we calculate the constants $S_f$,  $S_\lambda$, $K_f$ and $K_{\tilde{\lambda}}$. Under the given assumptions, the residual functions $f_l$ of the ResNet are bounded by 
	\begin{equation*} \label{eq:resnet_mlp_proof_1}
		\norm{f_l(h_{l-1},\theta_l)}_\infty  = \norm{\widetilde{W}_l \sigma_l(W_l h_{l-1} + b_l) + \tilde{b}_l}_\infty \leq  \widetilde{\omega}_\infty S_\sigma + \tilde{\beta}_\infty \eqqcolon S_f 
	\end{equation*}
	for all $h_l \in \R^{\nhid}$, $\theta_l \in \Theta_l$ and $l \in \{1,\ldots,L\}$. Furthermore, it holds for the input transformation $\lambda$ \begin{equation*}\label{eq:resnet_mlp_proof_2}
		\norm{\lambda(x)}_\infty  = \norm{\widetilde{W}_0 \sigma_0(W_0 x + b_0) + \tilde{b}_0}_\infty \leq \widetilde{\omega}_\infty S_\sigma + \tilde{\beta}_\infty \eqqcolon S_\lambda  
	\end{equation*}
	for all $x \in \mathcal{X}$ and $\theta_0 \in \Theta_0$. As the output transformation $\tilde{\lambda}$ has the same structure as the residual functions $f_l$, we set $f_{L+1} \coloneqq \tilde{\lambda}$ and estimate 
	\begin{align*}
		\; \norm{f_l(y_1,\theta_l) - f_l(y_2,\theta_l)}_\infty = & \; \norm{\widetilde{W}_{l} \sigma_{l}(W_{l} y_1 +b_{l}) + \tilde{b}_{l} - (\widetilde{W}_{l} \sigma_{l}(W_{l} y_2 +b_{l}) + \tilde{b}_{l})}_\infty \\
		\leq & \; \widetilde{\omega}_\infty \cdot \norm{\sigma_{l}(W_{l} y_1 +b_{l}) -\sigma_{l}(W_{l} y_2 +b_{l}) }_\infty \\
		\leq & \; \widetilde{\omega }_\infty  K_{\sigma} \cdot\norm{W_{l}y_1+b_{l}-(W_{l}y_2+b_{l})}_\infty \\
		\leq & \; \widetilde{\omega }_\infty  K_{\sigma}  \omega_\infty \cdot \norm{y_1-y_2}_\infty
	\end{align*}
    for $y_1,y_2\in \R^{\nhid}$ and $l \in \{1,\ldots,L+1\}$. Hence, all $f_l$, $l \in \{1,\ldots,L+1\}$, are globally Lipschitz continuous with Lipschitz constant $\widetilde{\omega }_\infty  K_{\sigma} \omega_\infty$, such that it follows $K_f = \widetilde{\omega }_\infty  K_{\sigma} \omega_\infty$ and $K_{\tilde{\lambda}} = \widetilde{\omega }_\infty  K_{\sigma} \omega_\infty $. The result follows by inserting the calculated constants  $S_f$,  $S_\lambda$, $K_f$ and $K_{\tilde{\lambda}}$ into~\eqref{eq:error_resnet_mlp}.
\end{proof}

\end{document}